\documentclass[reqno]{amsart}
\usepackage{amsfonts}
\usepackage{graphicx}
\usepackage{amsmath,amssymb,amsfonts,amsthm}
\usepackage[english]{babel}
\usepackage{hyperref}
\usepackage{marginnote}
\usepackage{color}

\newcommand{\dis}{\displaystyle}

\newcommand{\Red}{\textcolor{red}}

\newcommand{\R}{\mathbb{R}}
\newcommand{\T}{\mathbb{T}}
\theoremstyle{plain}

\newtheorem{theorem}{Theorem}[section]
\newtheorem{lemma}[theorem]{Lemma}
\newtheorem{proposition}[theorem]{Proposition}

\newtheorem{definition}[theorem]{Definition}

\theoremstyle{remark}

\newtheorem{remark}[theorem]{Remark}

\numberwithin{equation}{section}

\def\v{\varepsilon}

\def\k{\kappa}

\def\b{\beta}
\def\g{\gamma}
\def\d{\delta}

\def\s{\sigma}

\def\f{\frac}

\usepackage{tikz}

\newcommand{\dd}{{\rm d}}

\renewcommand{\S}{\mathbb{S}}

\newcommand{\Fi}{\mathbf{1}}

\newcommand{\CD}{\mathcal {D}}
\newcommand{\CE}{\mathcal{E}}

\newcommand{\CI}{\mathcal{I}}

\newcommand{\CL}{\mathcal{L}}

\newcommand{\CM}{\mathcal{M}}

\newcommand{\CV}{\mathcal{V}}

\newcommand{\na}{\nabla}

\newcommand{\al}{\alpha}

\newcommand{\ga}{\gamma}
\newcommand{\om}{\omega}
\newcommand{\la}{\lambda}

\newcommand{\pa}{\partial}
\newcommand{\ka}{\kappa}
\newcommand{\eps}{\epsilon}

\newcommand{\Ga}{\Gamma}

\newcommand{\vep}{{\varepsilon}}

\begin{document}

\title[Long-time dynamics of rarefied gas]{Effects of soft interaction
and non-isothermal  boundary upon long-time dynamics of rarefied gas}

\author[R.-J. Duan]{Renjun Duan}
\address[R.-J. Duan]{Department of Mathematics, The Chinese University of Hong Kong, Hong Kong}
\email{rjduan@math.cuhk.edu.hk}

\author[F.M. Huang]{Feimin Huang}
\address[F. M. Huang]{Institute of Applied Mathematics, Academy of Mathematics and Systems Science, Chinese Academy of Sciences, Beijing 100190, China, and University of Chinese Academy of Sciences
}
\email{fhuang@amt.ac.cn}

\author[Y. Wang]{Yong Wang}
\address[Y. Wang]{Institute of Applied Mathematics, Academy of Mathematics and Systems Science, Chinese Academy of Sciences, Beijing 100190, China, and University of Chinese Academy of Sciences
}
\email{yongwang@amss.ac.cn}

\author[Z. Zhang]{Zhu Zhang}
\address[Z. Zhang]{Department of Mathematics, The Chinese University of Hong Kong, Hong Kong}
\email{zzhang@math.cuhk.edu.hk}

\begin{abstract}
In the paper, assuming that the motion of rarefied gases in a bounded domain is governed by the angular cutoff Boltzmann equation with diffuse reflection boundary, we study the effects of both soft intermolecular interaction and non-isothermal wall temperature upon the long-time dynamics of solutions to the corresponding initial boundary value problem. Specifically, we are devoted to proving the existence and dynamical stability of stationary solutions whenever the boundary temperature has  suitably small variations around a  positive constant. For the proof of existence, we introduce a new mild formulation of solutions to the steady boundary-value problem along the speeded backward bicharacteristic, and develop the uniform estimates on approximate solutions in both $L^2$ and $L^\infty$. Such mild formulation proves to be useful for treating the steady problem with soft potentials even over unbounded domains. In showing the dynamical stability, a new point is that we can obtain the sub-exponential time-decay rate in $L^\infty$ without losing any velocity weight, which is actually quite different from the classical results as in \cite{Caf2,SG} for the torus domain and essentially due to the diffuse reflection boundary and the boundedness of the domain.
\end{abstract}

\date{\today}
\maketitle

\setcounter{tocdepth}{1}
\tableofcontents

\thispagestyle{empty}


\section{Introduction}

\subsection{Boltzmann equation}
Let a rarefied gas be contained in a bounded domain $\Omega$ in $\R^3$, and let $F=F(t,x,v)$ denote the density distribution function of gas particles with position $x\in\Omega$ and velocity $v\in\mathbb{R}^3$ at  time $t>0$. We assume that $F$ is governed by the Boltzmann equation
\begin{equation}\label{1.1}
\pa_tF+v\cdot\nabla_x F=Q(F,F).
\end{equation}
The Boltzmann collision term on the right-hand takes the non-symmetric bilinear form of
\begin{equation}\label{1.2}
Q(F_1,F_2)=\int_{\mathbb{R}^3}\int_{\mathbb{S}^2} B(|v-u|,\omega)[F_1(u')F_2(v')-F_1(u)F_2(v)]\,\dd\omega\dd u,
\\
\end{equation}
where the velocity pair $(v',u')$ is defined by the velocity pair $(v,u)$ as well as the parameter $\omega\in \S^2$ in terms of the relation
\begin{equation*}
v'=v-[(v-u)\cdot\omega]\omega,\quad u'=u+[(v-u)\cdot\omega]\omega,
\end{equation*}
according to conservation laws of momentum and energy
\begin{equation*}
v'+u'=v+u,\quad |v'|^2+|u'|^2=|v|^2+|u|^2,
\end{equation*}
due to the elastic collision of two particles. To the end the Boltzmann collision kernel $B(|v-u|,\omega)$, depending only on the relative velocity $|v-u|$ and $\cos \phi=\omega \cdot (v-u)/|v-u|$, is assumed to satisfy
\begin{equation}\label{1.4}
B(|v-u|,\omega)=|v-u|^{\kappa}b(\phi),
\end{equation}
with
\begin{equation}\label{1.4-1}
-3<\kappa<0,\quad 0\leq b(\phi)\leq C|\cos\phi|,
\end{equation}
for a generic constant $C>0$, namely, we consider in this paper  the full range of soft potentials under the Grad's angular cutoff assumption.

\subsection{Diffuse reflection boundary condition}

We assume that $\Omega=\{\xi(x)<0\}$ is connected and bounded with $\xi(x)$ being a smooth function in $\mathbb{R}^3$. At each boundary point with $\xi(x)=0$, we assume that $\nabla\xi(x)\neq 0$. The outward unit normal vector is therefore given by $n(x)=\nabla\xi(x)/|\nabla\xi(x)|$. We define that $\Omega$ is strictly convex if there is $c_{\xi}>0$ such that 
$\sum_{ij}\partial _{ij}\xi (x)\eta ^{i}\eta ^{j}\geq c_{\xi }|\eta |^{2}$
for all $x 
\in\bar{\Omega}$ and all $\eta \in \mathbb{R}^{3}$.

We denote the phase boundary of the phase space $\Omega \times \mathbb{R}^{3}$ as
$\gamma =\partial \Omega \times \mathbb{R}^{3}$, and split $\gamma$ into three disjoint parts, outgoing
boundary $\gamma _{+},$ the incoming boundary $\gamma _{-},$ and the
singular boundary $\gamma _{0}$ for grazing velocities$:$
\begin{align}
\gamma _{+} &=\{(x,v)\in \partial \Omega \times \mathbb{R}^{3}:
n(x)\cdot v>0\}, \nonumber\\
\gamma _{-} &=\{(x,v)\in \partial \Omega \times \mathbb{R}^{3}:
n(x)\cdot v<0\}, \nonumber\\
\gamma _{0} &=\{(x,v)\in \partial \Omega \times \mathbb{R}^{3}:
n(x)\cdot v=0\}.\nonumber
\end{align}
We supplement the Boltzmann equation \eqref{1.1} with the diffuse reflection boundary condition
\begin{align}\label{diffuseB.C}
F(t,x,v)|_{\gamma_-}=\mu_{\theta}(v) \int_{v'\cdot n(x)>0} F(t,x,v') \{v'\cdot n(x)\}\,\dd v',
\end{align}
where  $
\mu_{\theta}(v)$ is a local Maxwellian with a non-isothermal wall temperature $\theta=\theta(x)>0$:
\begin{equation*}
\mu_{\theta}(v)=\f{1}{2\pi \theta^2(x)} {e^{-\frac{|v|^2}{2\theta(x)}}}.
\end{equation*}
Throughout this paper, we assume that $\theta(x)$ has a small variation around a fixed postive temperature $\theta_0>0$.
Without loss of generality, we assume $\theta_0=1$, and for brevity we denote the global Maxwellian
\begin{equation}\label{GM}
\mu=\mu(v)\equiv \mu_{\theta_0}(v)=\f{1}{2\pi}{e^{-\f{|v|^2}{2}}}.
\end{equation}

\subsection{Main results}
Note that
\begin{equation}\label{S3.5}
\int_{v\cdot n(x)>0}\mu_{\theta}(v) \{v\cdot n(x)\}\, \dd v=1,
\end{equation}
for any $x\in \pa\Omega$, and hence $\mu_{\theta}(v)$  satisfies the boundary condition \eqref{diffuseB.C}. However,   it is straightforward to see that the stationary local Maxwellian $\mu_{\theta}(v)$  does not satisfy the Boltzmann equation \eqref{1.1} because of spatial variation unless $\theta(x)$ is constant on $\pa\Omega$. One may expect that the long-time behavior of solutions to  \eqref{1.1} and \eqref{diffuseB.C} could be determined by the time-independent steady equation with the same boundary condition. Thus the study of this paper includes two parts. In the first part we investigate the steady problem in order to obtain the existence of stationary solutions, and in the second part we are devoted to showing the dynamical stability of the obtained stationary solutions under small perturbations and further under a class of large perturbations in velocity weighted $L^\infty$ spaces.

In what follows we present the main results of this paper. The first one is to clarify the well-posedness of the boundary-value problem on the Boltzmann equation with diffuse reflection boundary condition
\begin{equation}
\label{S3.1}
\left\{\begin{aligned}
&v\cdot \nabla_x F=Q(F,F),\quad \   (x,v)\in \Omega\times\mathbb{R}^3,\\
&F(x,v)|_{\gamma_-}=\mu_{\theta}(v) \int_{v'\cdot n(x)>0} F(x,v') \{v'\cdot n(x)\}\,\dd v'.
\end{aligned}\right.
\end{equation}
We define a velocity weight function
\begin{equation}\label{WF}
w=w(v):= (1+|v|^2)^{\f\beta2} e^{\varpi |v|^\zeta},
\end{equation}
where  $\beta>0$ and $0<\zeta\leq 2$ are given constants, and $(\varpi,\zeta)$ belongs to
\begin{equation}
\label{index}
\{\zeta=2,\,0<\varpi <\f18\}\cup\{0<\zeta<2,\, \varpi>0\}.
\end{equation}
Here and in the sequel, for brevity we have omitted the explicit dependence of $w$ on all parameters $\beta$, $\varpi$ and $\zeta$.

\begin{theorem}
\label{thm1.1}
Let $-3<\kappa<0$, $\b>3+|\kappa|$, and $(\varpi,\zeta)$ belong to \eqref{index}. For given $M>0$, there exist $\delta_0>0$ and $C>0$ such that if 
\begin{equation}\label{S1.1}
\delta:=|\theta-\theta_0|_{L^\infty(\partial\Omega)}\leq  \delta_0,
\end{equation}
then there exists a  unique nonnegative solution
$F_*(x,v)=M\mu(v)+\mu^{\frac{1}{2}}(v)f_*(x,v)\geq0$
to the steady problem \eqref{S3.1}, satisfying the mass conservation 
\begin{equation*}
\int_\Omega\int_{\R^3} f_\ast (x,v)\mu^{\frac{1}{2}}(v)\,\dd v\dd x=0,
\end{equation*}
and the estimate
\begin{align}\label{S1.2}
\|wf_*\|_{L^\infty}+|wf_*|_{L^\infty{(\g)}}\leq C\d.
\end{align}
Moreover, if $\Omega$ is strictly convex and $\theta(x)$ is continuous on $\partial\Omega$, then $F_*$ is continuous on  $(x,v)\in\bar{\Omega}\times\mathbb{R}^3\backslash \gamma_0$.
\end{theorem}

For simplicity, through the paper we would take $M=1$ in Theorem \ref{thm1.1} without loss of generality, namely, the stationary solution $F_\ast(x,v)$ has the same total mass as $\mu(v)$ in $\Omega$. Note that when there is no spatial variation on the boundary temperature, i.e., $\delta=0$, the stationary solution is reduced to the global Maxwellian.

The second  result is concerned with the dynamical stability of $F_\ast(x,v)$ under small perturbations in $L^\infty$. We assume that  \eqref{1.1} is also supplemented with initial data
\begin{equation}
\label{ad.id}
F(t,x,v)|_{t=0}=F_0(x,v).
\end{equation}
The goal is to show the large-time convergence of solutions of the initial-boundary value problem \eqref{1.1}, \eqref{diffuseB.C} and \eqref{ad.id} to the stationary solution $F_\ast(x,v)$, whenever they are sufficiently close to each other in some sense at initial time.

\begin{theorem}\label{thm1.2}
Let $-3<\ka<0$, $\b>3+|\k|$ and $(\varpi,\zeta)$ belong to \eqref{index}. 
Assume \eqref{S1.1}
with $\delta_0>0$ chosen to be further small enough. There exist constants 
$\vep_0>0$, $C_0>0$ and $\lambda_0>0$ such that if
$F_0(x,v)={F_*}(x,v)+\mu^{\frac{1}{2}}(v)f_{0}(x,v)\geq 0$ satisfies the mass conservation
\begin{equation}\label{mass}
\int_{\Omega}\int_{\mathbb{R}^{3}}f_0(x,v)\mu^{\frac{1}{2}}(v)\,\dd v\dd x=0,
\end{equation}
and
\begin{equation}\label{D1.1}
\|wf_0\|_{L^{\infty}}\leq \vep_0, 
\end{equation}
then the initial-boundary value problem \eqref{1.1}, \eqref{diffuseB.C} and \eqref{ad.id} on the Boltzmann equation 
admits a unique solution $F(t,x,v)={F_*(x,v)}+\mu^{\frac{1}{2}}(v)f(t,x,v)\geq 0$ satisfying
\begin{equation}
\label{sol.cons}
\int_{\Omega}\int_{\mathbb{R}^{3}}f(t,x,v)\mu^{\frac{1}{2}}(v)\,\dd v\dd x=0,
\end{equation}
and
\begin{align}\label{D1.3}
\|wf(t)\|_{L^{\infty}}+|wf(t)|_{L^{\infty}(\g)}\leq C_0e^{-\lambda_0 t^\alpha}\|wf_0\|_{L^{\infty}},
\end{align}
for all $t\geq 0$, where $\al\in (0,1)$ is given by
\begin{equation}
\label{def.alpha}
\alpha:=\frac{\zeta}{\zeta+|\ka|}.
\end{equation}
Moreover, if $\Omega$ is strictly convex, $F_0(x,v)$ is continuous except on $\g_0$ and satisfying
\begin{align}\label{D1.2a}
F_0(x,v)|_{\gamma_-}=\mu_{\theta}(v) \int_{v'\cdot n(x)>0} F_0(x,v') \{v'\cdot n(x)\} \,\dd v',
\end{align}
and $\theta(x) $ is continuous on $\partial \Omega$, then $F(t,x,v)$ is continuous in $[0,\infty)\times \{\bar{\Omega}\times \mathbb{R}^{3}\setminus\g_0\}$.
\end{theorem}


We remark that the value of $\al$ in \eqref{def.alpha}, which is optimal in terms of the exponential velocity weighted function space, can be formally determined as in \cite{Caf2}; we will come back to this point later. By \eqref{D1.3}, we have obtained the global existence and large-time behavior of solutions simultaneously in the velocity-weighted $L^\infty$ space which is the same as that  initial data belong to.

One may notice from \eqref{D1.1} in Theorem \ref{thm1.2} above that the initial perturbation $f_0(x,v)$
is  required to generally have a small amplitude in the velocity-weighted $L^\infty$ space.  The goal of the third  result is to relax such restriction by allowing $F_0(x,v)$ to have large oscillations around the stationary solution $F_\ast(x,v)$ with the price that the initial perturbation $f_0(x,v)$ is small enough in some $L^p$ norm  for $1<p<\infty$.


\begin{theorem}\label{thm1.3}
Assume that all conditions in Theorem \ref{thm1.2} are satisfied, and additionally let
\begin{equation}
\label{cond.p}
\max\{\f32,\frac{3}{3+\ka}\}<p<\infty,\quad \b>\max\{3+|\ka|, 4\}.
\end{equation}
Assume \eqref{S1.1}
with $\delta_0>0$ chosen to be further small enough and initial data $F_0(x,v)=F_*(x,v)+\mu^{\frac{1}{2}}(v)f_0(x,v)\geq0$ satisfies the mass conservation \eqref{mass}. There exist constants 
$\vep_1>0$, $C_1>{1}$ and $C_2>1$ such that if  $f_0(x,v)$ satisfies
\begin{align}\label{M0}
M_0:=\|wf_0\|_{L^\infty}\leq C_1|\log\delta|,
\end{align}
and
\begin{equation}
\label{cond.lp}
 \|f_0\|_{L^p}\leq \vep_1,
\end{equation}
then the initial-boundary value problem \eqref{1.1}, \eqref{diffuseB.C} and \eqref{ad.id} on the Boltzmann equation admits a unique solution $F(t,x,v)={F_*(x,v)}+\mu^{\frac{1}{2}}(v)f(t,x,v)\geq 0$ satisfying \eqref{sol.cons} and
\begin{equation}\label{D1.3a}
\|wf(t)\|_{L^{\infty}}+|wf(t)|_{L^{\infty}{(\g)}}\leq  C_2 e^{C_2M_0}e^{-\lambda_0 t^{\alpha}}\|wf_0\|_{L^{\infty}},
\end{equation}
for all $t\geq 0$, where $\alpha$ is the same as in \eqref{def.alpha} and $\la_0$ is the same as in \eqref{D1.3}.
Moreover, if $\Omega$ is strictly convex, $F_0(x,v)$ is continuous except on $\g_0$  satisfying \eqref{D1.2a}, and $\theta(x) $ is continuous over $\partial \Omega$, then $F(t,x,v)$ is continuous in $[0,\infty)\times \{\bar{\Omega}\times \mathbb{R}^{3}\setminus\g_0\}$.
\end{theorem}

\begin{remark}
We give a few remarks in order on the above theorem.
\begin{itemize}
  \item[(a)] Note that $C_1$ is independent of $\delta$. Then, from \eqref{M0}, $M_0=\|wf_0\|_{L^\infty}$ can be arbitrarily large, provided that both $\delta$ and $\|f_0\|_{L^p}$ are sufficiently small. Particularly, if one takes $\delta=0$ corresponding to the isothermal boundary temperature, there is no restriction on the upper bound of $M_0$. However, it is unclear how to remove the condition \eqref{M0} whenever $\delta>0$.
  \item[(b)] From \eqref{cond.p}, $p$ has to be large enough as $\ka$ gets close to $-3$. The condition \eqref{cond.lp} for the smallness of $f_0$ in $L^p$ is different from that in \cite{DHWY} and \cite{DW} where $L^1$ norm and $L^2$ norm were used respectively. Note that  \eqref{cond.lp} can be also guaranteed by  the smallness of $L^1$ or $L^2$ norm of $f_0$ and the velocity-weighted $L^\infty$ bound with the help of the interpolation.
  \item[(c)]  As already mentioned for Theorem \ref{thm1.2}, by \eqref{D1.3a} we have obtained the global existence and large-time behavior of solutions simultaneously in the velocity-weighted $L^\infty$ space which is the same as that initial data belong to. Estimate \eqref{D1.3a}  also implies that the solution may grow with an exponential rate of $M_0$ within  a short time.
\end{itemize}

\end{remark}

\subsection{Comments and literature}

The focus of this paper is on the effects of both the soft intermolecular interaction and the non-isothermal wall temperature on the large-time behaviour of solutions to the initial-boundary value problem on the Boltzmann equation. In what follows we review some known results related to our results and also give comments on how such effects occur.

\medskip
\noindent{(a) {\it Effect of soft potentials.}}
First of all, we discuss the effect of soft potentials on the global well-posedness of the Boltzmann equation in perturbation framework. Compared to the hard potentials, the main difficulty is the lack of the spectral gap of the linearized Boltzmann operator $L$, for instance, the multiplication operator $\nu(v)\sim \langle v\rangle^\ka$ has no strictly positive lower bound over large velocities $|v|$ for $\ka<0$.

In the spatially periodic domain $\T^3$, Caflish \cite{Caf1,Caf2} first  constructed the global-in-time solution for $-1<\ka<0$ and also studied the large-time behavior of solutions, where the proof is based on the time-decay property of the linearised equation together with the bootstrap argument on the nonlinear equation. One important observation by Caflish is that the function $\exp\{-\langle v\rangle^\ka t -c|v|^2\}$, obtained as the solution to the spatially homogeneous equation
$
\pa_t f+\langle v\rangle^\ka f=0
$
with initial data $f(0,v)=\exp\{-c|v|^2\}$, decays in time with a rate $\exp\{-\la t^\beta\}$ with $\beta=2/(2+|\ka|)$ by taking the infimum of $\langle v\rangle^\ka t +c|v|^2$ in $v\in \R^3$. We remark that such sub-exponential time-decay is ensured essentially by adding more exponential velocity weight at initial time; see \cite[equations (3.1) and (3.2) of Theorem 3.1 on page 76]{Caf2}.

Independently, Ukai-Asano \cite{UA} developed the semigroup theory in the case of soft potentials $-1<\ka<0$, and also obtained the global solution as well as the large-time behavior of solutions for the problem in the whole space $\R^3$. As pointed out by \cite[Theorem 9.1 and Remark 9.1 on page 96]{UA}, no solutions have been found in the large in time if initial data and solutions belong to the function space with the same velocity weights. We remark that it is the same situation if one adopts the approach of  \cite{UA} to treat the case of $\T^3$, for instance, one can obtain the arbitrarily large algebraic time-decay rate by postulating more polynomial velocity weights on initial data.

By the pure energy method in high-order Sobolev spaces,  Guo \cite{Guo3} constructed the global solutions  over $\T^3$ for the full range of soft potentials $-3<\ka<0$, but the large-time behavior of solutions was left. This problem was later completely solved by Strain-Guo in \cite{SG-06} and \cite{SG} in terms of the same spirit as in \cite{UA} and \cite{Caf2} by putting additional polynomial or exponential velocity weights on initial data. Such approach was also applied by Strain \cite{St} to study the asymptotic stability of the relativistic Boltzmann equation for the soft potentials in $\T^3$.

In the case of $\R^3$, we also mention Duan-Yang-Zhao \cite{DYZ2} and Strain \cite{St-op} to treat the optimal large-time behavior of solutions for $-3<\ka<0$. Particularly, \cite{DYZ2} found  a velocity weight function containing an exponential factor $\exp\{c|v|^2/(1+t)^q\}$. We remark that this kind of weight could be useful for simultaneously dealing with the global existence and   large-time behavior of solutions for the problem in the torus domain or even in the general bounded domain (for instance, \cite{LYa}), since the typical function  $\exp\{-\langle v\rangle^\ga t -c|v|^2/(1+t)^q\}$ induces a time-decay rate $\exp\{-\la t^{\beta'}\}$ with $\beta'=(2-q|\ka|)/(2+|\ka|)$. Therefore, the large-time behavior of solutions is gained by making the velocity weight in the solution space  become lower and lower as time goes on. Indeed, this is also in the same spirit as in  \cite{St-op} on the basis of the velocity-time splitting technique.

By comparison with those results mentioned above, Theorem \ref{thm1.2} or \ref{thm1.3} implies that the large-time behavior of solutions to the initial-boundary value problem under consideration of this paper is established in the situation where solutions and initial data enjoy the same exponential velocity weight. In other words, to obtain the sub-exponential time-decay for soft potentials, it is no need to put any additional velocity weight on initial data. Roughly speaking, the main reason to realize this point is due to not only the boundedness of the domain but also the diffuse reflection boundary condition, which will be explained in more detail later on. We remark that the results are  nontrivial to obtain even if the wall temperature is reduced to a constant implying that the stationary solution $F_\ast(x,v)$ is a global Maxwellian.

\medskip
\noindent{(b) {\it Effect of non-isothermal boundary.}}  The non-isothermal wall temperature provides an inhomogeneous source to force the Boltzmann solution to tend in large time to nontrivial stationary profiles. We review related works in the following two aspects which also involve the case of isothermal boundary. We mainly focus on general bounded domains. There exist also a number of papers in the setting of one-dimensional bounded intervals with different types of boundary conditions, cf.~\cite{Ma,Sone07}. Among them, we point out that Arkeryd together with his collaborators have made great contributions to this direction, for instance, \cite{ACI,AEMN-1,AN-00} and reference therein, where solutions are constructed mainly for large boundary data. The existence and  dynamical stability of the stationary solution  in a slab with diffuse reflection boundary was considered by Yu \cite{Yu1} in terms of a new probabilistic approach. Hydrodynamic limit to the compressible Navier-Stokes equations for the stationary Boltzmann  equation in a slab was studied by Esposito-Lebowitz-Marra \cite{ELM,ELM-non}. For other related works on the effects of  non-isothermal boundary, we also mention \cite{CCLS, KLT14,KLT13}.

\medskip
\noindent{\it $\bullet$ Time-dependent IBVP in general bounded domains.} A first investigation of the IBVP was made by Hamdache \cite{Ha} for a large-data existence theory in the sense of DiPerna-Lions \cite{DiL}. Extensions of such result have been made in \cite{AC,AM,Cer,Mi} in several directions including the case of general diffuse reflection with variable wall temperature. The large-time behavior of weak solutions was studied in \cite{AN, De, DV}. In the perturbation framework, via the idea of \cite{Vidav}, Guo \cite{Guo2} developed a new approach to treat the global existence, uniqueness and continuity of bounded solutions with different types of boundary conditions. Further progress on high-order Sobolev regularity of solutions was recently made in \cite{GKTT-IM}; see also references therein. For other related works on the study of the IBVP on the nonlinear Boltzmann equation, we would mention \cite{BG} for the general Maxwell boundary condition, \cite{GL} for the global existence of solutions with weakly inhomogeneous data in the case of specular reflection,  \cite{KL} for the specular boundary condition in convex domains with $C^3$  smoothness, and \cite{LYa} for a direct extension of \cite{Guo2} from hard potentials to soft potentials.

\medskip

\noindent{\it $\bullet$ Steady problem in general bounded domains.}  There are much less known results on the mathematical analysis of the stationary Boltzmann equation in a general $3$D bounded domain. First of all, it seems still open to establish a large-data DiPerna-Lions existence theory in the steady case; see \cite{AN-02}, however, for an $L^1$ existence theorem with inflow data when the collision operator is truncated for small velocities. In Guiraud \cite{Gui70,Gui72}, existence of stationary solutions was proved in convex bounded domains, but the positivity of obtained solutions remained unclear. Via the approach in \cite{Guo2}, Esposito et al \cite{EGKM} constructed the small-amplitude non-Maxwellian stationary solution for diffuse reflection when the space-dependent wall temperature has a small variation around a positive constant for hard potentials, and further obtained the positivity of stationary solutions as a consequence of the dynamical stability for the time-evolutionary Boltzmann equation. Indeed, \cite{EGKM} motivates us to study the steady Boltzmann equation for soft potentials, and we will explain the new mild formulation of solutions as well as new a priori estimates in more detail later on.

The hydrodynamic limit of the stationary Boltzmann equation on bounded domains in the incompressible setting was recently justified in \cite{EGKM-hy}. Notice that such research topic was also discussed in \cite{AGK} where the authors have particularly shown the non-existence of steady solutions for the Boltzmann equation with smooth divergence-free external forces in bounded domains with specular reflection.

\subsection{Strategy of the proof}
In what follows, we briefly explain the key points in our proof of Theorem \ref{thm1.1}, Theorem \ref{thm1.2} and Theorem \ref{thm1.3} respectively.

\medskip
\noindent{(a)} First, for the proof of Theorem \ref{thm1.1}, the key step is to establish a priori $L^\infty$-estimates on the steady solutions. The major difficulty comes from the degeneracy of collision frequency $\nu(v)\rightarrow 0$ as $|v|\rightarrow \infty$. Our strategy of overcoming this relies on introducing a new mild formulation of the steady Boltzmann equation along a speeded backward bi-characteristics on which the particles with large velocity move much faster than one along the classical characteristics. Precisely, we need to consider the solvability of the linearized steady Boltzmann equation with inhomogeneous source and boundary data
\begin{equation}
\label{ad.int.p0}
\left\{\begin{aligned}
&v\cdot \na_x f + L f=g,\\
&f|_{\ga_-}=P_\ga f +r.
\end{aligned}\right.
\end{equation} 
See Lemma \ref{lemS3.7}, particularly \eqref{S3.107} for the $L^\infty$ bound of $f$ in terms of $g$ and $r$. Basing on Lemma \ref{lemS3.7}, Theorem \ref{thm1.1} follows by showing the convergence of the iterative approximate solution sequence. 

To show Lemma \ref{lemS3.7}, we turn to study in Lemma \ref{lemS3.4} the solvability of the following approximate boundary-value problem:
\begin{equation}
\label{ad.int.p1}
\left\{\begin{aligned}
&\CL_\la f:=\varepsilon f+v\cdot \na_x f + \nu(v) f-\la Kf=g,\\
&f|_{\ga_-}=(1-\frac{1}{n})P_\ga f +r.
\end{aligned}\right.
\end{equation}
Here, compared to the previous works \cite{EGKM,EGKM-hy},  we input an extra parameter $\la\in [0,1]$ in order to carry out a new strategy of the  construction of solutions by making the interplay of $L^2$ and $L^\infty$ estimates. Specifically, we divide the proof by several steps as follows:

\medskip
\noindent{\it Step 1.} To show the well-posedness of $\CL_0^{-1}$ for $\la=0$. The reason why we start from the case of $\la=0$ is that there is no linear collision term $Kf$. In this case, we are able to directly construct the approximate solutions by solving the inflow problem, so the $L^\infty$ bound of approximate solutions is a  consequence of $L^\infty$ bounds of the source term $g$ as well as the corresponding boundary data. The uniform bound of solutions can be obtained in the same way as in the next Step 2.

\medskip
\noindent{\it Step 2.} To obtain the a priori estimates of solutions in both $L^2$ and $L^\infty$ uniform in all parameters $\varepsilon$, $n$ and $\la$. For the $L^2$ estimate, it is based on the fact that 
$$
\nu-\la K=(1-\la )\nu +\la L
$$ 
with $0\leq \la\leq 1$ is still nonnegative. For the velocity-weighted $L^\infty$ estimate on $h=wf$, we formally multiply the equation of \eqref{ad.int.p1} by $(1+|v|^2)^{|\ka|/2}$ so as to get
\begin{equation*}
\hat{v}\cdot \na_x h +(1+|v|^2)^{|\ka|/2}[\varepsilon+\nu(v)] h=\la (1+|v|^2)^{|\ka|/2} K_w h +(1+|v|^2)^{|\ka|/2}w g,
\end{equation*} 
and then write it as the mild form along the backward bi-characteristic $[x-\hat{v}(t-s), v]$, where $\hat{v}=(1+|v|^2)^{|\ka|/2} v$ is the transport velocity speeded up by comparison with the original velocity $v$. The advantage of such new mild  formulation is that the corresponding new collision frequency
$$
\hat{\nu}(v)=(1+|v|^2)^{|\ka|/2}[\varepsilon+\nu(v)] 
$$
has a uniform-in-$\varepsilon$ strictly positive lower bound independent of $v$. This is crucial for obtaining $L^\infty$-estimates for the steady problem. It should be pointed out that by using such new mild formulation, the $L^\infty$-estimates are also valid for the case that the domain is unbounded, so that our method in principle could be used to further study other physically important problems, such as exterior problems and shock wave theory.

\medskip
\noindent{\it Step 3.} To prove the well-posedness of $\CL_\la^{-1}$ for any $\la\in [0,\la_\ast]$ with a constant $\la_\ast>0$ small enough. The main idea of showing the existence of solutions is based on the fixed point argument for the solution operator
$$
\CL_{\la}^{-1} f=\CL_{0}^{-1} (\la Kf +g).
$$
Note that the contraction property is essentially the consequence of the fact that we are restricted to $\la>0$ small enough. Once the existence of solutions is established, we also have the uniform estimates in $L^2$ and $L^\infty$ obtained in Step 2.

\medskip
\noindent{\it Step 4.} To prove the well-posedness of $\CL_{\la_\ast+\la}^{-1}$ for any $\la\in [0,\la_\ast]$ small enough. Formally, we have 
\begin{equation*}
\CL_{\la_\ast+\la}^{-1}=\CL_{\la_\ast}^{-1}(\la Kf+g).
\end{equation*}
Therefore, we may make use of the same arguments as in Step 3 and also obtain the corresponding uniform estimates. In the end, by repeating such procedure we can establish the solvability of $\CL_{\la}^{-1}$ in the case of $\la=1$, and complete the construction of approximate solutions to the original boundary-value problem \eqref{ad.int.p0}.

\medskip
\noindent{(b)} 
Secondly, for the proof of Theorem \ref{thm1.2}, the key step is to study the time-decay structure of linearized IBVP problem around the steady solution provided by Theorem \ref{thm1.1}. In general, it is hard to obtain a satisfactory decay due to the degeneracy of $\nu(v)$ at large velocity. Unfortunately, our new mild formulation above no longer works for the time-dependent problem. Some new thought should be involved in. As mentioned before, so far there are basically two ways to get the decay of the Boltzmann solution for soft potentials. The first one, which was developed by Guo and Strain \cite{SG}, is to first establish the global existence of the solution with an extra sufficiently strong velocity weight and then obtain the decay of the solution without weight by an interpolation technique. Following their idea, Liu and Yang \cite{LYa} extend their work into the IBVP problem. The other one, which is developed by Duan, Yang and Zhao \cite{DYZ2}, is to introduce a velocity weight involving a time dependent factor $\exp\{c |v|^\zeta/(1+t)^q\}$ in order to compensate the degeneracy of collision frequency. One can see that there is an extra restriction in both theories that the initial datum must involve additional velocity weight. One of main contributions in the present work is to remove such a restriction.

More precisely, we are able to obtain simultaneously the global existence and the sub-exponential decay of the solution, {\it without} loss of any weight. The key observation used in our arguments is twofold. The first is to split the large velocity part and small velocity part in the following type estimate:
$$
\int_{\max\{t-t_\mathbf{b},0\}}^te^{-\nu(v)(t-s)}\dd s=\int_{\max\{t-t_\mathbf{b},0\}}^te^{-\nu(v)(t-s)}\{\Fi_{\{|v|\leq d_\Omega\}}+\Fi_{\{|v|>d_\Omega\}}\}\dd s,
$$
so it holds that 
\begin{align}\label{I1}
\int_{\max\{t-t_\mathbf{b},0\}}^te^{-\nu(v)(t-s)}\dd s
&\leq \int_{\max\{t-t_\mathbf{b},0\}}^te^{-\bar{\nu}_0(t-s)}\Fi_{\{|v|\leq d_\Omega\}}+\Fi_{t-1\leq s\leq t}\Fi_{\{|v|>d_\Omega\}}\dd s,
\end{align}
which is essentially based on the elementary fact that the backward exit time
\begin{align}\label{I2}
t_\mathbf{b}(x,v)\leq \f{\text{diam}(\Omega)}{|v|}.
\end{align}
One can see that even for the soft potential, \eqref{I1} still involves an exponential decay structure. The second is to notice that due to the diffusive reflection boundary condition, the boundary terms naturally exponentially decay in velocity. So, we can make use of the Caflish's idea 
to obtain that
$$
e^{-\nu(v)(t-t_1)}e^{-\f{|v|^2}{16}}\leq Ce^{-\lambda_1(t-t_1)^\alpha},
$$ 
where $\la_1>0$ is obtained by taking the infimum of $\nu(v)(t-t_1)+|v|^2/16$ with respect to velocity.  
This reveals the decay structure for boundary terms. However, \eqref{I1} and \eqref{I2} only work when $\Omega$ is bounded and the solution satisfies the diffuse reflection boundary condition. So far we don't know how to deal with the same problem for the specular reflection or even in the torus.

\medskip
\noindent{(c)} 
Thirdly, the strategy of the proof of Theorem \ref{thm1.3} is to use the linear decay theory provided by the second part to find a large time $T_0=T_0(M_0,\vep_0)$, such that
\begin{align}\label{wf0}
\|wf(T_0)\|_{L^\infty}\leq \vep_0.
\end{align}
Then we can extend our solution into $[T_0,\infty)$ by the previous small-amplitude theory. Since the initial data is allowed to have large oscillation around $F_*$, more efforts should be paid for treating the nonlinear term $w(v)\Gamma(f,f)$. The key point is to bound it pointwisely by a product of $L^\infty$-norm and $L^p$-norm with 
$\max\{\f32,\f{3}{3+\ka}\}<p<\infty$ and then apply a nonlinear iteration. To show \eqref{wf0}, we have to require the following estimates holds:
$$
C\left(\sup_{0\leq t\leq T_0}\|wf(t)\|_{L^\infty}\right)^3\cdot|\theta-1|_{L^\infty(\pa\Omega)}\lesssim \vep_0.
$$
Hence the restriction \eqref{M0} on the amplitude of initial data is naturally required.

\subsection{Plan of the paper}
In Section 2 we will provide some basic estimates on linear and nonlinear collision terms. In Section 3, we study the steady problem and give the proof of Theorem \ref{thm1.1} for the existence of the stationary solution. In Section 4, we study the large-time asymptotic stability of the obtained stationary solution under small perturbations and give the proof of Theorem \ref{thm1.2}. In Section 5, we further extend the result to the situation where initial perturbation can have large amplitude with an extra restriction but be small in $L^p$ norm, and give the proof of Theorem \ref{thm1.3}. In Appendix, we give the proof of a technical lemma which has been used before, and also give the proof of the local-in-time existence of solutions for completeness.

\subsection{Notations}
Throughout this paper, $C$ denotes a generic positive constant which may vary from line to line.  $C_a,C_b,\cdots$ denote the generic positive constants depending on $a,~b,\cdots$, respectively, which also may vary from line to line.  $A\lesssim B$ means that there exists a constant $C>0$ such that $A\leq C B$.
$\|\cdot\|_{L^2}$ denotes the standard $L^2(\Omega\times\mathbb{R}^3_v)$-norm and $\|\cdot\|_{L^\infty}$ denotes the $L^\infty(\Omega\times\mathbb{R}^3_v)$-norm. For the functions depending only in velocity $v$, we denote $\|\cdot\|_{L^p_v}
$ as the $L^p(\mathbb{R}^3_v)$-norm and $\langle\cdot,\cdot\rangle$ as the $L^2(\Omega\times \mathbb{R}^3_v)$ inner product or $L^2(\mathbb{R}^3_v)$ inner product.
Moreover, we denote $\|\cdot\|_{\nu}:=\|\sqrt{\nu}\cdot\|_{L^2}$. For the phase boundary integration, we define $d\gamma\equiv |n(x)\cdot v| dS(x)dx$, where $dS(x)$ is the surface measure and define $|f|_{L^p}^p=\int_{\gamma}|f(x,v)|^pd\gamma$ and the corresponding space is denoted as $L^p(\partial\Omega\times\mathbb{R}^3)=L^p(\partial\Omega\times\mathbb{R}^3;d\gamma)$. Furthermore, we denote $|f|_{L^p(\gamma_{\pm})}=|fI_{\gamma_{\pm}}|_{L^p}$ and $|f|_{L^\infty(\gamma_{\pm})}=|fI_{\gamma_{\pm}}|_{L^\infty}$. For simplicity, we denote $|f|_{L^\infty(\g)}=|f|_{L^\infty(\g_+)}+|f|_{L^\infty(\g_-)}$.

\section{Preliminaries}

\subsection{Basic properties of $L$}
First of all, associated with the global Maxwellian $\mu=\mu(v)$ in \eqref{GM}, we introduce the linearized collision operator $L$ around $\mu$ and the nonlinear collision operator $\Ga(\cdot,\cdot)$ respectively as
\begin{align}\label{1.9-1}
Lf
=-\f1{\sqrt{\mu}}\Big\{Q(\mu,\sqrt{\mu}f)+Q(\sqrt{\mu}f,\mu)\Big\},
\end{align}
and
\begin{equation}\label{1.10}
\Gamma(f,f) 
=\f1{\sqrt{\mu}}Q^+(\sqrt{\mu}f,\sqrt{\mu}f)-\f1{\sqrt{\mu}}Q^-(\sqrt{\mu}f,\sqrt{\mu}f)
:=\Gamma^+(f,f)-\Gamma^-(f,f),
\end{equation}
where $Q^{+}$ and $Q^{-}$ correspond to the gain part and loss part in $Q$ in \eqref{1.2} respectively.
As in \cite{Gl}, under the Grad's angular cutoff assumption \eqref{1.4} and \eqref{1.4-1}, $L$ can be decomposed as $L=\nu-K$, where $\nu=\nu(v)$ is the velocity multiplication operator given by
\begin{equation}
\label{1.12-1}
\nu(v)=\int_{\mathbb{R}^3}\int_{\mathbb{S}^2}B(|v-u|,\omega)\mu(u)\,\dd\omega \dd u\sim (1+|v|)^{\kappa},
\end{equation}
and $K=K_1-K_2$ is the integral operator in velocity given by
\begin{align}
(K_1f)(v)&=\int_{\mathbb{R}^3}\int_{\mathbb{S}^2}B(|v-u|,\omega)\mu^{\frac{1}{2}}(v)\mu^{\frac{1}{2}}(u)f(u)\,\dd\omega\dd u,\label{1.11}\\
(K_2f)(v)&=\int_{\mathbb{R}^3}\int_{\mathbb{S}^2}B(|v-u|,\omega)\mu^{\frac{1}{2}}(u)\left[\mu^{\frac{1}{2}}(u')f(v')
+\mu^{\frac{1}{2}}(v')f(u')\right]\,\dd\omega\dd u.
\label{1.12}
\end{align}
It is well-known that $L$ is a self-adjoint nonnegative-definite operator in $L^2_v$ space with the kernel
$$
\text{Ker}\,L=\text{span}\,\{\phi_0,\cdots, \phi_4\},
$$
where $\phi_{i}=\phi_{i}(v)$ $(i=0,1,\cdots,4)$ are the normal orthogonal basis of the null space $\text{Ker}\,L$ given by
\begin{eqnarray*}
\phi_0&=&(2\pi)^{-\f14}\mu^{\frac{1}{2}}(v),\\
\phi_i&=&(2\pi)^{-\f14}v_i\mu^{\frac{1}{2}}(v),\quad i=1,2,3,\\
 \phi_4&=&\f{(2\pi)^{-\f14}}{\sqrt{6}}(|v|^2-3)\mu^{\frac{1}{2}}(v).
\end{eqnarray*}
For each $f=f(v)\in L^2_v$, we denote the macroscopic part ${P}f$ as  the  projection of $f$ onto $\text{Ker}\,L$, i.e., 
\begin{align}\label{1.10-1}
{P}f=\sum_{i=0}^{4}\langle f,\phi_i\rangle
\,\phi_{i},
\end{align}
and further denote 
 $(I-P)f=f-Pf$ to be the microscopic  part of $f$. It is well-known (see \cite{Guo3} for instance) 
that there is a constant $c_0>0$ such that
\begin{align}\label{1.10-2}
\langle Lf,f \rangle\geq c_0\int_{\R^3}\nu(v)|
{(I-P)}f|^2\,\dd v.
\end{align}
Note that $L$ has no spectral gap in case of soft potentials with $-3<\kappa<0$, particularly, the collision frequency $\nu(v)$ tends to zero as $|v|\to \infty$ due to \eqref{1.12-1}.


\subsection{Estimates on collision operators}

Recall  $K=K_1-K_2$ with $K_1$ and $K_2$ given in terms of \eqref{1.11} and \eqref{1.12}. As in \cite{Guo3}, it holds that 
\begin{align}
\label{def.K}
Kf(v)=\int_{\mathbb{R}^3}k(v,u)f(u)\,\dd u,
\end{align}
where the integral kernel $k(v,u)$ is real and symmetric.

\begin{lemma}[\cite{DHWY}]
The following estimate holds true:
\begin{align}\label{2.27}
|k(v,u)|\leq C|v-u|^\kappa e^{-\f{|v|^2}{4}}e^{-\f{|u|^2}{4}}
+\f{C_\kappa}{|v-u|^{\f{3-\kappa}2}}e^{-\f{|v-u|^2}{8}}e^{-\f{\left||v|^2-|u|^2\right|^2}{8|v-u|^2}}.
\end{align}
Moreover, it holds that
\begin{align*}
\int_{\mathbb{R}^3}|k(v,u)|(1+|u|)^{-{\beta}}\,\dd u\leq C_{\kappa}(1+|v|)^{-1-{\beta}},
\end{align*}
for any ${\beta}\geq0$.
\end{lemma}
In order to deal with difficulties in the case of the soft potentials, as in \cite{SG} we introduce  a smooth cutoff function $0\leq\chi_m(s)\leq 1$ for $s\geq 0$ with $0<m\leq 1$ such that 
\begin{equation}
	\chi_m(s)=1\ \mbox{for}~0\leq s\leq m;\quad \chi_m(s)=0\  \mbox{for}~s\geq2m.\notag
\end{equation}
Then we define
\begin{align*}
	(K^mf)(v)&=\int_{\mathbb{R}^3}\int_{\mathbb{S}^2}B(|v-u|,\omega)\chi_m(|v-u|)\mu^{\frac{1}{2}}(u)\left[\mu^{\frac{1}{2}}(u')f(v')+\mu^{\frac{1}{2}}(v')f(u')\right]\,\dd\omega \dd u\nonumber\\
	&\quad-\int_{\mathbb{R}^3}\int_{\mathbb{S}^2}B(|v-u|,\omega)\chi_m(|v-u|)\mu^{\frac{1}{2}}(v)\mu^{\frac{1}{2}}(u)f(u)\,\dd\omega \dd u\nonumber\\
	&:=
	(K_2^mf)(v)-(K^m_1f)(v),
\end{align*}
and
\begin{equation*}
K^c=K-K^m.
\end{equation*}
Similar to \eqref{def.K}, we denote
\begin{align}\label{2.30a}
(K^mf)(v)=\int_{\mathbb{R}^3} k^m(v,u) f(u)\,\dd u,\quad 
(K^cf)(v)=\int_{\mathbb{R}^3} k^c(v,u) f(u)\,\dd u.
\end{align}
In the following lemma we recall some basic estimates on $K^m$ and $K^c$, whose proof can be found in \cite{SG} and further refined in a recent work \cite{DHWY}.

\begin{lemma}
Let  $-3<\kappa<0$. Then,
for any $0<m\leq1$, it holds that
	\begin{equation}\label{2.31}
		|(K^m f)(v)|\leq Cm^{3+\kappa}e^{-\f{|v|^2}{{6}}}\|f\|_{L^\infty_v},
	\end{equation}
where $C>0$ is independent of $m$.
The kernels $k^m(v,u)$ and $k^c(v,u)$ in \eqref{2.30a} satisfy
\begin{align*}
|k^m(v,u)|\leq C_\kappa \Big\{|v-u|^\kappa+|v-u|^{-\f{3-\kappa}2}\Big\}e^{-\frac{|v|^2+|u|^2}{16}},
\end{align*}
and
\begin{align}\label{2.33}
|k^c(v,u)|&\leq\f{C_\kappa m^{a(\kappa-1)}}{|v-u|^{1+\f{(1-a)}{2}(1-\kappa)}}\f{1}{(1+|v|+|u|)^{a(1-\kappa)}}e^{-\f{|v-u|^2}{10}}e^{-\f{\left||v|^2-|u|^2\right|^2}{16|v-u|^2}}\nonumber\\
&\quad+C|v-u|^\kappa [1-\chi_m(|v-u|)] e^{-\f{|v|^2}{4}}e^{-\f{|u|^2}{4}},
\end{align}	
where $0\leq a\leq 1$ is an arbitrary constant, and $C_\kappa$  is a constant depending only on $\kappa$. It is worth to  point out that $C_\kappa$ is independent of $ a$ and $m$.  Moreover, by denoting
\begin{equation*}
{k_w^c}
(v,u)=k^{c}(v,u)\frac{w(v)}{w(u)},
\end{equation*}
it holds that
\begin{equation}\label{2.40}
\int_{\mathbb{R}^3}|{k_w^c}(v,u)|{e^{\f{|v-u|^2}{32}}}\,\dd u \leq C
m^{\kappa-1}(1+|v|)^{\kappa-2},
\end{equation}
and 	
\begin{align}\label{2.40-1}
\int_{\mathbb{R}^3}|{k_w^c}(v,u)|{e^{\f{|v-u|^2}{32}}}\,\dd u \leq C
(1+|v|)^{-1},
\end{align}
where $C>0$ is independent of $m$.
\end{lemma}

Furthermore, we need the following two lemmas for the later use.

\begin{lemma}[\cite{SG}]
Assume \eqref{index}, then for any $\eta>0$, it holds that
\begin{align}\label{2.28-1}
\left|\left\langle e^{\f{\varpi}{2}|\cdot|^{
\zeta}}Kf, f\right\rangle\right|\leq \left\|e^{\f{\varpi}{4}|\cdot|^{
\zeta}}f\right\|_{\nu}\left(\eta\left\| e^{\f{\varpi}{4}|\cdot|^{
\zeta}}f\right\|_{\nu}+C_{\eta}\left\|\Fi_{\{|\cdot|\leq C_\eta\}}f\right\|_{L^{2}}\right).
\end{align}
\end{lemma}

\begin{lemma}[\cite{LYa}]
It holds that
\begin{align}\label{2.43}
\|\nu^{-1}w\Gamma(f_1,f_2)\|_{L^\infty_v}\leq C\|wf_1\|_{L^\infty_v}\cdot\|wf_2\|_{L^\infty_v}.
\end{align}
\end{lemma}

In the end we conclude this subsection with the following $L^p$ $(p>1)$ estimate on the nonlinear collision term, which will be used in Section 5.

\begin{lemma}\label{lm2.6}
Let $1<p<\infty$ and $w_{\beta}:=(1+|v|^2)^{\f{\beta}{2}}$ with $\beta p>3$. Then it holds that
\begin{align}\label{lm2.6-1}
\|\nu^{-1/p'}\Gamma(f,g)\|_{L^p_{v}}\leq C\min\left\{\|w_\beta f\|_{L^\infty_v}\cdot\|\nu^{1/p}g\|_{L^p_v},\,\|w_\beta g\|_{L^\infty_v}\cdot\|\nu^{1/p}f\|_{L^p_v}\right\},
\end{align}
where $p'$ is the conjugate of $p$ satisfying $1/p+1/p'=1.$
\end{lemma}

\begin{proof}
We first consider the loss part. By H\"older inequality, we have
\begin{align}\label{lm2.6-2}
\|\nu^{-1/p'}\Gamma^-(f,g)\|_{L^p_v}^p&\leq \int_{\mathbb{R}^3}\nu(v)^{-(p-1)}|f(v)|^p\dd v\bigg(\int_{\mathbb{R}^3}\int_{\mathbb{S}^2}B(v-u,\omega)\sqrt{\mu(u)}|g(u)|^p\dd \omega\dd u\bigg)\nonumber\\
&\qquad\times \bigg(\int_{\mathbb{R}^3}\int_{\mathbb{S}^2}B(v-u,\omega)\sqrt{\mu(u)}\dd\omega\dd u\bigg)^{p-1}\nonumber\\
&\leq C\int_{\mathbb{R}^3}\int_{\mathbb{R}^3}\int_{\mathbb{S}^2}B(v-u,\omega)\sqrt{\mu(u)}|g(u)|^p|f(v)|^p\dd\omega\dd u\dd v.
\end{align}
Without loss of generality, we assume that
$$\|w_\beta f\|_{L^\infty_v}\cdot\|\nu^{1/p}g\|_{L^p_v}\geq\|w_\beta g\|_{L^\infty_v}\cdot\|\nu^{1/p}f\|_{L^p_v}.$$
Then from \eqref{lm2.6-2}, it holds that
\begin{align}\label{lm2.6-3}
\|\nu^{-1/p'}\Gamma^-(f,g)\|_{L^p_v}^p&\leq C\|w_\beta g\|^p_{L^\infty_v}\cdot\int_{\mathbb{R}^3}|f(v)|^p\dd v\int_{\mathbb{R}^3}\int_{\mathbb{S}^2}B(v-u,\omega)(1+|u|^2)^{-\f{\beta p}{2}}\dd \om\dd u\nonumber\\
&\leq C\|w_\beta g\|_{L^\infty_v}^p\cdot\|\nu^{1/p}f\|^p_{L^p_v}\nonumber\\
&\leq C\min\left\{\|w_\beta f\|_{L^\infty_v}\cdot\|\nu^{1/p}g\|_{L^p_v},\|w_\beta g\|_{L^\infty_v}\cdot\|\nu^{1/p}f\|_{L^p_v}\right\}^p.
\end{align}
Similarly, for the gain term $\Gamma^+(f,g)$, by H\"older inequality, we have
\begin{align}
\|\nu^{-1/p'}\Gamma^+(f,g)\|_{L^p_v}^p&\leq \int_{\mathbb{R}^3}\nu(v)^{-(p-1)}\dd v\bigg(\int_{\mathbb{R}^3}\int_{\mathbb{S}^2}B(v-u,\omega)\sqrt{\mu(u)}|f(v')|^p|g(u')|^p\dd \om\dd u\bigg)\nonumber\\
&\qquad\times \bigg(\int_{\mathbb{R}^3}\int_{\mathbb{S}^2}B(v-u,\omega)\sqrt{\mu(u)}\dd\omega\dd u\bigg)^{p-1}\nonumber\\
&\leq C\int_{\mathbb{R}^3}\int_{\mathbb{R}^3}\int_{\mathbb{S}^2}B(v-u,\omega)\sqrt{\mu(u)}|g(u')|^p|f(v')|^p\dd\omega\dd u\dd v.\nonumber
\end{align}
Making change of variable $(v,u)\rightarrow (v',u')$, it follows that
\begin{align}
\|\nu^{-1/p'}\Gamma^+(f,g)\|_{L^p_v}^p\leq C\int_{\mathbb{R}^3}\int_{\mathbb{R}^3}\int_{\mathbb{S}^2}B(v-u,\omega)|g(u)|^p|f(v)|^p\dd\omega\dd u\dd v.\nonumber
\end{align}
By the same argument as in \eqref{lm2.6-3}, one has
\begin{align}
\|\nu^{-1/p'}\Gamma^+(f,g)\|_{L^p_v}^p\leq C\min\left\{\|w_\beta f\|_{L^\infty_v}\cdot\|\nu^{1/p}g\|_{L^p_v},\|w_\beta g\|_{L^\infty_v}\cdot\|\nu^{1/p}f\|_{L^p_v}\right\}^p.\nonumber
\end{align}
From this and \eqref{lm2.6-3}, we prove \eqref{lm2.6-1}. Therefore, the proof of Lemma \ref{lm2.6} is complete.
\end{proof}

\subsection{Pointwise weighted estimates on nonlinear term $\Gamma(f,f)$}

Recall \eqref{WF}, \eqref{index}, and \eqref{1.10}. We shall give the estimates on the upper bound of $|w(v)\Gamma^{\pm}(f,f)(v)|$ for pointwise $v$ in terms of the product of the weighted $L^\infty$ norm and the $L^p$ norm for a suitable $p>1$ which is finite. These estimates will play an essential role in the proof of Theorem \ref{thm1.3} treating the initial data of large oscillations but with small $L^p$ perturbations.

\begin{lemma}\label{lm2.5}
Let $-3<\ka<0$, $\max\{\f32, \frac{3}{3+\ka}\}<p<\infty$, $(\varpi,\zeta)$ belong \eqref{index},  
and $\beta>4$.
Then, for each $v\in \R^3$, 
it holds that:
\begin{align}
&|w\Gamma^{-}(f,f)(v)|\leq C\nu(v)\|wf\|_{L^\infty_v}\cdot\left(\int_{\mathbb{R}^{3}}|f(u)|^{p}\,\dd u\right)^{\frac{1}{p}}, \label{2.41}\\
&|w\Gamma^{+}(f,f)(v)|\leq C\nu(v)\|wf\|_{L^\infty_v}\cdot\left(\int_{\mathbb{R}^{3}}|wf(u)|^{p}(1+|u|)^{-p(\beta-4)-4}\,\dd u\right)^{\frac{1}{p}}, \label{2.42}
\end{align}
where the generic constant $C>0$ is independent of  $v$.
\end{lemma}

\begin{proof}
First, we consider  \eqref{2.41} regarding the estimate on the loss term. Indeed, it follows from H\"older inequality that
\begin{equation*}
\begin{aligned}
|w\Gamma^{-}(f,f)(v)|&\leq|wf(v)|\int_{\mathbb{R}^{3}}\int_{\mathbb{S}^{2}}B(|v-u|,\omega)\mu^{\frac{1}{2}}(u)|f(u)|\,\dd u\dd\omega\\
&\leq C\|wf\|_{L^\infty_v}\cdot\left(\int_{\mathbb{R}^{3}}|f(u)|^{p}\,\dd u\right)^{\frac{1}{p}}\cdot
\left(\int_{\mathbb{R}^{3}}|v-u|^{\frac{p\ka}{p-1}}\mu^{\frac{p}{2(p-1)}}(u)\,\dd u\right)^{\frac{p}{p-1}}\\
&\leq C \nu(v)\|wf\|_{L^\infty_v}\cdot{\left(\int_{\mathbb{R}^{3}}|f(u)|^{p}\,\dd u\right)^{\frac{1}{p}}},
\end{aligned}
\end{equation*}
where we have used the fact that $\frac{p\ka}{p-1}>-3$ due to $p>\frac{3}{3+\ka}$. Then, \eqref{2.41} is proved.

To prove \eqref{2.42} for the gain term,
we denote
\begin{equation}
\hat{w}(v):=e^{\varpi|v|^{\zeta}},\nonumber
\end{equation}
which is the pure exponential factor of $w(v)$,  and also set $g(v)=\hat{w}(v)|f(v)|$. Since $|u|^{2}+|v|^{2}=|u'|^{2}+|v'|^{2}$, we have  either $|u'|^{2}\geq\frac{1}{2}\left(|u|^{2}+|v|^{2}\right)$ or $|v'|^{2}\geq\frac{1}{2}\left(|u|^{2}+|v|^{2}\right)$. Without loss of generality, we may assume that $|v'|^{2}\geq\frac{1}{2}\left(|u|^{2}+|v|^{2}\right)$, and then we have
\begin{equation}
\label{l2.6ap1}
|w\Gamma^{+}(f,f)(v)|\leq C \hat{w}(v)\|wf\|_{L^\infty_v}\cdot\int_{\mathbb{R}^{3}}\int_{\mathbb{S}^{2}}B(|v-u|,\omega)\mu^{\frac{1}{2}}(u)|g(v')|\hat{w}(v')^{-1}\hat{w}(u')^{-1}\,\dd u\dd\omega.
\end{equation}
Note that 
in virtue of  $0<\zeta\leq 2$, it holds that
$$
\begin{aligned}
\hat{w}(v)&=e^{\varpi|v|^{\zeta}}\leq e^{\varpi (|u|^{2}+|v|^{2})^{\frac{\zeta}{2}}}=e^{\varpi(|u'|^{2}+|v'|^{2})^{\frac{\zeta}{2}}}\leq e^{\varpi (|u'|^{\zeta}+|v'|^{\zeta})}=\hat{w}(v')\hat{w}(u').
\end{aligned}
$$
Using this, \eqref{l2.6ap1} gives that
\begin{equation}\label{2.45}
|w\Gamma^{+}(f,f)(v)|\leq C \|wf\|_{L^\infty_v}\cdot\int_{\mathbb{R}^{3}}\int_{\mathbb{S}^{2}}B(|v-u|,\omega)\mu^{\frac{1}{2}}(u)|g(v')|\,\dd u\dd\omega.
\end{equation}
To further estimate the integral term on the right-hand side of \eqref{2.45}, we denote $z=u-v$, $z_{\parallel}=\{(u-v)\cdot\omega\}\omega$, and $z_{\perp}=z-z_{\parallel}$, then the collision kernel can be estimated as
\begin{equation}\label{2.45-1}
B(|v-u|,\omega)\leq {C|z_{\parallel}|\left(|z_{\parallel}|^{2}+|z_{\perp}|^{2}\right)^{\frac{\ka-1}{2}}.}
\end{equation}
Plugging \eqref{2.45-1} back to \eqref{2.45} and making change of variable $u\to z$, $|w\Gamma^{+}(f,f)(v)|$ can be further bounded by
$$
C\|wf\|_{L^\infty_v}\cdot\int_{\mathbb{R}^{3}}\int_{\mathbb{S}^{2}}|z_{\parallel}|\cdot|z|^{\ka-1}g(v+z_{\parallel})\mu^{\frac{1}{2}}(v+z)\,\dd z\dd\omega.
$$
By writing $\dd z\dd\omega=\frac{2}{|z_{\parallel}|^2} \dd z_{\parallel}\dd z_{\perp}$, we derive that the above term is bounded by
\begin{equation}\label{2.46}
C\|wf\|_{L^\infty_v}\cdot\int_{\mathbb{R}^{3}}|z_{\parallel}|^{-1}g(v+z_{\parallel})\,\dd z_{\parallel}
\int_{\Pi_{\perp}}(|z_{\parallel}|^{2}+|z_{\perp}|^{2})^{\frac{\ka-1}{2}}\mu^{\frac{1}{2}}(v_{\perp}+z_{\perp})\,\dd{z_{\perp}},
\end{equation}
where $\Pi_{\perp}=\left\{z_{\perp}\in \mathbb{R}^{3}: z_{\parallel}\cdot z_{\perp}=0 \right\}$ and $v_{\perp}=\frac{ v \cdot z_{\perp}}{|v|\cdot |z_{\perp}|}z_{\perp}$ is the projection of $v$ to $\Pi_{\perp}$. To estimate 
\eqref{2.46}, we divide it by two cases.

\medskip
\noindent{\it Case 1: $-1<\ka<0$}. In this case, thanks to $-2<\ka-1<-1$, we have
$$
\int_{\Pi_{\perp}}\big(|z_{\parallel}|^{2}+|z_{\perp}|^{2}\big)^{\frac{\ka-1}{2}}\mu^{\frac{1}{2}}(v_{\perp}+z_{\perp})\,\dd{z_{\perp}}\leq\int_{\Pi_{\perp}}|z_{\perp}|^{\ka-1}\mu^{\frac{1}{2}}(v_{\perp}+z_{\perp})\,\dd{z_{\perp}}\leq C,
$$
for a finite constant $C>0$.
Then it holds that for $p>\f32$,
\begin{align}\label{2.46-2}
|w\Gamma^{+}(f,f)(v)|&\leq C\int_{\mathbb{R}^3}|z_{\parallel}|^{-1}g(v+z_{\parallel})\dd z_{\parallel}=C\int_{\mathbb{R}^3}|u-v|^{-1}g(u)\dd u\nonumber\\
&\leq C\bigg(\int_{\mathbb{R}^3}\f{(1+|u|)^{-4}}{|u-v|^{\f{p}{p-1}}}\bigg)^{\f{p-1}{p}}\bigg(\int_{\mathbb{R}^3}|g(u)|^p(1+|u|)^{4p-4}\dd u\bigg)^{\f1p}\nonumber\\
&\leq C(1+|v|)^{-1}\bigg(\int_{\mathbb{R}^3}|g(u)|^p(1+|u|)^{4p-4}\dd u\bigg)^{\f1p}.
\end{align}
{\it Case 2 :} $-3<\ka\leq -1$. To avoid the higher singularity on the right hand side of \eqref{2.46}, we choose $0<\eps=\eps(p,\ka)<\min\{3+\frac{p\ka}{p-1}, \frac{2p}{p-1}\}$ and bound \eqref{2.46} in terms of
\begin{align}\label{2.46-1}
|w\Gamma^{+}(f,f)| &\leq C\|wf\|_{L^\infty_v}\cdot\int_{\mathbb{R}^{3}}|z_{\parallel}|^{\ka-\frac{(p-1)\eps}{p}}g(v+z_{\parallel})\dd z_{\parallel}
\int_{\Pi_{\perp}}|z_{\perp}|^{-2+\frac{(p-1)\eps}{p}}\mu^{\frac{1}{2}}(v_{\perp}+z_{\perp})\dd{z_{\perp}}\nonumber\\
&\leq C\|wf\|_{L^\infty_v}\cdot\int_{\mathbb{R}^{3}}|z_{\parallel}|^{\ka-\frac{(p-1)\eps}{p}}g(v+z_{\parallel})\dd z_{\parallel},
\end{align}
where in the second line we have used  the fact that
$$
\int_{\Pi_{\perp}}|z_{\perp}|^{-2+\frac{(p-1)\eps}{p}}\mu^{\frac{1}{2}}(v_{\perp}+z_{\perp})\dd{z_{\perp}}<\infty.
$$
Making change of variable $ z_{\parallel}+v\rightarrow u $ and using H\"older's inequality, it holds that
\begin{multline}\label{2.47}
\int_{\mathbb{R}^{3}}|z_{\parallel}|^{\ka-\frac{(p-1)\eps}{p}}g(v+z_{\parallel})dz_{\parallel}=
\int_{\mathbb{R}^{3}}|u-v|^{\ka-\frac{(p-1)\eps}{p}}g(u)\dd u\\
\leq\left[\int_{\mathbb{R}^{3}}\frac{(1+|u|)^{-4}}{|u-v|^{\eps+\frac{p|\ka|}{p-1}}}\dd u\right]^{\frac{p-1}{p}}\times
\left[\int_{\mathbb{R}^{3}}|g(u)|^{p}(1+|u|)^{4p-4}\dd u\right]^{\frac{1}{p}}.
\end{multline}
Since $0<\eps+\frac{p|\ka|}{p-1}<3$, the right hand side of \eqref{2.47} can be further bounded by
$$
C (1+|v|)^{\ka-\frac{p-1}{p}\eps}\left[\int_{\mathbb{R}^{3}}|g(u)|^{p}(1+|u|)^{4p-4}\dd u\right]^{\frac{1}{p}}
$$
which combining with \eqref{2.46-1} and \eqref{2.46-2},  
immediately yields \eqref{2.42}. Hence the proof of Lemma {\ref{lm2.5}} is complete.
\end{proof}


\section{Steady problem}

To construct the solution to the steady Boltzmann equation \eqref{S3.1} and \eqref{diffuseB.C}, we first consider the approximate linearized steady problem
\begin{align}\label{S3.3}
\begin{cases}
\v f+ v\cdot \nabla_xf+Lf=g,\\
f(x,v)|_{\gamma_{-}}=P_{\gamma}f+r,
\end{cases}
\end{align}
where $P_{\g}f$ is defined as
\begin{equation}\label{S3.4}
P_{\gamma}f(x,v)=\mu^{\frac{1}{2}}(v)\int_{v'\cdot n(x)>0} f(x,v') \mu^{\frac{1}{2}}(v') \{v'\cdot n(x)\}\dd v'.
\end{equation}
As in \cite{EGKM}, the penalization term $\v f$ is used to guarantee the conservation of mass.
Recall the weight function $w(v)$ defined by \eqref{WF} with \eqref{index}.
We also define
\begin{equation*}
h(x,v):=w(v) f(x,v), 
\end{equation*}
then 
\eqref{S3.3} can be rewritten as
\begin{align}\label{S3.8}
\begin{cases}
\dis \v h+v\cdot \nabla_x h+\nu(v) h=K_{w} h+wg,\\[2mm]
\dis h(x,v)|_{\gamma_-}=\f{1}{\tilde{w}(v)} \int_{v'\cdot n(x)>0} h(x,v') \tilde{w}(v') \dd\sigma'+wr(x,v),
\end{cases}
\end{align}
where
$$
\tilde{w}(v)\equiv \f{1}{w(v)\mu^{\frac{1}{2}}(v)},
\quad
K_wh=wK(\f{h}{w}).
$$

\subsection{A priori $L^\infty$ estimate}
For the approximate steady Boltzmann equation \eqref{S3.8},  the most difficult part is to obtain the $L^\infty$-bound due to the degeneration of frequency $\nu(v)$ as $|v|\rightarrow\infty$. To overcome this difficulty, the main idea is to introduce a new characteristic line.

\begin{definition}

Given $(t,x,v)$, let $[\hat{X}(s),V(s)]$ be the backward bi-characteristics for the {{\it steady}} Boltzmann equation {\eqref{S3.1}}, which is determined by
\begin{align}\label{ode1}
\begin{cases}
\dis \frac{d\hat{X}(s)}{ds}=(1+|V(s)|^2)^{\frac{|
{\ka}|}{2}} V(s):=\hat{V}(s),\\[3mm]
\dis \frac{dV(s)}{ds}=0,\\[3mm]
[\hat{X}(t),V(t)]=[x,v].
\end{cases}
\end{align}
The solution is then given by
\begin{align}\label{S3.10}
[\hat{X}(s;t,x,v),V(s;t,x,v)]=[x-\hat{v}(t-s),v],\quad 
\hat{v}:=(1+|v|^2)^{\f{|\kappa|}{2}}v,
\end{align}
which is called the speeded backward bi-characteristic for the problem \eqref{S3.1}.
\end{definition}

We note that compared to the usual characteristic line as used in the time-evolutionary case,  the particle along \eqref{ode1} or \eqref{S3.10} with given $(x,v)$  travels with the velocity $\hat{v}$ which has the much faster speed than $|v|$ itself for $|v|$ for large velocity. This is the key idea to overcome the difficulty of soft potentials in treating the steady problem on the Boltzmann equation.


In terms of the speeded backward bi-characteristic, we need to redefine the corresponding backward exit time et al.  Indeed, for each $(x,v)$ with $x\in \bar{\Omega}$ and $v\neq 0,$ we define the   {\it backward exit time} $\hat{t}_{\mathbf{b}}(x,v)\geq 0$ to be the last moment at which the
back-time straight line $[\hat{X}({-\tau}
;0,x,v),V({-\tau}
;0,x,v)]$ remains in $\bar{\Omega}$:
\begin{equation*}
\hat{t}_{\mathbf{b}}(x,v)={\sup\, \{s \geq 0:x-\hat{v}\tau\in\bar{\Omega}\text{ for }0\leq \tau\leq s\}.}
\end{equation*}
We therefore have $x-\hat{t}_{\mathbf{b}}\hat{v}\in \partial \Omega $ and $\xi (x-\hat{t}_{\mathbf{b}}\hat{v})=0.$ We also define
\begin{equation*}
\hat{x}_{\mathbf{b}}(x,v) 
=x-\hat{t}_{\mathbf{b}}\hat{v}\in \partial \Omega .
\end{equation*}
Note that the fact that $v\cdot n(\hat{x}_{\mathbf{b}})=v\cdot n(\hat{x}_{\mathbf{b}}(x,v)) \leq 0$ always holds true.

Let $x\in \bar{\Omega}$, $(x,v)\notin \gamma _{0}\cup \g_{-}$ and
$
(t_{0},x_{0},v_{0})=(t,x,v)$. For $v_{k+1}\in \hat{\mathcal{V}}_{k+1}:=\{v_{k+1}\cdot n(\hat{x}_{k+1})>0\}$, the back-time cycle is defined as
\begin{equation}\label{S3.13}
\left\{\begin{aligned}
\hat{X}_{cl}(s;t,x,v)&=\sum_{k}\Fi_{[\hat{t}_{k+1},\hat{t}_{k})}(s)\{\hat{x}_{k}-\hat{v}_k(\hat{t}_{k}-s)\},\\[1.5mm]
V_{cl}(s;t,x,v)&=\sum_{k}\Fi_{[\hat{t}_{k+1},\hat{t}_{k})}(s)v_{k},
\end{aligned}\right.
\end{equation}
with
\begin{equation*}
(\hat{t}_{k+1},\hat{x}_{k+1},v_{k+1})
=(\hat{t}_{k}-\hat{t}_{\mathbf{b}}(\hat{x}_{k},v_{k}), \hat{x}_{\mathbf{b}}(\hat{x}_{k},v_{k}),v_{k+1}).
\end{equation*}
We also define the iterated integral
$$
\int_{\Pi_{j=1}^{k-1}\hat{\CV}_j}\Pi_{j=1}^{k-1}\dd\hat{\sigma}_j
:=\int_{\hat{\CV}_1}\cdots\left\{\int_{\hat{\CV}_{k-1}}\dd\hat{\sigma}_{k-1}\right\}\cdots\dd\hat{\sigma}_1,
$$
where $$\dd\hat{\sigma}_j:=\mu(v_j)\{n(\hat{x}_j)\cdot v_j\},\quad j=1,\cdots, k-1$$ are probability measures.
%

\begin{lemma}\label{lemS3.1}
Let $(\eta,\zeta)$ belong to
$$
\{\zeta=2, 0\leq\eta<1/2\}\cup\{0\leq \zeta<2,\eta\geq 0\}.
$$
For $T_0>0$ sufficiently large, there exist constants $\hat{C}_1$ and $\hat{C}_2$ independent of $T_0$ such that for $k=\hat{C}_1T_0^{\frac54}$ and $(t,x,v)\in[0,T_0]\times\bar{\Omega}\times\mathbb{R}^3$, it holds that
\begin{align}\label{S3.15}
\int_{\Pi _{j=1}^{k-1}\hat{\mathcal{V}}_{j}} \Fi_{\{\hat{t}_k>0\}}~  \Pi _{j=1}^{k-1} {e^{\eta|v_j|^\zeta}} \dd\hat{\sigma} _{j}\leq \left(\frac12\right)^{\hat{C}_2T_0^{\frac54}}.
\end{align}
\end{lemma}

\begin{proof}
We take $\epsilon>0$ small enough, and define the
non-grazing sets
$$
\hat{\mathcal{V}}_{j}^\epsilon=\{v_j\in\hat{\mathcal{V}}_{j}~:~ v_j\cdot n(\hat{x}_j)\geq \epsilon~\mbox{and}~|v_j|\leq \frac1\epsilon\},\quad j\geq 1.
$$
Then a direct calculation shows that
\begin{align}
\int_{\hat{\mathcal{V}}_{j}\backslash \hat{\mathcal{V}}_{j}^{\epsilon}} {e^{\eta|v_j|^\zeta}}\dd\hat{\sigma}_j\leq C\epsilon,\nonumber
\end{align}
where the constant  $C>0$ is independent of $j$. By similar arguments as in \cite[Lemma 2]{Guo2}, one can prove
$$
\hat{t}_j-\hat{t}_{j+1}\geq \frac{|v_j\cdot n(\hat{x}_j)|}{C_\Omega |v_j|^2(1+|v_j|^2)^{\frac{|\kappa|}{2}}}
$$
with a positive constant $C_\Omega$ depending only on the domain.  If $v_j\in\hat{\mathcal{V}}_j$, then we have
$\hat{t}_j-\hat{t}_{j+1}\geq \frac{\epsilon^{3+|\kappa|}}{C_\Omega }$. Therefore, if $\hat{t}_k=\hat{t}_k(t,x,v,v_1, \cdots, v_{k-1})>0$, there can be at most $\left[ \frac{C_\Omega T_0}{\epsilon^{3+|\kappa|}}\right]+1$ number of $v_j\in \hat{\mathcal{V}}_j^\epsilon$ for $1\leq j\leq k-1$. Hence we have
\begin{align}\label{S3.17}
&\int_{\Pi _{j=1}^{k-1}\hat{\mathcal{V}}_{j}} \Fi_{\{\hat{t}_k>0\}}~  \Pi _{j=1}^{k-1} e^{\eta|v_j|^\zeta}\dd\hat{\sigma} _{j}\nonumber\\
&\leq \sum_{n=1}^{\left[ \frac{C_\Omega T_0}{\epsilon^{3+|\kappa|}}\right]+1} \int_{\{\mbox{There are n number} ~v_j\in \hat{\mathcal{V}}_j^{\epsilon} ~\mbox{for some} ~1\leq j\leq k-1\}} \Pi _{j=1}^{k-1} e^{\eta|v_j|^\zeta}\dd\hat{\sigma} _{j}\nonumber\\
&\leq \sum_{n=1}^{\left[ \frac{C_\Omega T_0}{\epsilon^{3+|\kappa|}}\right]+1} \left(
\begin{gathered}
k-1\nonumber\\
n
\end{gathered}\right) \left|\sup_{j}\int_{\hat{\mathcal{V}}_j^\epsilon}{e^{\eta|v_j|^\zeta}} \dd\hat{\sigma}_j\right|^n\cdot \left|\sup_{j}\int_{\hat{\mathcal{V}}_j\backslash\hat{\mathcal{V}}_j^\epsilon}{e^{\eta|v_j|^\zeta}} \dd\hat{\sigma}_j\right|^{k-1-n}\nonumber\\
&\leq \left( \left[ \frac{C_\Omega T_0}{\epsilon^{3+|\kappa|}}\right]+1\right)\cdot (k-1)^{\left[ \frac{C_\Omega T_0}{\epsilon^{3+|\kappa|}}\right]+1} (C\epsilon)^{{k-1-2\left[ \frac{C_\Omega T_0}{\epsilon^{3+|\kappa|}}+1\right]}}.
\end{align}
One can take $k-{1}=N\left(\left[ \frac{C_\Omega T_0}{\epsilon^{3+|\kappa|}}\right]+1\right)$ with
$\left[ \frac{C_\Omega T_0}{\epsilon^{3+|\kappa|}}\right]\geq1$ and $N> 2(3+|\kappa|)$, so that \eqref{S3.17} 
can be bounded as
\begin{align*}
\int_{\Pi _{j=1}^{k-1}\hat{\mathcal{V}}_{j}} ~\Fi_{\{\hat{t}_k>0\}}~  \Pi _{j=1}^{k-1}{e^{\eta|v_j|^\zeta}} \dd\hat{\sigma} _{j}
&\leq \left\{ 2N\left(\left[ \frac{C_\Omega T_0}{\epsilon^{3+|\kappa|}}\right]+1\right)\right\}^{\left[ \frac{C_\Omega T_0}{\epsilon^{3+|\kappa|}}\right]+{1}}  (C\epsilon)^{\f{N}{2}\left({1}+\left[ \frac{C_\Omega T_0}{\epsilon^{3+|\kappa|}}\right]\right)}\nonumber\\
&\leq \left\{ 4N\left[ \frac{C_\Omega T_0}{\epsilon^{3+|\kappa|}}\right] (C\epsilon)^{\frac{N}{2}}\right\}^{\left[ \frac{C_\Omega T_0}{\epsilon^{3+|\kappa|}}\right]+{1}} \nonumber\\
&\leq \left\{ C_{\Omega,N} \cdot T_0\cdot \epsilon^{\frac{N}{2}-3-|\kappa|}\right\}^{\left[ \frac{C_\Omega T_0}{\epsilon^{3+|\kappa|}}\right]+{1}}.
\end{align*}
We choose
$$
\epsilon=\left(\frac{1}{2C_{\Omega,N} \cdot T_0}\right)^{\frac{1}{\frac{N}{2}-3-|\kappa|}}
$$
such that $ C_{\Omega,N} \cdot T_0\cdot \epsilon^{\frac{N}{2}-3-|\kappa|}=1/2$. Note that for
large $T_0$, it holds that $\eps>0$ is small, and
$$
\left[ \frac{C_\Omega T_0}{\epsilon^{3+|\kappa|}}\right]+{1}\cong C_{\Omega,N}T_0^{1+\frac{3+|\kappa|}{\frac{N}{2}-3-|\kappa|}}.
$$
Finally, we take $N=6(3+|\kappa|)$, so that $\left[ \frac{C_\Omega T_0}{\epsilon^{3+|\kappa|}}\right]+{1}=CT_0^{\frac54}$  and
$$
k=6(3+|\kappa|)\left\{\left[ \frac{C_\Omega T_0}{\epsilon^{3+|\kappa|}}\right]+1\right\}+{1}=CT_0^{\frac54}.
$$
Therefore, \eqref{S3.15}  follows. This completes the proof of Lemma \ref{lemS3.1}.
\end{proof}

Along the back-time cycle \eqref{S3.13}, we can represent the solution of \eqref{S3.8} in a mild formulation which enables us to get the $L^\infty$ bound of solutions in the steady case.  Indeed, for later use, we consider the following iterative linear problems involving a parameter $\la\in [0,1]$:
\begin{equation}\label{S3.19}
\begin{cases}
\v h^{i+1}+v\cdot\nabla_x h^{i+1}+\nu(v) h^{i+1}=\lambda K_{
{w}}^m h^i+\lambda K^c_w h^i +wg,\\[3mm]
\dis h^{i+1}(x,v)|_{\gamma_-}=\f{1}{\tilde{w}(v)} \int_{v'\cdot n(x)>0} h^i(x,v') \tilde{w}(v') \dd\sigma'+w(v)r(x,v),
\end{cases}
\end{equation}
for $i=0,1, 2,\cdots$, where $h^0=h^0(x,v)$ is given. For the mild formulation of \eqref{S3.19}, we have the following lemma whose proof is omitted for brevity as it is similar to that in \cite{Guo2}.


\begin{lemma}
Let $0\leq \lambda\leq 1$. Denote $\hat{\nu}(v):=(1+|v|^2)^{\frac{|\kappa|}{2}}[\v+\nu(v)]$. For each $t\in[0,T_0]$ and for each $(x,v)\in \bar{\Omega}\times \mathbb{R}^3\setminus (\gamma_0\cup \gamma_-)$,  
we have
\begin{align}\label{S3.20}
&h^{i+1}(x,v)=\sum_{n=1}^4J_n+\sum_{n=5}^{14}\Fi_{\{t_1>s\}}J_n,
\end{align}
with
\begin{align}
&J_1= \Fi_{\{\hat{t}_1\leq 0\}} e^{-\hat{\nu}(v)t} h^{i+1}(x-\hat{v}t),\nonumber\\
&J_2+J_3+J_4=\int_{\max\{\hat{t}_1,0\}}^t e^{-\hat{\nu}(v)(t-s)} (1+|v|^2)^\f{|\ka|}{2}\Big[\lambda K_w^mh^{i}+\lambda K^c_wh^{i}+wg\Big](x-\hat{v}(t-s),v)\dd s,\nonumber\\
&J_5=e^{-\hat{\nu}(v)(t-\hat{t}_1)}w(v) r(\hat{x}_1,v),\notag\\
&J_6=\frac{e^{-\hat{\nu}(v)(t-\hat{t}_1)}}{\tilde{w}(v)} \int_{\Pi _{j=1}^{k-1}\hat{\mathcal{V}}_{j}}
\sum_{l=1}^{k-2} \Fi_{\{\hat{t}_{l+1}>0\}} w(v_l)r(\hat{x}_{l+1},v_{l})\dd \hat{\Sigma}_{l}(\hat{t}_{l+1}),\nonumber\\
&J_7=\frac{e^{-\hat{\nu}(v)(t-\hat{t}_1)}}{\tilde{w}(v)} \int_{\Pi _{j=1}^{k-1}\hat{\mathcal{V}}_{j}} \sum_{l=1}^{k-1} \Fi_{\{\hat{t}_{l+1}\leq0<\hat{t}_l\}} h^{i+1-l}(\hat{x}_l-\hat{v}_l\hat  {t}_l,v_l) \dd\hat{\Sigma}_{l}(0),\notag\\
&J_8+J_9+J_{10}=\frac{e^{-\hat{\nu}(v)(t-\hat{t}_1)}}{\tilde{w}(v)} \int_{\Pi _{j=1}^{k-1}\hat{\mathcal{V}}_{j}} \sum_{l=1}^{k-1}\int_0^{\hat{t}_l} \Fi_{\{\hat{t}_{l+1}\leq0<\hat{t}_l\}}\nonumber\\
&\qquad\qquad\qquad\qquad\times (1+|v_l|^2)^{\f{|\ka|}{2}}\Big[\lambda K_w^mh^{i-l}+\lambda K^c_wh^{i-l}+wg\Big](\hat{x}_l-\hat{v}_l(\hat{t}_l-s),v_l) \dd\hat{\Sigma}_l(s),\notag\\
&J_{11}+J_{12}+J_{13}=\frac{e^{-\hat{\nu}(v)(t-\hat{t}_1)}}{\tilde{w}(v)} \int_{\Pi _{j=1}^{k-1}\hat{\mathcal{V}}_{j}} \sum_{l=1}^{k-1}\int_{\hat{t}_{l+1}}^{\hat{t}_l} \Fi_{\{\hat{t}_{l+1}\leq0<\hat{t}_l\}}\nonumber\\
&\qquad\qquad\qquad\qquad \times(1+|v_l|^2)^{\f{|\ka|}{2}}\Big[\lambda K_w^mh^{i-l}+\lambda K^c_wh^{i-l}+wg\Big](\hat{x}_l-\hat{v}_l(\hat{t}_l-s),v_l) \,\dd\hat{\Sigma}_l(s),\nonumber
\end{align}
\begin{align}
&J_{14}=\frac{e^{-\hat{\nu}(v)(t-\hat{t}_1)}}{\tilde{w}(v)} \int_{\Pi _{j=1}^{k-1}\hat{\mathcal{V}}_{j}}  I_{\{\hat{t}_{k}>0\}} h^{i+1-k}(\hat{x}_k,v_{k-1}) \dd\hat{\Sigma}_{k-1}(\hat{t}_k),\nonumber
\end{align}
where we have denoted
\begin{align*}
&\dd\hat{\Sigma}_l(s) = \big\{\Pi_{j=l+1}^{k-1}\dd\hat{\sigma}_j\big\}\cdot \big\{\tilde{w}(v_l) e^{-\hat{\nu}(v_l)(\hat{t}_l-s)} \dd\hat{\sigma}_l\big\}\cdot \big\{\Pi_{j=1}^{l-1} e^{-\hat{\nu}(v_j)(\hat{t}_j-\hat{t}_{j+1})} \dd\hat{\sigma}_j\big\}.
\end{align*}
\end{lemma}


\begin{lemma}\label{lemS3.3}
Let  $\beta>3$.  Let $h^i$, $i=0,1,2,\cdots$, be the solutions to \eqref{S3.19}, satisfying 
$$
\|h^i\|_{L^\infty}+|h^i|_{L^\infty(\gamma)}<\infty.
$$ 
Then there exists $T_0>0$ large enough  such that for $i\geq k:=
{\hat{C}_1}T_0^{\frac54}$, it holds 
that
\begin{align}\label{S3.24}
\|h^{i+1}\|_{L^\infty}{+|h^{i+1}|_{L^\infty(\g)}}&\leq \frac18 \sup_{0\leq l\leq k} \{\|h^{i-l}\|_{L^\infty}\}+C\Big\{\|\nu^{-1}wg\|_{L^\infty}+|wr|_{{L^\infty(\g_-)}}\Big\}\nonumber\\
&\quad+C \sup_{0\leq l\leq k}\left\{\left\|\frac{\sqrt{\nu}h^{i-l}}{w}\right\|_{L^2}\right\}.
\end{align}
Moreover, if $h^i\equiv h$ for $i=1,2,\cdots$, i.e., $h$ is a solution, then \eqref{S3.24} is reduced to the following estimate
\begin{align}\label{S3.24-1}
\|h\|_{L^\infty}{+|h|_{L^\infty(\g)}}&\leq C\Big\{\|\nu^{-1}wg\|_{L^\infty}+|wr|_{{L^\infty (\g_-)}}\Big\}+\left\|\frac{\sqrt{\nu}h}{w}\right\|_{L^2}.
\end{align}
Here it is emphasized that the positive constant  $C>0$ does not depend on $\lambda\in[0,1]$ and $\vep>0$.
\end{lemma}

\begin{proof}
By the definition of $\hat{\nu}(v)$, we first note that
\begin{equation}\label{S3.25}
\hat{\nu}(v)\geq (1+|v|^2)^{\frac{|\kappa|}{2}}\nu(v)\geq \hat{\nu}_0>0,
\end{equation}
 where $\hat{\nu}_0$ is a positive constant independent of $\v$ and $v\in\mathbb{R}^3$. For $J_1$, it follows from \eqref{S3.25} that
\begin{equation}\label{S3.26}
|J_1|\leq e^{-\hat{\nu}_0 t} \|h^{i+1}\|_{L^\infty}.
\end{equation}
For $J_2$, it follows from \eqref{2.31} that
\begin{align}\label{S3.43}
|J_2| &\leq Cm^{3+\kappa}\int_{\max\{\hat{t}_1,0\}}^t e^{-\hat{\nu}_0 (t-s)}(1+|v|^2)^{\f{|\ka|}{2}} e^{-\frac{|v|^2}{20}} \|h^{i}\| \dd s\nonumber\\
&\leq Cm^{3+\kappa} e^{-\frac{|v|^2}{32}}  \|h^i\|_{L^\infty}.
\end{align}
For those terms involving the source $g$, we notice that
\begin{align}\label{S3.27}
\frac{1}{\tilde{w}(v)}\leq 
Cw(v)e^{-\f{|v|^2}{4}}\leq Ce^{-\f{|v|^2}{8}},
\end{align}
which immediately yields  that
\begin{equation}\label{S3.28}
\left\{\begin{aligned}
&\int_{\Pi _{j=1}^{k-1}\hat{\mathcal{V}}_{j}} \tilde{w}(v_l)[1+|v_l|^2]^{\frac{|\kappa|}{2}} \  \Pi_{j=1}^{k-1} \dd\hat{\sigma}_j\leq C<\infty,\quad\mbox{for}\quad 1\leq l\leq k-1,\\[1.5mm]
&\int_{\Pi _{j=1}^{k-1}\hat{\mathcal{V}}_{j}} \sum_{l=1}^{k-1} \Fi_{\{\hat{t}_{l+1}\leq0<\hat{t}_l\}} \tilde{w}(v_l)[1+|v_l|^2]^{\frac{|\kappa|}{2}} \ \Pi_{j=1}^{k-1} \dd\hat{\sigma}_j\leq Ck.
\end{aligned}\right.
\end{equation}
Then it follows from \eqref{S3.27} and \eqref{S3.28} that
\begin{align}\label{S3.29}
|J_4|+|J_{10}|+|J_{13}|\leq Ck \|\nu^{-1}wg\|_{L^\infty},
\end{align}
\begin{align}\label{S3.30}
|J_5|+|J_6|\leq Ck  |wr|_{L^\infty{(\g_-)}},
\end{align}
and
\begin{align}\label{S3.31}
|J_7|&\leq C e^{-\frac18|v|^2} e^{-\hat{\nu}_0 (t-t_1)} \int_{\Pi _{j=1}^{k-1}\hat{\mathcal{V}}_{j}} \sum_{l=1}^{k-1} \Fi_{\{\hat{t}_{l+1}\leq0<\hat{t}_l\}} \dd\hat{\Sigma}_{l}(0)\cdot \sup_{1\leq l\leq k-1}\{\|h^{i+1-l}\|_{L^\infty}\}\nonumber\\
&\leq Ck  e^{-\hat{\nu}_0 t} e^{-\frac18|v|^2} \cdot \sup_{1\leq l\leq k-1}\{\|h^{i+1-l}\|_{L^\infty}\}.
\end{align}
For the term $J_{14}$, it follows from \eqref{S3.27} and  Lemma \ref{lemS3.1} that
\begin{align}\label{S3.32}
|J_{14}|\leq C e^{-\frac18|v|^2} \left(\frac12\right)^{\hat{C}_2 T_0^{\frac54}}\cdot {|h^{i+1-k}|_{L^\infty(\ga_-)}},
\end{align}
where we have taken $k=
{\hat{C}_1}T_0^{\frac54}$ and $T_0$ is a large constant to be  chosen later. From the boundary condition given in the second equation of \eqref{S3.19}, it further holds that
\begin{align}
|h^{i+1-k}|_{L^\infty(\ga_-)}\leq C|h^{i-k}|_{L^\infty(\g_+)}+|wr|_{L^\infty(\g_-)}.\nonumber
\end{align}
For $J_8$, using \eqref{2.31}, \eqref{S3.25}, \eqref{S3.27} and \eqref{S3.28}, one obtains that
\begin{align}\label{S3.33}
{|J_8|}&\leq C m^{3+\kappa} e^{-\f18|v|^2}\cdot \sup_{1\leq l\leq k-1}\{\|h^{i-l}\|_{L^\infty}\} \nonumber\\
&\qquad\quad\times\int_{\Pi _{j=1}^{k-1}\hat{\mathcal{V}}_{j}} \sum_{l=1}^{k-1}I_{\{\hat{t}_{l+1}\leq0<\hat{t}_l\}}\int_0^{t_l} e^{-\hat{\nu}_0(t-s)} \dd s  \ \nu(v_l)^{-1} \tilde{w}(v_l)e^{-\frac{|v_l|^2}{8}}\Pi_{j=1}^{k-1} \dd\hat{\sigma}_j\nonumber\\
&\leq C m^{3+\kappa} e^{-\f18|v|^2} \cdot \sup_{1\leq l\leq k-1}\{\|h^{i-l}\|_{L^\infty}\} \int_{\Pi _{j=1}^{k-1}\hat{\mathcal{V}}_{j}} \sum_{l=1}^{k-1}I_{\{\hat{t}_{l+1}\leq0<\hat{t}_l\}} \nu(v_l)^{-1} \tilde{w}(v_l) \Pi_{j=1}^{k-1} \dd\hat{\sigma}_j\nonumber\\
&\leq Ck  m^{3+\kappa} e^{-\f18|v|^2} \cdot \sup_{1\leq l\leq k-1}\{\|h^{i-l}\|_{L^\infty}\}.
\end{align}
Here we remark that the factor $e^{-\f18|v|^2} $ on the right-hand side of \eqref{S3.33} is very crucial for the later use of the Vidav's iteration.
{For $J_9$}, it holds that 
\begin{align}\label{S3.34}
|J_9| &\leq  C e^{-\f18|v|^2} \sum_{l=1}^{k-1}\int_{\Pi _{j=1}^{l-1}\hat{\mathcal{V}}_{j}} \dd\hat{\sigma}_{l-1}\cdots \dd\hat{\sigma}_1 \int_0^{\hat{t}_l} e^{-\hat{\nu}_0(t-s)} \dd s\nonumber\\
&\qquad\times \int_{\mathcal{V}_l}\int_{\mathbb{R}^3} \Fi_{\{\hat{t}_{l+1}\leq0<\hat{t}_l\}} \nu(v_l)^{-1} \tilde{w}(v_l) |k^c_w(v_l,v') h^{i-l}(\hat{x}_l-\hat{v}_{{l}}(\hat{t}_l-s),v')|\dd v' \dd\hat{\sigma}_l \nonumber\\
&=C e^{-\f18|v|^2} \sum_{l=1}^{k-1}\int_{\Pi _{j=1}^{l-1}\hat{\mathcal{V}}_{j}} \dd\hat{\sigma}_{l-1}\cdots \dd\hat{\sigma}_1 \int_0^{\hat{t}_l} e^{-\hat{\nu}_0(t-s)} \dd s\int_{\mathcal{V}_l\cap \{|v_l|\geq N\}}\int_{\mathbb{R}^3} (\cdots)\dd v' \dd\hat{\sigma}_l \nonumber\\
&\quad+C e^{-\f18|v|^2} \sum_{l=1}^{k-1}\int_{\Pi _{j=1}^{l-1}\hat{\mathcal{V}}_{j}} \dd\hat{\sigma}_{l-1}\cdots \dd\hat{\sigma}_1 \int_0^{\hat{t}_l} e^{-\hat{\nu}_0(t-s)} \dd s\int_{\mathcal{V}_l\cap \{|v_l|\leq N\}}\int_{\mathbb{R}^3} (\cdots)\dd v' \dd\hat{\sigma}_l\nonumber\\
&:=\sum_{l=1}^{k-1} \left(J_{91l}+J_{92l}\right).
\end{align}
We shall estimate the right-hand terms of \eqref{S3.34} as follows. By using \eqref{2.40-1}, 
we have
\begin{align}\label{S3.35}
\sum_{l=1}^{k-1} J_{91l}&\leq C e^{-\f18|v|^2} \sum_{l=1}^{k-1}\int_{\Pi _{j=1}^{l-1}\hat{\mathcal{V}}_{j}} \dd\hat{\sigma}_{l-1}\cdots \dd\hat{\sigma}_1 \int_0^{\hat{t}_l} e^{-\hat{\nu}_0(t-s)} \dd s\nonumber\\
&\qquad\qquad\qquad\times\int_{\mathcal{V}_l\cap \{|v_l|\geq N\}} e^{-\f18|v_l|^2} dv_l \cdot \sup_{1\leq l\leq k-1}\{\|h^{i-l}\|_{L^\infty}\}\nonumber\\
&\leq Ck e^{-\f18|v|^2} e^{-\frac1{16}N^2} \cdot \sup_{1\leq l\leq k-1}\{\|h^{i-l}\|_{L^\infty}\}.
\end{align}
And, for each term $J_{92l}$, we also have
\begin{align}\label{S3.36}
J_{92l}&\leq C e^{-\f18|v|^2} \int_{\Pi _{j=1}^{l-1}\hat{\mathcal{V}}_{j}} \dd\hat{\sigma}_{l-1}\cdots \dd\hat{\sigma}_1 \int_{\hat{t}_l-\frac1N}^{\hat{t}_l} e^{-\hat{\nu}_0(t-s)} \dd s\int_{\mathcal{V}_l\cap \{|v_l|\leq N\}}\int_{\mathbb{R}^3} (\cdots)\dd v' \dd\hat{\sigma}_l\nonumber\\
&\quad+C e^{-\f18|v|^2} \int_{\Pi _{j=1}^{l-1}\hat{\mathcal{V}}_{j}} \dd\hat{\sigma}_{l-1}\cdots \dd\hat{\sigma}_1 \int_0^{\hat{t}_l-\frac1N} e^{-\hat{\nu}_0(t-s)} \dd s\int_{\mathcal{V}_l\cap \{|v_l|\leq N\}} e^{-\f18|v_l|^2} \dd v_l\nonumber\\
&\qquad\qquad\times \int_{\Red{\{}|v'|\geq 2N\Red{\}}}  |k_w^c(v_l,v')| e^{\frac{|v_l-v'|^2}{64}} \dd v' e^{-\frac{N^2}{64}}\cdot \sup_{1\leq l\leq k-1}\{\|h^{i-l}\|_{L^\infty}\}\nonumber\\
&\quad+C e^{-\f18|v|^2} \int_{\Pi _{j=1}^{l-1}\hat{\mathcal{V}}_{j}} \dd\hat{\sigma}_{l-1}\cdots \dd\hat{\sigma}_1 \int_0^{\hat{t}_l-\frac1N} e^{-\hat{\nu}_0(t-s)} \dd s\int_{\mathcal{V}_l\cap \{|v_l|\leq N\}}\int_{\Red{\{}|v'|\leq 2N\Red{\}}} \nonumber\\
&\qquad\qquad\times  \Fi_{\{\hat{t}_{l+1}\leq0<\hat{t}_l\}} e^{-\frac18|v_l|^2} |k^c_w(v_l,v') h^{i-l}(x_l-\hat{v}_{{l}}(\hat{t}_l-s),v')|\dd v' \dd v_l \nonumber\\
&\leq C e^{-\f18|v|^2} \int_{\Pi _{j=1}^{l-1}\hat{\mathcal{V}}_{j}} \dd\hat{\sigma}_{l-1}\cdots \dd\hat{\sigma}_1 \int_0^{\hat{t}_l-\frac1N} e^{-\hat{\nu}_0(t-s)} \dd s\int_{\mathcal{V}_l\cap \{|v_l|\leq N\}}\int_{|v'|\leq 2N} \nonumber\\
&\qquad\qquad\times  \Fi_{\{\hat{t}_{l+1}\leq0<\hat{t}_l\}} e^{-\frac18|v_l|^2} |k^c_w(v_l,v') h^{i-l}(x_l-\hat{v}_{{l}}(\hat{t}_l-s),v')|\dd v' \dd v_l\nonumber\\
&\qquad+\frac{C}{N}e^{-\f18|v|^2}\cdot \|h^{i-l}\|_{L^\infty}.
\end{align}
To estimate the first term on the right-hand side of  \eqref{S3.36},  it follows from {\eqref{2.33}} 
that
\begin{align}\label{S3.38}
&\int_{\mathcal{V}_l\cap \{|v_l|\leq N\}}\int_{|v'|\leq 2N} \Fi_{\{\hat{t}_{l+1}\leq0<\hat{t}_l\}} e^{-\frac18|v_l|^2} |k^c_w(v_l,v') h^{i-l}(\hat{x}_l-\hat{v}_{{l}}(\hat{t}_l-s),v')|\dd v' \dd v_l\nonumber\\
&\leq C_N \left\{\int_{\mathcal{V}_l\cap \{|v_l|\leq N\}}\int_{|v'|\leq 2N} e^{-\frac18|v_l|^2} |k^c_w(v_l,v')|^2 \dd v' \dd v_l \right\}^{\frac12}\nonumber\\
&\qquad\times \left\{\int_{\mathcal{V}_l\cap \{|v_l|\leq N\}}\int_{|v'|\leq 2N} \Fi_{\{\hat{t}_{l+1}\leq0<\hat{t}_l\}} \left|\frac{{\sqrt{\nu(v')}}h^{i-l}(\hat{x}_l-\hat{v}_{{l}}(\hat{t}_l-s),v')}{w(v')}\right|^2 \dd v' \dd v_l \right\}^{\frac12}\nonumber\\
&\leq C_N  m^{\kappa-1}\left\{\int_{\mathcal{V}_l\cap \{|v_l|\leq N\}}\int_{|v'|\leq 2N} \Fi_{\{\hat{t}_{l+1}\leq0<\hat{t}_l\}} \left|\frac{\sqrt{\nu(v')}h^{i-l}(\hat{x}_l-\hat{v}_{{l}}(\hat{t}_l-s),v')}{w(v')}\right|^2 \dd v' \dd v_l \right\}^{\frac12}.
\end{align}
Let $y_l=\hat{x}_l-\hat{v}(\hat{t}_l-s)\in \Omega$ for $s\in [0,\hat{t}_l-\frac1N]$.  A direct computation
shows that
\begin{align}\label{S3.39}
\left|\f{\pa y_l}{\pa v_l}\right|=|\hat{t}_l-s|\cdot\left|\frac{\partial \hat{v}}{\partial v}\right|\geq \f{(1+|v|^2)^{\frac{3|\kappa|}{2}-1}}{N^3}\cdot \{1+(1+|\kappa|)|v|^2\}\geq \f{1}{N^3}.
\end{align}
Thus, by making change of variable $\hat{v}_l\rightarrow y_l$ and using \eqref{S3.39},  one obtains that
\begin{align}
&\left\{\int_{\mathcal{V}_l\cap \{|v_l|\leq N\}}\int_{|v'|\leq 2N} \Fi_{\{\hat{t}_{l+1}\leq0<\hat{t}_l\}} \left|\frac{\sqrt{\nu(v')}h^{i-l}(\hat{x}_l-\hat{v}_{{l}}(\hat{t}_l-s),v')}{w(v')}\right|^2 \dd v' \dd v_l \right\}^{\frac12}\nonumber\\
&\leq C_{N}\left\{\int_{\Omega}\int_{|v'|\leq 2N}\left|\frac{\sqrt{\nu(v')}h^{i-l}(y_l,v')}{w(v')}\right|^2 \dd v' \dd y_l \right\}^{\frac12}\leq C_N \left\|\frac{\sqrt{\nu}h^{i-l}}{w}\right\|_{L^2},\nonumber
\end{align}
which together with \eqref{S3.38} and \eqref{S3.36} yield that
\begin{align}\label{S3.40}
J_{92l}\leq \frac{C}{N}e^{-\f18|v|^2}\cdot \sup_{1\leq l\leq k-1}\{\|h^{i-l}\|_{L^\infty}\}+C_N  m^{\kappa-1} e^{-\f18|v|^2}\left\|\frac{\sqrt{\nu}h^{i-l}}{w}\right\|_{L^2}.
\end{align}
Thus it follows from \eqref{S3.40}, \eqref{S3.35} and \eqref{S3.34} that
\begin{align}\label{S3.41}
|J_9| &\leq  \frac{Ck}{N}e^{-\f18|v|^2}\cdot \sup_{1\leq l\leq k-1}\{\|h^{i-l}\|_{L^\infty}\}+C_N k m^{\kappa-1} e^{-\f18|v|^2}\cdot \sup_{1\leq l\leq k-1}\left\{\left\|\frac{\sqrt{\nu}h^{i-l}}{w}\right\|_{L^2}\right\}.
\end{align}
By similar arguments as in \eqref{S3.33}-\eqref{S3.41}, one can obtain
\begin{align}\label{S3.42}
|J_{11}|+|J_{12}|&\leq Cke^{-\f18|v|^2}\left\{m^{3+\kappa}+\frac1N\right\}\cdot \sup_{1\leq l\leq k-1}\{\|h^{i-l}\|_{L^\infty}\}\nonumber\\
&\qquad+C_N km^{\kappa-1} e^{-\f18|v|^2}\cdot \sup_{1\leq l\leq k-1}\left\{\left\|\frac{\sqrt{\nu}h^{i-l}}{w}\right\|_{L^2}\right\}.
\end{align}
Now substituting \eqref{S3.42}, \eqref{S3.41}, \eqref{S3.33}, \eqref{S3.32}, \eqref{S3.31}, \eqref{S3.30}, \eqref{S3.29}, \eqref{S3.43} and \eqref{S3.26} into \eqref{S3.20}, we get, for $t\in[0,T_0]$,  that
\begin{align}\label{S3.44}
|h^{i+1}(x,v)| & \leq \int_{\max\{\hat{t}_1,0\}}^t e^{-\hat{\nu}_0(t-s)} (1+|v|^2)^{\f{|\ka|}{2}} \int_{\mathbb{R}^3}|k^c_w(v,v')h^{i}(x-\hat{v}(t-s),v')|\dd v'\dd s\nonumber\\
&\quad + A_i(t,v),
\end{align}
where we have denoted
\begin{align*}
A_i(t,v)&:=Ck e^{-\f1{32}|v|^2}\left\{m^{3+\kappa}+e^{-\hat{\nu}_0 t}+ \left(\frac12\right)^{\hat{C}_2T_0^{\frac54}}+\frac{1}N\right\}\cdot \sup_{0\leq l\leq k-1}\{\|h^{i-l}\|_{L^\infty}{+|h^{i-l}|_{L^\infty(\g_+)}}\}\nonumber\\
&\quad+e^{-\hat{\nu}_0 t} \|h^{i+1}\|_{L^\infty}+Ck \Big\{\|\nu^{-1}wg\|_{L^\infty}+|wr|_{{L^\infty(\g_-)}}\Big\}\nonumber\\
&\quad+C_{N,k,m}\  e^{-\f18|v|^2}\cdot \sup_{1\leq l\leq k-1}\left\{\left\|\frac{\sqrt{\nu}h^{i-l}}{w}\right\|_{L^2}\right\}.
\end{align*}
We denote $x'=x-\hat{v}(t-s)\in \Omega$ and $\hat{t}_1'=\hat{t}_1(s,x',v')$  for $s\in(\min\{t_1,0\},t)$.  Using the Vidav's iteration in \eqref{S3.44}, then we obtain that
\begin{align}\label{S3.46}
|h^{i+1}(x,v)| &\leq A_i(t,v)+\int_{0}^t e^{-\hat{\nu}_0(t-s)} (1+|v|^2)^{\f{|\ka|}{2}} \int_{\mathbb{R}^3}|k^c_w(v,v')| A_{i-1}(s,v')\dd v'\dd s \nonumber\\
&\quad+\int_{0}^tds \int_0^s e^{-\hat{\nu}_0(t-\tau)} \dd\tau \int_{\mathbb{R}^3}\int_{\mathbb{R}^3}|k^c_w(v,v') k^c_w(v',v'')|\nonumber\\
&\qquad\times \Fi_{\{\max\{\hat{t}_1,0\}<s<t\}}  \Fi_{\{\max\{\hat{t}'_1,0\}<\tau<s\}}  |h^{i-1}(x'-\hat{v}'(s-\tau),v'')|\dd v'' \dd v'\nonumber\\
&:=A_i(t,v)+B_1+B_2.
\end{align}
For the term $B_1$, using {\eqref{2.40}} and {\eqref{2.40-1}}, one has
\begin{align}\label{S3.47}
B_1&\leq Ck\left\{m^{3+\kappa}+m^{\kappa-1}e^{-\f12\hat{\nu}_0t}+
\left(\frac12\right)^{{\hat{C}_2T_0^{\frac54}}}
+\frac1N\right\} \cdot \sup_{0\leq l\leq k}\{\|h^{i-l}\|_{L^\infty}+{|h^{i-l}|_{L^\infty(\g_+)}}\}\nonumber\\
&\qquad +C_{k,m}\Big\{\|\nu^{-1}wg\|_{L^\infty}+|wr|_{{L^\infty(\g_-)}}\Big\}+C_{N,k,m}\  \sup_{0\leq l\leq k}\left\{\left\|\frac{\sqrt{\nu}h^{i-l}}{w}\right\|_{L^2}\right\}.
\end{align}
For the term $B_2$, we split the estimate by several cases.

\noindent{\it Case 1.} For $|v|\geq N$, we have from {\eqref{2.40}} that
\begin{align}\label{S3.48}
B_2\leq Cm^{2(\kappa-1)} \|h^{i-1}\|_{L^\infty} (1+|v|)^{-4}\leq \frac{Cm^{2(\kappa-1)} }{N^4} \|h^{i-1}\|_{L^\infty}.
\end{align}

\noindent{\it Case 2.}  For $|v|\leq N, |v'|\geq 2N$ or $|v'|\leq 2N, |v''|\geq 3N$. In this case, we note from {\eqref{2.40}}  that
\begin{align*}
\begin{cases}
\dis \int_{|v|\leq N, |v'|\geq 2N} \Big|k_w^c(v,v') e^{\frac{|v-v'|^2}{32}}\Big|\dd v'
\leq C m^{\kappa-1} (1+|v|)^{\kappa-2},\\[3mm]
\dis \int_{|v'|\leq 2N, |v''|\geq 3N} \Big|k_w^c(v',v'') e^{\frac{|v'-v''|^2}{32}}\Big|\dd v''
\leq C m^{\kappa-1} (1+|v'|)^{\kappa-2}.
\end{cases}
\end{align*}
This yields  that
\begin{align}\label{S3.50}
&\int_{0}^t\dd s \int_0^s e^{-\hat{\nu}_0(t-\tau)} \dd\tau \left\{\int_{|v|\leq N, |v'|\geq 2N} +\int_{|v'|\leq 2N, |v''|\geq 3N} \right\}(\cdots)\dd v'' \dd v'\nonumber\\
&\leq e^{-\frac{N^2}{32}}\|h^{i-1}\|_{L^\infty}\int_{|v|\leq N, |v'|\geq 2N} |k^c_w(v,v')e^{\frac{|v-v'|^2}{32}}| \cdot|k^c_w(v',v'')| \nu(v)^{-1}\nu(v')^{-1}\dd v'' \dd v'\nonumber\\
&\quad + e^{-\frac{N^2}{32}}\|h^{i-1}\|_{L^\infty} \int_{|v|\leq N, |v'|\geq 2N} |k^c_w(v,v')|\cdot |k^c_w(v',v'')e^{\frac{|v'-v''|^2}{32}}| \nu(v)^{-1}\nu(v')^{-1}\dd v'' \dd v'\nonumber\\
&\leq C m^{2(\kappa-1)}  e^{-\frac{N^2}{32}}\|h^{i-1}\|_{L^\infty}.
\end{align}

\noindent{\it Case 3.}  For $|v|\leq N$, $|v'|\leq 2N$, and $|v''|\leq 3N$, we first note that
\begin{align}\label{S3.51}
&\int_{0}^t\dd s \int_0^s e^{-\hat{\nu}_0(t-\tau)} \dd\tau
\int_{|v'|\leq 2N, |v''|\leq 3N}(\cdots)\dd v'' \dd v'\nonumber\\
&\leq  \frac{C}{N} m^{2(\kappa-1)}  \|h^{i-1}\|_{L^\infty}+\int_{0}^t\dd s \int_0^{s-\frac1N} e^{-\hat{\nu}_0(t-\tau)} \dd\tau \int_{\mathbb{R}^3}\int_{\mathbb{R}^3}(\cdots)\dd v'' \dd v'\nonumber\\
&\leq \frac{C}{N} m^{2(\kappa-1)}  \|h^{i-1}\|_{L^\infty}\nonumber\\
&\quad+C_{N,k}\int_{0}^t\dd s \int_0^{s-\frac1N} e^{-\hat{\nu}_0(t-\tau)} \dd\tau
\left\{\int_{|v'|\leq 2N, |v''|\leq 3N} |k^c(v,v') k^c(v',v'')|^2\dd v'' \dd v'\right\}^{\f12}\nonumber\\
&\quad \times\left\{\int_{|v'|\leq 2N, |v''|\leq 3N} \Fi_{\{\max\{\hat{t}_1,0\}<s<t\}}  \Fi_{\{\max\{\hat{t}'_1,0\}<\tau<s\}} \Big|\frac{\sqrt{\nu(v'')}h^{i-1}(y',v'')}{w(v'')}\Big|^2\dd v'' \dd v'\right\}^{\f12}\nonumber\\
&\leq \frac{C}N m^{2(\kappa-1)}  \|h^{i-1}\|_{L^\infty}+C_{N,k,m}  \int_{0}^t\dd s \int_0^{s-\frac1N} e^{-\hat{\nu}_0(t-\tau)} \dd\tau \nonumber\\
&\quad\times \left\{\int_{|v'|\leq 2N, |v''|\leq 3N} \Fi_{\{\max\{\hat{t}_1,0\}<s<t\}}  \Fi_{\{\max\{\hat{t}'_1,0\}<\tau<s\}} \Big|\frac{\sqrt{\nu(v'')}h^{i-1}(y',v'')}{w(v'')}\Big|^2\dd v'' \dd v'\right\}^{\f12},
\end{align}
where we have denoted $y'=x'-\hat{v}'(s-\tau)\in \Omega$ for $s\in (\max\{\hat{t}_1,0\}, s)$ and $\tau\in (\max\{\hat{t}'_1,0\}, s)$. Similar to \eqref{S3.39}, we make  change of variable $v'\mapsto y'$, so that the second term on the right-hand side of \eqref{S3.51} is bounded as
\begin{align}\nonumber
&\int_{0}^t\dd s \int_0^{s-\frac1N} e^{-\hat{\nu}_0(t-\tau)} \dd\tau  \left\{\int_{|v'|\leq 2N, |v''|\leq 3N} (\cdots)\dd v'' \dd v'\right\}^{\f12}\leq  CN^{\frac32}\left\|\frac{\sqrt{\nu}h^{i-1}}{w}\right\|_{L^2},
\end{align}
which together with \eqref{S3.51} yield that
\begin{multline}\label{S3.52}
\int_{0}^t\dd s \int_0^s e^{-\hat{\nu}_0(t-\tau)} \dd\tau
\int_{|v'|\leq 2N, |v''|\leq 3N}(\cdots)\dd v'' \dd v'\\
\leq \frac{C}{N}m^{2(\kappa-1)}  \|h^{i-1}\|_{L^\infty}+C_{N,k,m} \left\|\frac{\sqrt{\nu}h^{i-1}}{w}\right\|_{L^2}.
\end{multline}
Combining \eqref{S3.48}, \eqref{S3.50} and \eqref{S3.52}, we have
\begin{align}
B_2 \leq \frac{C}{N}m^{2(\kappa-1)}  \|h^{i-1}\|_{L^\infty}+C_{N,k,m} \left\|\frac{\sqrt{\nu}h^{i-1}}{w}\right\|_{L^2}.\nonumber
\end{align}
Hence, the above estimate together with \eqref{S3.47} and \eqref{S3.46} yields that for any $t\in[0,T_0]$,
\begin{align*}
|h^{i+1}(x,v)|&\leq Ck\left\{m^{3+\kappa}+m^{\kappa-1}e^{-\f12\hat{\nu}_0t}+\left(\frac12\right)^{{\hat{C}_2T_0^{\frac54}}}
+\frac{m^{2(\kappa-1)}}N\right\} \notag\\
&\qquad\qquad\times \sup_{0\leq l\leq k}\{\|h^{i-l}\|_{L^\infty}{+|h^{i-l}|_{L^\infty(\g_+)}}\}\nonumber\\
&\quad +e^{-\hat{\nu}_0 t} \|h^{i+1}\|_{L^\infty}+C_{k,m}\Big\{\|\nu^{-1}wg\|_{L^\infty}+|wr|_{{L^\infty(\g_-)}}\Big\}\nonumber\\
&\quad+C_{N,k,m}\  \sup_{0\leq l\leq k}\left\{\left\|\frac{\sqrt{\nu}h^{i-l}}{w}\right\|_{L^2}\right\}.
\end{align*}
Now we take $k=\hat{C}_1t^{\frac54}=\hat{C}_1T_0^{\frac54}$ and choose $m=T_0^{-\frac{9}{4(3+\kappa)}}$. We first fix $t=T_0$ large enough, and then choose $N$ large enough, so that  one has $e^{-\hat{\nu}_0t}\leq \frac12$ and
\begin{align}\nonumber
Ck\left\{m^{3+\kappa}+m^{\kappa-1}e^{-\f12\hat{\nu}_0t}+\left(\frac12\right)^{{\hat{C}_2T_0^{\frac54}}}
+\frac{m^{2(\kappa-1)}}N\right\}\leq \frac1{16}.
\end{align}
Therefore \eqref{S3.24} follows. This completes  the proof of Lemma \ref{lemS3.3}.
\end{proof}


\subsection{Approximate sequence}
Now we are in a position to  construct solutions to \eqref{S3.3} or equivalently \eqref{S3.5}. First of all, we consider the following approximate problem
\begin{align}\label{S3.54}
\begin{cases}
\dis \v f^n+ v\cdot \nabla_xf^n+\nu(v) f^n-Kf^n=g,\\[2mm]
\dis f^n(x,v)|_{\gamma_{-}}=(1-\frac{1}{n})P_{\gamma}f^n+r,
\end{cases}
\end{align}
where $\v\in(0,1]$ is arbitrary and $n>1$ is an integer. Recall $k=\hat{C}_1T_0^{\frac54}$ with $T_0$ large enough. To the end, we choose $n_0>1$ large enough such that
$$
\frac18 (1-\frac2n+\frac{3}{2n^2})^{-\frac{k+1}{2}}\leq \frac12
$$
for any $n\geq n_0$.

\begin{lemma}\label{lemS3.4}
Let $\v>0$, $n\geq n_0$, and $\beta>3$. Assume $\|\nu^{-1}wg\|_{L^\infty}+|wr|_{L^\infty{(\g_-)}}<\infty$.
Then there exists a unique solution $f^n$ to \eqref{S3.54} satisfying
\begin{align*}
\|wf^{n}\|_{L^\infty}+|wf^n|_{L^\infty(\gamma)}\leq C_{\v,n}\Big( |wr|_{L^\infty({\gamma_-})}+\|\nu^{-1}wg\|_{L^\infty} \Big),
\end{align*}
where the positive constant $C_{\v,n}>0$ depends only on $\v$ and $n$. Moreover, if  $\Omega$ is a strictly convex domain,  
g is continuous in $\Omega\times\mathbb{R}^3$ and $r$ is continuous in $\g_-$, then $f^n$ is continuous away from grazing set $\gamma_0$.
\end{lemma}

\begin{proof}  We consider the solvability of the  following boundary value problem
\begin{equation}\label{S3.56-1}
\begin{cases}
\mathcal{L}_\lambda f:=\v f+v\cdot \nabla_x f+\nu(v) f-\lambda Kf=g,\\
f(x,v)|_{\gamma_-}=(1-\frac1n) P_{\gamma} f+r(x,v),
\end{cases}
\end{equation}
for $\lambda\in[0,1]$. For brevity we denote $\mathcal{L}_\lambda^{-1}$ to be the solution operator associated with the problem, meaning that $f:=\mathcal{L}_\lambda^{-1} g$ is a solution to the BVP \eqref{S3.56-1}. Our idea is to prove the existence of $\mathcal{L}_0^{-1}$, and then extend to obtain the existence of  $\mathcal{L}_1^{-1}$ in a continuous argument on $\la$.  Since the proof is very long, we split it  into  several steps.   

\medskip
\noindent{\it Step 1.} In this step, we prove the existence of $\mathcal{L}_0^{-1}$. We consider the following approximate sequence
\begin{align}\label{S3.57}
\begin{cases}
\mathcal{L}_0 f^{i+1}=\v f^{i+1}+v\cdot \nabla_x f^{i+1} + \nu(v) f^{i+1}=g,\\
f^{i+1}(x,v)|_{\gamma_-}=(1-\frac1n) P_{\gamma} f^{i}+r,
\end{cases}
\end{align}
for $i=0,1,2,\cdots$, where we have set $f^0\equiv0$. We will construct $L^\infty$ solutions to \eqref{S3.57} for $i=0,1,2,\cdots$, and 
establish uniform $L^\infty$-estimates.

Firstly, we will solve inductively the linear equation \eqref{S3.57} by 
the method of characteristics. Let $h^{i+1}(x,v)=w(v)f^{i+1}(x,v)$.  For 
almost every $(x,v)\in\bar{\Omega}\times\mathbb{R}^3\backslash (\gamma_0\cup \gamma_-)$, 
one can write
\begin{align}\label{S3.58}
h^{i+1}(x,v)
&=e^{-(\v+\nu(v))t_{\mathbf{b}}} w(v)\left[(1-\frac1n) P_\gamma f^{i}+r\right](x_{\mathbf{b}}(x,v),v) \nonumber\\
&\quad+\int_{t_1}^t  e^{-(\v+\nu(v))(t-s)} (wg)(x-v(t-s),v) \dd s,
\end{align}
and for $(x,v)\in \g_-$, we write
\begin{align}\label{S3.58-1}
h^{i+1}(x,v)=w(v)\left[(1-\f1n)P_\g f^i+r\right](x,v).
\end{align}
Noting the definition of $P_{\gamma} f$, we have
\begin{equation}\label{S3.59}
|wP_{\gamma} f|_{L^\infty}\leq C|f|_{L^\infty(\gamma_+)}.
\end{equation}
We consider \eqref{S3.58} with $i=0$. Noting $h^0\equiv0$,   then 
it is straightforward to see that
\begin{align*}
\|h^{1}\|_{L^\infty}\leq |wr|_{L^\infty(\g_-)}+ \frac{C}{\v} \|wg\|_{L^\infty}<\infty.
\end{align*}
Therefore we have obtained the solution to \eqref{S3.57} with $i=0$.  Assume that we have already solved \eqref{S3.57} for $i\leq l$ and obtained
\begin{equation}\label{S3.62}
\|h^{l+1}\|_{L^\infty}+|h^{l+1}|_{L^\infty(\gamma)}\leq C_{\v,n,l+1} \Big( |wr|_{L^\infty{(\g_-)}}+ \|wg\|_{L^\infty}\Big)<\infty.
\end{equation}
We now consider \eqref{S3.57} for $i=l+1$. Noting \eqref{S3.62}, then we can solve \eqref{S3.57} by using \eqref{S3.58} and \eqref{S3.58-1} with $i=l+1$. We still need to prove $h^{l+2}\in L^\infty$. Indeed, it follows from \eqref{S3.58}, \eqref{S3.58-1} and \eqref{S3.59} that
\begin{align*}
\|h^{l+2}\|_{L^\infty}+{|h^{l+2}|_{L^\infty(\g)}}&\leq C(|wr|_{L^\infty{(\g_-)}}+ |h^{l+1}|_{L^\infty{(\g_+)}} )+\frac{C}{\v} \|wg\|_{L^\infty}\nonumber\\
&\leq C_{\v,n,l+2}  \Big( |wr|_{L^\infty{{(}\g_-}{)}}+ \|wg\|_{L^\infty}\Big)<\infty.
\end{align*}
Therefore, inductively we have solved  \eqref{S3.57} for $i=0,1,2,\cdots$ and obtained
\begin{align}\label{S3.65}
\|h^{i}\|_{L^\infty}+|h^{i}|_{L^\infty(\gamma)}\leq C_{\v,n,i}\Big( |wr|_{L^\infty{(\g_-)}}+ \|wg\|_{L^\infty}\Big)<\infty,
\end{align}
for $i=0,1,2,\cdots$.
The positive  constant $C_{\v,n,i}$ may increase to infinity as $i\rightarrow \infty$.
Here, we emphasize that we first need to know the sequence $\{h^i\}_{i=0}^{\infty}$ is in $L^\infty$-space, otherwise one can not use Lemma \ref{lemS3.3} to get uniform $L^\infty$ estimates.

If  $\Omega$ is a convex domain, let $(x,v)\in \Omega\times\mathbb{R}^3\backslash \gamma_0$, then it holds $v\cdot n(x_{\mathbf{b}}(x,v))<0$ which yields that $t_{\mathbf{b}}(x,v)$ and $x_{\mathbf{b}}(x,v)$ are smooth by Lemma 2 in \cite{Guo2}. Therefore if $g$ and $r$ are continuous, we have that $f^i(x,v)$ is continuous away from grazing set.

\vspace{1mm}

Secondly, in order to take the limit $i\rightarrow\infty$, one has to get some uniform  estimates.  Multiplying \eqref{S3.57} by $f^{i+1}$ and integrating the resultant equality over $\Omega\times\mathbb{R}^3$, one obtains that
\begin{align}\label{S3.66}
& \v\|f^{i+1}\|^2_{L^2}+\frac12|f^{i+1}|^2_{L^2(\gamma_+)}+\|f^{i+1}\|^2_{\nu}\nonumber\\
&\leq \frac12(1-\frac2n+\frac{3}{2n^2})|f^{i}|^2_{L^2(\gamma_+)}+C_n |r|^2_{L^2(\gamma_-)}+\frac{C}{\v}\|g\|^2_{L^2}+ \frac{\v}{4}\|f^{i+1}\|^2_{L^2},
\end{align}
where we have used $|P_\gamma f^i|_{L^2(\gamma_-)}=|P_\gamma f^i|_{L^2(\gamma_+)}\leq |f^i|_{L^2(\gamma_+)}$. Then, from \eqref{S3.66},  we have
\begin{align*}
&\frac{3}{2}\v\|f^{i+1}\|^2_{L^2}+|f^{i+1}|^2_{L^2(\gamma_+)}+2\|f^{i+1}\|^2_{\nu}\nonumber\\
&\leq (1-\frac2n+\frac{3}{2n^2})|f^{i}|^2_{L^2(\gamma_+)}+C_{\v,n}\Big\{ |r|^2_{L^2(\gamma_-)}+\|g\|^2_{L^2}\Big\}.
\end{align*}
Now we take the difference $f^{i+1}-f^i$ in \eqref{S3.57}, then by similar energy estimate as above, we obtain
\begin{align*}
&\frac{3}{2}\v\|f^{i+1}-f^i\|^2_{L^2}+|f^{i+1}-f^i|^2_{L^2(\gamma_+)}+2\|f^{i+1}-f^i\|^2_{\nu}\nonumber\\
&\leq (1-\frac2n+\frac{3}{2n^2})|f^{i}-f^{i-1}|^2_{L^2(\gamma_+)}\notag\\
&\leq \cdots\notag\\
&\leq (1-\frac2n+\frac{3}{2n^2})^{i} |f^1|^2_{L^2(\gamma)}\nonumber\\
&\leq C_\v\cdot  (1-\frac2n+\frac{3}{2n^2})^{i} \cdot \Big( |wr|_{L^\infty{(\g_-)}}+ \|wg\|_{L^\infty}\Big)<\infty.
\end{align*}
Noting $1-\frac2n+\frac{3}{2n^2}<1$, thus $\{f^{i}\}_{i=0}^\infty$ is a Cauchy sequence in $L^2$, i.e.,
\begin{align*}
\|f^{i}-f^j\|^2_{L^2}+|f^{i}-f^j|^2_{L^2(\gamma_+)}+\|
{f^i-f^j}\|^2_{\nu}\rightarrow0,\quad\mbox{as} \  i,j\rightarrow\infty.
\end{align*}
And we also have, for $i=0,1,2,\cdots$, that
\begin{align}\label{S3.70}
\|f^{i}\|^2_{L^2}+|f^{i}|^2_{L^2(\gamma_+)}+\|f^{i}\|^2_{\nu}\leq C_{\v,n}\Big\{ |r|^2_{L^2(\gamma_-)}+\|g\|^2_{L^2}\Big\},
\end{align}
where $C_{\v,n}>0$ is a positive constant which depends only on $\v$ and $n$.

\vspace{1mm}

Next we consider the uniform  $L^\infty$ estimate. Here we point out that Lemma \ref{lemS3.3} still holds by replacing 1 with $1-\frac1n$ in the boundary condition, and the constants in Lemma \ref{lemS3.3} do not depend on $n\geq1$.  Thus we apply Lemma \ref{lemS3.3} to obtain that
\begin{align*}
\|h^{i+1}\|_{L^\infty}&\leq \frac18 \sup_{0\leq l\leq k} \{\|h^{i-l}\|_{L^\infty}\}
+C\Big\{ |wr|_{L^\infty{(\gamma_-)}}+\|\nu^{-1}wg\|_{L^\infty}\Big\}+C\sup_{0\leq l\leq k} \Big\{ \|f^{i-l}\|_{\nu} \Big\}\nonumber\\
&\leq \frac18 \sup_{0\leq l\leq k} \{\|h^{i-l}\|_{L^\infty}\}
+C_{\v,n}\Big\{ |wr|_{L^\infty{(\gamma_-)}}+\|\nu^{-1}wg\|_{L^\infty}\Big\},
\end{align*}
where we have used \eqref{S3.70} in the second inequality.  Now we apply Lemma \ref{lemA.1} to obtain that for $i\geq k+1$,
\begin{align}\label{S3.72}
\|h^{i}\|_{L^\infty}&\leq \frac18 \max_{0\leq l\leq k} \Big\{\|h^{1}\|_{L^\infty}, \ \|h^{2}\|_{L^\infty}, \cdots, \|h^{2k}\|_{L^\infty}\Big\}\nonumber\\
&\qquad+\frac{8+k}{7} C_{\v,n}\Big\{ |wr|_{L^\infty{(\gamma_-)}}+\|\nu^{-1}wg\|_{L^\infty}\Big\}\nonumber\\
&\leq C_{\v,n,k}\Big\{ |wr|_{L^\infty{(\gamma_-)}}+\|\nu^{-1}wg\|_{L^\infty}\Big\},
\end{align}
where we have used \eqref{S3.65} in the second inequality.
Hence it follows from \eqref{S3.72} and \eqref{S3.65} that
\begin{align}\label{S3.73}
\|h^{i}\|_{L^\infty}
&\leq C_{\v,n,k}\Big\{ |wr|_{L^\infty{(\gamma_-)}}+\|\nu^{-1}wg\|_{L^\infty}\Big\}, \quad\mbox{for} \  i\geq 1.
\end{align}
Taking the difference $h^{i+1}-h^i$ and then applying Lemma \ref{lemS3.3} to $h^{i+1}-h^i$, we have that for $i\geq k$,
\begin{align}\label{S3.74}
&\|h^{i+2}-h^{i+1}\|_{L^\infty}\nonumber\\
&\leq \frac18 \max_{0\leq l\leq k} \Big\{\|h^{i+1-l}-h^{i-l}\|_{L^{\infty}}\Big\}+C \sup_{0\leq l\leq k} \Big\{\|f^{i+1-l}-f^{i-l}\|_{\nu}\Big\}\nonumber\\
&\leq \frac18 \max_{0\leq l\leq k} \Big\{\|h^{i+1-l}-h^{i-l}\|_{L^{\infty}}\Big\}+C_{\v} \cdot \Big\{ |wr|_{L^\infty{(\gamma_-)}}+\|\nu^{-1}wg\|_{L^\infty}\Big\} \cdot  \eta_n^{i-k}\nonumber\\
&\leq \frac18 \max_{0\leq l\leq k} \Big\{\|h^{i+1-l}-h^{i-l}\|_{L^{\infty}}\Big\}+C_{\v,k} \Big\{ |wr|_{L^\infty{(\gamma_-)}}+\|\nu^{-1}wg\|_{L^\infty}\Big\} \cdot  \eta_n^{i+k+1},
\end{align}
where we have denoted  $\eta_n:=
(1-\frac2n+\frac{3}{2n^2})^{1/2}
<1$. Here we choose $n$ large enough so that $\frac18 \eta_n^{-k-1}\leq \frac12$, then it follows from \eqref{S3.74} and Lemma \ref{lemA.1}  that
\begin{align}\label{S3.75}
\|h^{i+2}-h^{i+1}\|_{L^\infty}&\leq \left(\frac18\right)^{\left[\frac{i}{k+1}\right]} \max_{0\leq l\leq 2k}\Big\{\|h^1\|_{L^\infty}, \|h^2\|_{L^\infty},  \cdots , \|h^{2k+1}\|_{L^\infty}\Big\}\nonumber\\
&\quad+C_{\v, k}  \Big\{ |wr|_{L^\infty{(\gamma_-)}}+\|\nu^{-1}wg\|_{L^\infty}\Big\} \cdot  \eta_n^{i}\nonumber\\
&\leq C_{\v,n,k} \cdot \Big\{|wr|_{L^\infty{(\gamma_-)}}+\|\nu^{-1}wg\|_{L^\infty}\Big\}\cdot\left\{\left(\frac18\right)^{\left[\frac{i}{k+1}\right]}  +\eta_n^{i}\right\},
\end{align}
for $i\geq k+1$. Then \eqref{S3.75} implies immediately that $\{h^i\}_{i=0}^\infty$ is a Cauchy sequence in $L^\infty$, i.e., there exists a limit function $
{h}\in L^\infty$ so that $\|h^{i}-
{h}\|_{L^\infty}\rightarrow 0$ as $i\rightarrow \infty$. Thus we obtained a function 
{$f:=\f{h}{w}$} solves
\begin{equation*}
\begin{cases}
\mathcal{L}_0 f=\v f+ v\cdot \nabla_xf+\nu(v) f=g,\\[2mm]
\dis f(x,v)|_{\gamma_{-}}=(1-\frac{1}{n})P_{\gamma}f+r,
\end{cases}
\end{equation*}
with $n\geq n_0$ large enough.
Moreover, from \eqref{S3.73}, 
there exists a constant $C_{\v,n,k}$ such that
\begin{align*}
\|h\|_{L^\infty}+|h|_{L^\infty(\g)}\leq C_{\v,n,k}\Big\{ |wr|_{L^\infty(\gamma_-)}+\|\nu^{-1}wg\|_{L^\infty}\Big\}.
\end{align*}

\medskip
\noindent{\it Step 2.} {\it A priori estimates.} For any given $\lambda\in[0,1]$, let $f^n$ be the solution of \eqref{S3.56-1}, i.e.,
\begin{align}\label{S3.79}
\begin{cases}
\mathcal{L}_\lambda f^n=\v f^{n}+v\cdot\nabla_x f^{n}+\nu(v) f^{n}-\lambda Kf^n=g,\\[2mm]
f^{n}(x,v)|_{\gamma_-}=(1-\frac1n) P_{\gamma}f^{n}+r(x,v).
\end{cases}
\end{align}
Moreover we also assume that $\|wf^{n}\|_{L^\infty}+|wf^n|_{L^\infty(\gamma)}<\infty$. Firstly,  we shall consider   {\it a priori}  $L^2$-estimates. Multiplying \eqref{S3.79} by $f^{n}$, one has that
\begin{align}\label{S3.80}
& \v\|f^{n}\|^2_{L^2}+\frac12|f^{n}|^2_{L^2(\gamma_+)}-\frac12|(1-\frac1n) P_{\gamma}f^{n}+r|^2_{L^2(\gamma_-)}+\|f^{n}\|^2_{\nu}\nonumber\\
&\leq \lambda\langle Kf^n, f^{n}\rangle +\frac{\v}{4} \|f^n\|_{L^2}+\frac{C}{\v}\|g\|^2_{L^2}.
\end{align}
We note that $\langle Lf^n, f^{n}\rangle\geq0 $, which implies that
\begin{equation}\label{S3.81}
\lambda\langle Kf^n, f^{n}\rangle \leq \lambda\|f^n\|_{\nu}.
\end{equation}
On the other hand, a direct computation shows that
\begin{equation}\label{S3.82}
|(1-\frac1n) P_{\gamma}f^{n}+r|^2_{L^2(\gamma_-)}\leq (1-\frac2n+\frac{3}{2n^2}) |f^n|_{L^2(\gamma_+)}+C_n |r|^2_{L^2(\gamma_-)}.
\end{equation}
Substituting \eqref{S3.81} and \eqref{S3.82} into \eqref{S3.80}, one has that
\begin{align}\label{S3.83}
\|\mathcal{L}^{-1}_\lambda g\|^2_{L^2}+|\mathcal{L}^{-1}_\lambda g|^2_{L^2(\gamma_+)}=\|f^{n}\|^2_{L^2}+|f^{n}|^2_{L^2(\gamma_+)}\leq  C_{\v,n} \Big[|r|^2_{L^2(\gamma_-)}+\|g\|^2_{L^2}\Big].
\end{align}
Let $h^n:=w f^n$.  Then, by using \eqref{S3.24-1} and \eqref{S3.83}, we  obtain
\begin{align}\label{S3.86}
\|w\mathcal{L}^{-1}_\lambda g\|_{L^\infty}=\|h^n\|_{L^\infty}\leq C_{\v,n} \Big\{ |wr|_{L^\infty(\gamma_-)}+\|\nu^{-1}wg\|_{L^\infty} \Big\}.
\end{align}

\vspace{1mm}

On the other hand, Let $\nu^{-1}wg_1 \in L^\infty$ and $\nu^{-1}wg_2 \in L^\infty$.  Let  $f^n_1=\mathcal{L}^{-1}_\lambda g_1$ and $f^n_2=\mathcal{L}^{-1}_\lambda g_2$ be the solutions to \eqref{S3.79} with $g$ replaced by $g_1$ and $g_2$, respectively. Then we have that
\begin{align*}
\begin{cases}
\v (f^{n}_2-f^n_1)+v\cdot\nabla_x (f^{n}_2-f^n_1)+\nu(v) (f^{n}_2-f^n_1)-\lambda K(f^{n}_2-f^n_1)=g_2-g_1,\\[2mm]
(f^{n}_2-f^n_1)(x,v)|_{\gamma_-}=(1-\frac1n) P_{\gamma}(f^{n}_2-f^n_1).
\end{cases}
\end{align*}
By similar arguments as in \eqref{S3.79}-\eqref{S3.86}, we obtain
\begin{equation}\label{S3.88}
\|\mathcal{L}^{-1}_\lambda g_2-\mathcal{L}^{-1}_\lambda g_1\|^2_{L^2}+|\mathcal{L}^{-1}_\lambda g_2-\mathcal{L}^{-1}_\lambda g_1|^2_{L^2(\gamma_+)}\leq  C_{\v,n} \|g_2-g_1\|^2_{L^2},
\end{equation}
and
\begin{equation}\label{S3.89}
\|w(\mathcal{L}^{-1}_\lambda g_2-\mathcal{L}^{-1}_\lambda g_1)\|_{L^\infty}\leq C_{\v,n}  \|\nu^{-1}w(g_2-g_1)\|_{L^\infty}.
\end{equation}
The uniqueness of solution to \eqref{S3.79} also follows from \eqref{S3.88}.
We point out that the constant $C_{\v,n}$ in \eqref{S3.83}, \eqref{S3.86}, \eqref{S3.88} and \eqref{S3.89} does not depend on $\lambda\in[0,1]$. This property is crucial for us to extend $\mathcal{L}_0^{-1}$ to $\mathcal{L}_1^{-1}$ by a bootstrap argument.

\medskip
\noindent{\it Step 3.}  In this step,  we shall prove the existence of solution $f^n$ to \eqref{S3.56-1} for sufficiently small $0<\lambda\ll1$, i.e., to prove the existence of operator $\mathcal{L}_{\lambda}^{-1}$. Firstly, we define the Banach space
\begin{equation}\nonumber
\mathbf{X}:=\Big\{f=f(x,v) :  \   wf\in L^\infty(\Omega\times\mathbb{R}^3), \ wf\in L^\infty(\gamma), \  \mbox{and} \  f(x,v)|_{\gamma_-}=(1-\frac1n) P_\gamma f+r \Big\}.
\end{equation}
Now we define
\begin{align}\nonumber
T_\lambda f=\mathcal{L}_0^{-1} \Big(\lambda K f+g\Big).
\end{align}
For any $f_1, f_2\in \mathbf{X}$, by using \eqref{S3.89}, we have  that
\begin{align}
\|w(T_\lambda f_1-T_\lambda f_2)\|_{L^\infty}&= \left\|w\{\mathcal{L}_0^{-1}(\lambda Kf_1+g)-\mathcal{L}_0^{-1}(\lambda Kf_2+g)\} \right\|_{L^\infty}\nonumber\\
&\leq C_{\v,n}\|\nu^{-1} w\{(\lambda Kf_1+g)-(\lambda Kf_2+g)\}\|_{L^\infty}\nonumber\\
&\leq \lambda \|\nu^{-1} w (Kf_1-Kf_2)\|_{L^\infty} \nonumber\\
&\leq \lambda C_{K,\v,n} \|w(f_1-f_2)\|_{L^\infty},\nonumber
\end{align}
where we have used \eqref{2.31} and \eqref{2.40} with $m=1$ in the last inequality.
We take $\lambda_\ast>0$ sufficiently small such that $\lambda_\ast C_{K,\v,n}\leq 1/2$, then $T_\lambda : \mathbf{X}\rightarrow \mathbf{X}$ is a contraction mapping for $\lambda\in[0,\lambda_\ast]$. Thus $T_\lambda$ has a fixed point, i.e.,  $\exists\, f^\lambda\in \mathbf{X}$ such that
\begin{equation}\nonumber
f^\lambda=T_\lambda f^\lambda=\mathcal{L}_0^{-1} \Big(\lambda K f^\lambda+g\Big),
\end{equation}
which yields immediately that
\begin{align}\nonumber
\mathcal{L}_\lambda f^\lambda=\v f^\lambda + v\cdot \nabla_x f^\lambda+\nu f^\lambda-\lambda Kf^{\lambda}=g.
\end{align}
Hence, for any $\lambda\in[0,\lambda_\ast]$,  we have solved \eqref{S3.56-1} with $f^\lambda=\mathcal{L}_\lambda^{-1}g\in\mathbf{X}$.  Therefore we have obtained the existence of $\mathcal{L}_\lambda^{-1}$ for $\lambda\in[0,\lambda_\ast]$. Moreover the operator $\mathcal{L}_\lambda^{-1}$  has the properties \eqref{S3.83}, \eqref{S3.86}, \eqref{S3.88} and \eqref{S3.89}.

Next we define
\begin{align}\nonumber
T_{\lambda_\ast+\lambda}f=\mathcal{L}_{\lambda_\ast}^{-1}\Big(\lambda K f+g\Big).
\end{align}
Noting the estimates for $\mathcal{L}_{\lambda_\ast}^{-1}$ are independent of $\lambda_\ast$. By similar arguments, we can prove $T_{\lambda_\ast+\lambda} : \mathbf{X}\rightarrow \mathbf{X}$ is a contraction mapping  for $\lambda\in[0,\lambda_\ast]$.  Then we obtain the exitence of operator $\mathcal{L}_{\lambda_\ast+\lambda}^{-1}$, and  \eqref{S3.83}, \eqref{S3.86}, \eqref{S3.88} and \eqref{S3.89}.   Step by step, we can finally obtain the existence of operator $\mathcal{L}_1^{-1}$, and $\mathcal{L}_1^{-1}$ satisfies the estimates in \eqref{S3.83}, \eqref{S3.86}, \eqref{S3.88} and \eqref{S3.89}. The continuity is easy to obtain since the convergence of sequence under consideration is always in $L^\infty$.  Therefore we complete the proof of Lemma \ref{lemS3.4}. \end{proof}


\begin{lemma}\label{lemS3.6}
Let $\v>0$ and $\beta>3$, and  assume $\|\nu^{-1}wg\|_{L^\infty}+|wr|_{L^\infty{(\g_-)}}<\infty$. Then there exists a unique solution $f^\v$ to solve the approximate linearized steady Boltzmann equation \eqref{S3.3}. Moreover, it satisfies
\begin{align}\label{S3.93}
\|wf^\v\|_{L^\infty} +|wf^\v|_{L^\infty(\gamma)} \leq C_{\v} \Big\{ |wr|_{L^\infty{(\g_-)}}+\|\nu^{-1}wg\|_{L^\infty} \Big\},
\end{align}
where the positive constant $C_{\v}>0$ depends only on $\v$. Moreover, if $\Omega$ is {a strictly} convex domain,  
{$g$ is continuous in $\Omega\times \mathbb{R}^3$ and $r$ is continuous in $\g_-$,} then $f^\v$ is continuous away from the  grazing set $\gamma_0$.
\end{lemma}

\begin{proof} Let $f^n$ be the solution of \eqref{S3.54} constructed in Lemma \ref{lemS3.4} for $n\geq n_0$ with $n_0$ large enough.   Multiplying \eqref{S3.54} by $f^n$ {and using the coercivity estimate \eqref{1.10-2}}, one obtains that
\begin{align}\label{S3.94}
\v\|f^n\|^2_{L^2}+|f^n|^2_{L^2(\gamma_+)}+2c_0 \|
{(I-P)}f^n\|^2_{\nu}
\leq C_{\v} \|g\|^2_{L^2}+ \Big| (1-\frac1n)P_\gamma f^n+r \Big|^2_{L^2(\gamma_-)}.
\end{align}
Here the projection $P$ is defined by \eqref{1.10-1}.
A direct calculation shows that
\begin{align}
\Big| (1-\frac1n)P_\gamma f^n+r \Big|^2_{L^2(\gamma_-)}&\leq |P_\gamma f^n|^2_{L^2(\gamma_-)}+2|P_\gamma f^n|_{L^2(\gamma_-)}\cdot |r|_{L^2(\gamma_-)}+|r|^2_{L^2(\gamma_-)}\nonumber\\
&\leq |P_\gamma f^n|^2_{L^2(\gamma_+)}+\eta |P_\gamma f^n|^2_{L^2(\gamma_+)}+C_\eta |r|^2_{L^2(\gamma_-)},\nonumber
\end{align}
which together with \eqref{S3.94} yield that
\begin{align}\label{S3.95}
\v\|f^n\|^2_{L^2}+|(I-P_\gamma)f^n|^2_{\gamma_+}+2c_0 \|
{(I-P)}f^n\|^2_{\nu}
\leq \eta  |P_\gamma f^n|^2_{L^2(\gamma_+)}+C_{\v,\eta} (\|g\|^2_{L^2} +  |r|^2_{L^2(\gamma_-)}).
\end{align}
where $\eta>0$ is a small constant to be chosen later.

We still need to bound the first term on the right-hand side of \eqref{S3.95}. Firstly a direct calculation shows that
\begin{equation}\label{S3.96}
\frac12 |P_\gamma f^n|^2_{L^2(\gamma_+)}\leq |P_\gamma f^n I_{\gamma_+\backslash \gamma_+^{\v'}}|^2_{L^2(\gamma_+)},
\end{equation}
provided that $0<\v'\ll 1$. We note that
\begin{equation}\nonumber
f^n(x,v)=(I-P_\gamma)f^n(x,v)+P_\gamma f^n(x,v),\quad \  \forall\, (x,v)\in \gamma_+,
\end{equation}
which yields that
\begin{align}\label{S3.97}
|P_\gamma f^n I_{\gamma_+\backslash \gamma_+^{\v'}}|^2_{L^2(\gamma)}=2|f^n I_{\gamma_+\backslash \gamma_+^{\v'}}|^2_{L^2(\gamma)}+2|(I-P_\gamma)f^n  I_{\gamma_+\backslash \gamma_+^{\v'}}|^2_{L^2(\gamma)}.
\end{align}
On the other hand, it follows from \eqref{S3.54} that
\begin{equation}\nonumber
\f{1}{2}v\cdot \nabla_x(|f^n|^2)=-\v |f^n|^2-f^n Lf^n+g f^n,
\end{equation}
which yields that
\begin{equation}\label{S3.98}
\|v\cdot \nabla_x(|f^n|^2)\|_{L^1}\leq C\Big\{\|f^n\|^2_{L^2}+\|
{(I-P)}f^n\|^2_{\nu}+\|g\|^2_{L^2}\Big\}.
\end{equation}
It follows from \eqref{S3.98} and \eqref{S3.95} that
\begin{align}
|f^n I_{\gamma_+\backslash \gamma_+^{\v'}}|^2_{L^2(\gamma)}&=|(f^n)^2 I_{\gamma_+\backslash \gamma_+^{\v'}}|_{L^1(\gamma)}\notag \\
&\leq C_{\v'}\Big\{ \|f^n\|^2_{L^2}+ \|v\cdot \nabla_x(|f^n|^2)\|_{L^1} \Big\}\nonumber\\
&\leq C_{\v'} \Big\{\|f^n\|^2_{L^2}+\|
{(I-P)}f^n\|^2_{\nu}+\|g\|^2_{L^2}\Big\}\nonumber\\
&\leq C_{\v',\v} \eta  |P_\gamma f^n|^2_{L^2(\gamma_+)}+C_{\v',\v,\eta} \Big\{\|g\|^2_{L^2} +  |r|^2_{L^2(\gamma_-)}\Big\},\nonumber
\end{align}
where in the second line we have used \eqref{S2.2} which will be given in Lemma \ref{lemS1.1} later on. 
The above estimate together with \eqref{S3.96} and \eqref{S3.97} yield that
\begin{align*}
|P_\gamma f^n|^2_{L^2(\gamma_+)}\leq C_{\v',\v} \eta  |P_\gamma f^n|^2_{L^2(\gamma_+)}+C_{\v',\v,\eta} \Big\{\|g\|^2_{L^2} +  |r|^2_{L^2(\gamma_-)}\Big\}.
\end{align*}
Taking $\eta$ small so that $C_{\v',v} \eta\leq \frac12$, one obtains that
\begin{align}\label{S3.100}
|P_\gamma f^n|^2_{L^2(\gamma_+)}\leq C_{\v',\v} \Big\{\|g\|^2_{L^2} +  |r|^2_{L^2(\gamma_-)}\Big\}\Red{.}
\end{align}
Combining \eqref{S3.100} and \eqref{S3.95}, one has
\begin{align}\label{S3.101}
\|f^n\|^2_{L^2}+|f^n|^2_{L^2(\gamma_+)}+2c_0 \|
{(I-P)}f^n\|^2_{\nu}
\leq C_{\v',\v} (\|g\|^2_{L^2} +  |r|^2_{L^2(\gamma_-)}).
\end{align}
We apply \eqref{S3.24-1} and use \eqref{S3.101} to obtain
\begin{align*}
\|wf^n\|_{L^\infty}{+|wf^n|_{L^\infty(\g)}}&\leq C\Big\{\|\nu^{-1}w g\|_{L^\infty}+|wr|_{L^\infty{(\g_-)}}+\|f^n\|_{L^2}\Big\}\nonumber\\
&\leq C_{\v',\v}  \Big\{\|\nu^{-1}w g\|_{L^\infty}+|wr|_{L^\infty{(\g_-)}}\Big\}
\end{align*}
Taking the difference $f^{n_1}-f^{n_2}$ with $n_1,n_2\geq n_0$, we know that
\begin{align}\label{S3.102}
\begin{cases}
\v (f^{n_1}-f^{n_2})+ v\cdot \nabla_x (f^{n_1}-f^{n_2})+L (f^{n_1}-f^{n_2})=0,\\[2mm]
 (f^{n_1}-f^{n_2})(x,v)|_{\gamma_{-}}=(1-\frac{1}{n_1})P_{\gamma} (f^{n_1}-f^{n_2})+(\frac{1}{n_2}-\frac{1}{n_1}) P_\gamma f^{n_2}.
\end{cases}
\end{align}
Multiplying \eqref{S3.102} by $ f^{n_1}-f^{n_2}$, and integrating it over $\Omega\times\mathbb{R}^3$, by similar arguments as in \eqref{S3.94}-\eqref{S3.101}, we can obtain
\begin{align}\label{S3.103}
&\|(f^{n_1}-f^{n_2})\|^2_{L^2}+|(f^{n_1}-f^{n_2})|^2_{L^2(\gamma_+)}+2c_0 \|
{(I-P)}(f^{n_1}-f^{n_2})\|^2_{\nu}\nonumber\\
&\leq C_{\v',\v} \Big| (\frac{1}{n_2}-\frac{1}{n_1}) P_\gamma f^{n_2}\Big|^2_{L^2(\gamma_-)}
\leq C_{\v',\v} \cdot\Big\{(\frac1{n_1})^2+(\frac1{n_2})^2\Big\}\cdot |f^{n_2}|^2_{L^2(\gamma_+)}\nonumber\\
&\leq C_{\v',\v} \cdot (\|g\|^2_{L^2} +  |r|^2_{L^2(\gamma_-)})\cdot \Big\{(\frac1{n_1})^2+(\frac1{n_2})^2\Big\}\rightarrow0,
\end{align}
as $n_1$, $n_2 \rightarrow \infty$,
where we have used the uniform estimate \eqref{S3.101} in the last inequality.  Applying \eqref{S3.24-1} to $  f^{n_1}-f^{n_2}$  and using \eqref{S3.103}, then one has
\begin{multline*}
\|w(f^{n_1}-f^{n_2})\|_{L^\infty}{+|w(f^{n_1}-f^{n_2})|_{L^\infty(\g)}}\leq C \Big|w (\frac{1}{n_2}-\frac{1}{n_1}) P_\gamma f^{n_2}\Big|_{L^\infty(\gamma_-)}+C\|f^{n_1}-f^{n_2}\|_{{L^2}}
\\
\leq C_{\v',\v} \cdot (\|wg\|_{L^\infty} +  |wr|_{L^\infty(\gamma_-)})\cdot \Big\{\frac1{n_1}+\frac1{n_2}\Big\}\rightarrow0,
\end{multline*}
as $n_1,\ n_2 \rightarrow \infty$, which yields that $wf^n$ is a Cauchy sequence in $L^\infty$. We denote $f^\v=\lim_{n\rightarrow\infty} f^n$, then it is direct to check that $f^\v$ is a solution to \eqref{S3.3}, and \eqref{S3.93} holds. The continuity of $f^\v$ is easy to obtain since the convergence of  sequences is always in $L^\infty$ and $f^n$ is continuous away from the grazing set.  Therefore we have completed the proof of Lemma \ref{lemS3.6}.
\end{proof}


Now we assume
\begin{align}\label{S3.55}
\iint_{\Omega\times\mathbb{R}^3} g(x,v) \mu^{\frac{1}{2}}(v) \dd v\dd x= \int_{\gamma_-} r(x,v) \mu^{\frac{1}{2}}(v) \dd\gamma=0.
\end{align}

\begin{lemma}\label{lemS3.5}
Let $f$ be a solution
\begin{align*}
	v\cdot \nabla_x f+Lf=g,\quad f_{\gamma_-} = P_{\gamma}f+r
\end{align*}
	in the weak sense of
\begin{align*}
	\int_{\gamma} \psi f \  \{v\cdot n(x)\} \dd v \dd S_x-\int_{\Omega\times\mathbb{R}^3} v\cdot \nabla_x\psi f \dd v\dd x= -\int_{\Omega\times\mathbb{R}^3} \psi \ Lf \dd v\dd x+\int_{\Omega\times\mathbb{R}^3} \psi g  \dd v \dd x,
\end{align*}
	where $\psi \in H^1(\Omega\times\mathbb{R}^3)$. Assume $\int_{\Omega\times\mathbb{R}^3}f \sqrt{\mu} \dd v\dd x=0$ and \eqref{S3.55}, then it holds that
\begin{align*}
	\|
{P}f\|^2_{L^2}\leq C\Big\{ \|
{(I-P)}f\|^2_{\nu}+\|g\|_{L^2} + |(I-P_\gamma)f|^2_{L^2(\gamma_+)} + |r|^2_{L^2(\gamma_-)}\Big\}.
\end{align*}
\end{lemma}

\begin{proof} The proof is almost the same to Lemma 3.3 in \cite{EGKM}, the details are omitted here for simplicity of presentation.
\end{proof}

\vspace{2mm}

\begin{lemma}\label{lemS3.7}
Assume \eqref{S3.55}. Let $\beta>3+|\kappa|$, and  assume $\|\nu^{-1}wg\|_{L^\infty}+|wr|_{L^\infty{(\g_-)}}<\infty$.   Then there exists a  unique solution $f=f(x,v)$  to the linearized steady Boltzmann equation
\begin{align}\label{S3.106}
v\cdot \nabla_xf+Lf=g,\quad
f(x,v)|_{\gamma_{-}}=P_{\gamma}f+r,
\end{align}
such that $\int_{\Omega\times\mathbb{R}^3} f \sqrt{\mu} dvdx=0 $ and
\begin{equation}\label{S3.107}
\|wf\|_{L^\infty} +|wf|_{L^\infty(\gamma)} \leq C\Big\{ |wr|_{L^\infty{(\gamma_-)}}+\|\nu^{-1}wg\|_{L^\infty} \Big\}.
\end{equation}
Moreover, if  $\Omega$ is a strictly convex domain, $g$ is continuous in $\Omega\times\mathbb{R}$ and $r$ is continuous in $\g_-$,  
then $f$ is continuous away from the grazing set $\gamma_0$.
\end{lemma}

\begin{proof}
Let $f^\v$ be the solution of \eqref{S3.3} constructed in Lemma \ref{lemS3.6} for $\v>0$.   Multiplying the first equation of \eqref{S3.3} by $\sqrt{\mu}$, taking integration over $\Omega\times\mathbb{R}^3$, and noting \eqref{S3.55}, it is straightforward to see that
\begin{equation*}
\int_{\Omega\times\mathbb{R}^3} f^\v(x,v)\mu^{\frac{1}{2}}(v) \ \dd v\dd x=0,
\end{equation*}
for any $\vep>0$. Multiplying the first equation of \eqref{S3.3} by $f^\v$ and integrating the resultant equation over $\Omega\times\mathbb{R}^3$, it follows from Cauchy inequality that
\begin{align}\label{S3.109}
&\v\|f^\v\|^2_{L^2}+\frac12|(I-P_\gamma)f^\v|^2_{L^2(\gamma_+)}+c_0\|
(I-P)f^\v\|^2_{\nu}\nonumber\\
&\leq \eta [\|f^\v\|^2_{\nu}+|P_\gamma f^\v|^2_{L^2(\gamma_+)}]+C_\eta [\|\nu^{-\frac12}g\|^2_{L^2}+|r|^2_{L^2(\gamma_-)}].
\end{align}
Applying Lemma \ref{lemS3.5} to $f^\v$, we obtain
\begin{align}\nonumber
\|
{P}f^\v\|^2_{L^2}\leq C\Big\{ \|
{(I-P)}f^\v\|^2_{\nu}+\v\|f^\v\|^2_{L^2}+ |(I-P_\gamma)f^\v|^2_{L^2(\gamma_+)} +\|g\|^2_{L^2}+ |r|^2_{L^2(\gamma_-)}\Big\},
\end{align}
which together with \eqref{S3.109} imply that
\begin{multline}\label{S3.110}
\v\|f^\v\|^2_{L^2}+\|f^\v\|^2_{\nu}+|(I-P_\gamma)f^\v|^2_{L^2(\gamma_+)}\\
\leq C\eta [\|f^\v\|^2_{\nu}+|P_\gamma f^\v|^2_{L^2(\gamma_+)}]+C_\eta [\|\nu^{-\frac12}g\|^2_{L^2}+|r|^2_{L^2(\gamma_-)}],
\end{multline}
where $\eta>0$ is a small positive constant to be determined later.

To control the term  $|P_\gamma f^\v|^2_{L^2(\gamma_+)}$ on the right-hand side of \eqref{S3.110}, we should be  careful since we do not have the uniform bound on  $\|f^\v\|_{L^2}$. Denote
$$
z^\v_{\gamma}(x):=\int_{v'\cdot n(x)>0} f^\v(x,v') \sqrt{\mu(v')} \{v'\cdot n(x)\}\dd v',
$$
then one has  $P_{\gamma} f^\v=z^\v_{\gamma}(x) \mu^{\frac{1}{2}}(v)$. A direct calculation shows that
\begin{align}\label{S3.111}
\int_{v\cdot n(x)\geq \v', \  \v'\leq |v|\leq \frac1{\v'}} \nu(v) \mu(v) |v\cdot n(x)| \dd v\geq c_1>0,
\end{align}
provided that $0<\v'\ll1$, where $c_1>0$ is a positive constant independent of $\v'$. By using \eqref{S3.111}, we have that
\begin{align*}
|\sqrt{\nu}P_\gamma f^\v I_{\gamma_+\backslash \gamma_+^{\v'}}|^2_{L^2}&=\int_{\partial\Omega}|z_\gamma^\v (x)|^2 {\dd S_x}\cdot
\int_{v\cdot n(x)\geq \v', \  \v'\leq |v|\leq \frac1{\v'}} \nu(v) \mu(v) |v\cdot n(x)| \dd v\nonumber\\
&\geq c_1 |P_\gamma f^\v|^2_{L^2(\gamma_+)},
\end{align*}
which yields immediately that
\begin{align}\label{S3.113}
|P_\gamma f^\v|^2_{L^2(\gamma_+)}&\leq C|\sqrt{\nu}P_\gamma f^\v I_{\gamma_+\backslash \gamma_+^{\v'}}|^2_{L^2}\notag\\
&\leq C \Big\{ |\sqrt{\nu} f^\v I_{\gamma_+\backslash \gamma_+^{\v'}}|^2_{L^2}+ |\sqrt{\nu}(I-P_\gamma) f^\v I_{\gamma_+\backslash \gamma_+^{\v'}}|^2_{L^2}\Big\}\nonumber\\
&\leq C|\sqrt{\nu} f^\v I_{\gamma_+\backslash \gamma_+^{\v'}}|^2_{L^2}+C\eta [\|f^\v\|^2_{\nu}+|P_\gamma f^\v|^2_{L^2(\gamma_+)}]\notag\\
&\quad+C_\eta [\|\nu^{-\frac12}g\|^2_{L^2}+|r|^2_{L^2(\gamma_-)}].
\end{align}
It follows from \eqref{S3.3} that
\begin{align*}
{\f12}v\cdot \nabla_x(\nu|f^\v|^2)=-\v \nu|f^\v|^2-\nu f^\v Lf^\v +\nu f^\v g,
\end{align*}
which implies that
\begin{align}\label{S3.115}
\|v\cdot \nabla_x(\nu|f^\v|^2)\|_{L^1}\leq C\{\|
{f^\vep}\|^2_{\nu}+\|g\|^2_{\nu}\}.
\end{align}
It follows from \eqref{S3.115} and \eqref{S2.2} that
\begin{align*}
\f{1}{2}|\sqrt{\nu} f^\v I_{\gamma_+\backslash \gamma_+^{\v'}}|^2_{L^2(\gamma_+)}
&={\f12}|\nu (f^\v)^2 I_{\gamma_+\backslash \gamma_+^{\v'}}|_{L^1(\gamma_+)}
\leq C\Big\{ \|\nu\ (f^\v)^2\|_{L^1} +\|v\cdot \nabla_x(\nu|f^\v|^2)\|_{L^1}\Big\}\nonumber\\
&\leq  C\{\|f^\varepsilon\|^2_{\nu}+\|g\|^2_{\nu}\},
\end{align*}
which together with \eqref{S3.113} and by taking $\eta>0$ suitably small, yield that
\begin{equation}\label{S3.117}
|P_\gamma f^\v|^2_{L^2(\gamma_+)}\leq C\|f^\v\|^2_{\nu}+C [\|\nu^{-\frac12}g\|^2_{L^2}+|r|^2_{L^2(\gamma_-)}].
\end{equation}
Combining \eqref{S3.117} and \eqref{S3.110}, then taking $\eta>0$ small, one has that
\begin{equation}\label{S3.118}
\|f^\v\|^2_{\nu}+|f^\v|^2_{L^2(\gamma_+)}\leq C [\|\nu^{-\frac12}g\|^2_{L^2}+|r|^2_{L^2(\gamma_-)}].
\end{equation}
Applying \eqref{S3.24-1} to $f^\v$ and using \eqref{S3.118}, then we obtain
\begin{equation}\label{S3.119}
\|wf^\v\|_{L^\infty}{+|wf^\vep|_{L^\infty(\g)}}\leq C [\|\nu^{-\frac12}wg\|_{L^\infty}+|wr|_{L^\infty(\gamma_-)}],
\end{equation}

Next we consider the convergence of $f^\v$ as $\v\rightarrow0+$. For any $\v_1,\v_2>0$, we consider the difference $f^{\v_2}-f^{\v_1}$ satisfying
\begin{equation}\label{S3.120}
\begin{cases}
v\cdot \nabla_x (f^{\v_2}-f^{\v_1})+ L(f^{\v_2}-f^{\v_1})=-\v_2 f^{\v_2}+\v_1 f^{\v_1},\\[2mm]
(f^{\v_2}-f^{\v_1})|_{\gamma_-}=P_\gamma (f^{\v_2}-f^{\v_1}).
\end{cases}
\end{equation}
Multiplying \eqref{S3.120} by $f^{\v_2}-f^{\v_1}$, integrating the resultant equation and by similar arguments as in \eqref{S3.109}-\eqref{S3.118}, one gets
\begin{align}\label{S3.121}
&\|f^{\v_2}-f^{\v_1}\|^2_{\nu}+|f^{\v_2}-f^{\v_1}|^2_{L^2(\gamma_+)}\nonumber\\
&\leq C \|\nu^{-\frac12}(\v_2 f^{\v_2}-\v_1 f^{\v_1})\|^2_{L^2}
\leq C(\v_1^2+\v_2^2) [\|w f^{\v_1}\|^2_{L^\infty} + \|w f^{\v_2}\|^2_{L^\infty} ]\nonumber\\
&\leq C(\v_1^2+\v_2^2)\cdot  [\|\nu^{-1}wg\|_{L^\infty}+|wr|_{L^\infty(\gamma_-)}]^2,
\end{align}
as $\v_1$, $\v_2\rightarrow 0+$, where we have used \eqref{S3.119} in the last inequality. Finally, applying \eqref{S3.24-1} to $f^{\v_2}-f^{\v_1}$, then we obtain
\begin{align}\label{S3.122}
\|\nu w(f^{\v_2}-f^{\v_1})\|_{L^\infty}{+|\nu w(f^{\v_2}-f^{\v_1})|_{L^\infty(\g)}}&\leq C\Big\{ \|w (\v_2 f^{\v_2}-\v_1 f^{\v_1})\|_{L^\infty} +\|f^{\v_2}-f^{\v_1}\|_{\nu} \Big\}\nonumber\\
&\leq C(\v_1+\v_2)\cdot  [\|\nu^{-1}wg\|_{L^\infty}+|wr|_{L^\infty(\gamma_-)}],
\end{align}
as $\v_1$, $\v_2\rightarrow 0+$,
where we have used \eqref{S3.121} and \eqref{S3.119} above. Here we have to demand $\beta>3+|\kappa|$ so that we can apply \eqref{S3.24-1}.  And we point out here that we can only obtain the convergence in a weak norm $L^\infty_{\nu w}$ but not $L^\infty_w$, the main reason is that for soft potentials, in order to get the uniform $L^\infty_w$ estimate, one has to demand $g$ to has the more velocity weight. With \eqref{S3.122}, we know that there exists a function $f$ so that $\|\nu w(f^{\v}-f)\|_{L^\infty}\rightarrow0$ as $\v\rightarrow 0+$. And it is 
direct to see that $f$ solves \eqref{S3.106}. Also,  \eqref{S3.107} follows immediately from \eqref{S3.119}.  If $\Omega$ is convex, the continuity of $f$ 
directly follows from the $L^\infty$-convergence. Therefore the proof of Lemma \ref{lemS3.7} is complete.
\end{proof}

\subsection{Proof of Theorem \ref{thm1.1}.} We consider the following iterative sequence
\begin{align}\label{S3.123}
\begin{cases}
v\cdot \nabla_x f^{j+1}+Lf^{j+1}=\Gamma(f^j,f^j),\\[1.5mm]
f^{j+1}|_{\gamma_-}=P_\gamma f^{j+1}+\frac{{\mu_\theta}-\mu}{\sqrt{\mu}} +\frac{{\mu_\theta}-\mu}{\sqrt{\mu}}\int_{v'\cdot n(x)>0} f^j(x,v')\sqrt{\mu(v')} \{v'\cdot n(x)\} \dd v',
\end{cases}
\end{align}
for $j=0,1,2\cdots$ with $f^0\equiv0$. A direct computation shows that
\begin{align}\label{S3.124}
\int_{\Omega\times\mathbb{R}^3} \Gamma(f^j,f^j)\mu^{\frac{1}{2}}(v) \dd v\dd x=0,\quad 
\int_{v\cdot n(x)<0} [\mu_{{\theta}}
(v)-\mu(v)] \{v\cdot n(x)\} \dd v =0,
\end{align}
which yield that
\begin{align*}
\int_{\gamma_-} \left\{\frac{{\mu_\theta}-\mu}{\sqrt{\mu}} +\frac{{\mu_\theta}-\mu}{\sqrt{\mu}}\int_{v'\cdot n(x)>0} f^j(x,v')\sqrt{\mu(v')} \{v'\cdot n(x)\} \dd v' \right\} \mu^{\frac{1}{2}}(v)\dd\gamma\equiv0.
\end{align*}
We note that
\begin{align}\label{S3.126}
\|\nu^{-1} w \Gamma(f^j,f^j)\|_{L^\infty}\leq C \|wf^j\|^2_{L^\infty},
\end{align}
and
\begin{multline}\label{S3.127}
\left| w\left\{\frac{{\mu_\theta}-\mu}{\sqrt{\mu}} +\frac{{\mu_\theta}-\mu}{\sqrt{\mu}}\int_{v'\cdot n(x)>0} f^j(x,v')\sqrt{\mu(v')} \{v'\cdot n(x)\} \dd v'\right\} \right|_{L^\infty{(\g_-)}}\\
\leq C\delta+C\delta |f^j|_{L^\infty{(\g_+)}}.
\end{multline}
Noting \eqref{S3.124}-\eqref{S3.127}, and using Lemma \ref{lemS3.7}, we can solve \eqref{S3.123} inductively for $j=0,1,2,\cdots$. Moreover, it follows from \eqref{S3.107}, \eqref{S3.126} and \eqref{S3.127} that
\begin{align}\label{S3.128}
\|wf^{j+1}\|_{L^\infty}+|wf^{j+1}|_{L^\infty(\gamma)}\leq C_1\delta+C_1\delta |f^j|_{L^\infty{(\g_+)}}+C_1 \|wf^j\|^2_{L^\infty}.
\end{align}
By induction, we shall prove  that
\begin{equation}\label{S3.129}
\|wf^{j}\|_{L^\infty}+|wf^{j}|_{L^\infty(\gamma)}\leq 2C_1\delta,\quad\mbox{for} \  j=1,2,\cdots.
\end{equation}
Indeed, for $j=0$, it follows from $f^0\equiv0$ and \eqref{S3.128} that
\begin{equation*}
\|wf^{1}\|_{L^\infty}+|wf^{1}|_{L^\infty(\gamma)}\leq C_1\delta.
\end{equation*}
Now we assume that \eqref{S3.129} holds for $j=1,2\cdots, l$,  then we consider the case for $j=l+1$. Indeed it follows from \eqref{S3.128}  that
\begin{align*}
\|wf^{l+1}\|_{L^\infty}+|wf^{l+1}|_{L^\infty(\gamma)}&\leq C_1\delta+C_1\delta |f^l|_{L^\infty{(\g)}}+C_1 \|wf^l\|^2_{L^\infty}\nonumber\\
&\leq C_1\delta(1+2C_1\delta+4C_1^2\delta)\leq  \f{3}{2}C_1\delta,
\end{align*}
where we have used  $\eqref{S3.129}$ with $j=l$, and chosen $\delta>0$ small enough such that  $2C_1\delta+4C_1^2 \delta\leq 1/2$.  
Therefore we have proved \eqref{S3.129} by induction. 

Finally we consider the convergence of sequence  $f^j$. For the difference $f^{j+1}-f^j$, we have
\begin{equation}\label{S3.131}
\begin{cases}
\dis v\cdot \nabla_x (f^{j+1}-f^j)+L(f^{j+1}-f^j)=\Gamma(f^{j}-f^{j-1},f^{j})+\Gamma(f^{j-1},f^{j}-f^{j-1}),\\[2mm]
\dis (f^{j+1}-f^j)|_{\gamma_-}=P_\gamma (f^{j+1}-f^j)\\[2mm]
\dis \qquad\qquad\qquad\qquad+\frac{{\mu_\theta}-\mu}{\sqrt{\mu}}\int_{v'\cdot n(x)>0} [f^j-f^{j-1}](x,v')\sqrt{\mu(v')} \{v'\cdot n(x)\} \dd v'.
\end{cases}
\end{equation}
Applying \eqref{S3.107} to \eqref{S3.131}, we have that
\begin{align}\label{S3.133}
&\|w\{f^{j+1}-f^j\}\|_{L^\infty}+|w\{f^{j+1}-f^j\}|_{L^\infty(\gamma)}\nonumber\\
&\leq C\Big\{ \|\nu^{-1}w\Gamma(f^{j}-f^{j-1},f^{j})\|_{L^\infty} +\|\nu^{-1}w\Gamma(f^{j-1},f^{j}-f^{j-1})\|_{L^\infty}\Big\}\nonumber\\
&\quad+C\Big|w\Big\{\frac{{\mu_\theta}-\mu}{\sqrt{\mu}}\int_{v'\cdot n(x)>0} [f^j-f^{j-1}](x,v')\sqrt{\mu(v')} \{v'\cdot n(x)\} \dd v'\Big\}\Big|_{L^\infty{(\gamma_-)}}\nonumber\\
&\leq C[\delta+\|wf^{j}\|_{L^\infty}+\|wf^{j-1}\|_{L^\infty}]\cdot \Big\{\|w(f^{j}-f^{j-1})\|_{L^\infty}+|w(f^j-f^{j-1})|_{L^\infty{(\gamma_+)}}\Big\}\nonumber\\
&\leq C\delta \Big\{\|w(f^{j}-f^{j-1})\|_{L^\infty}+|w(f^j-f^{j-1})|_{L^\infty{(\g_+)}}\Big\}\nonumber\\
&\leq \frac12 \Big\{\|w(f^{j}-f^{j-1})\|_{L^\infty}+|w(f^j-f^{j-1})|_{L^\infty{(\g_+)}}\Big\},
\end{align}
where we have used \eqref{S3.129} and taken $\delta>0$ small such that $C\delta\leq 1/2$. Hence $f^j$ is a Cauchy sequence in $L^\infty$, then we obtain the solution by taking the limit $f_*=\lim_{j\rightarrow\infty} f^j$.  The uniqueness can also be obtained by using the inequality as \eqref{S3.133}.

If $\Omega$ is convex, the continuity of $f_*$ is 
a direct consequence of $L^\infty$-convergence. The positivity of $F_*:=\mu+\sqrt{\mu} f_*$ will be proved in {Section 4}.  Therefore we complete the proof of Theorem \ref{thm1.1}. \qed


\section{Dynamical stability under small perturbations}

In this section, we are concerned with the large-time asymptotic stability of the steady solution $F_*$ obtained in Theorem \ref{thm1.1}. For this purpose, we introduce the perturbation
$$
f(t,x,v):=\f{F(t,x,v)-F_*(x,v)}{\sqrt{\mu}}.
$$
Then the initial-boundary value problem on $f(t,x,v)$ reads as
\begin{equation}\label{4.1}
\left\{
\begin{aligned}
&\partial_{t}f+v\cdot\nabla_{x}f+Lf=-L_{\sqrt{\mu}f_{*}}f+\Gamma(f,f),\\
&f(t,x,v)|_{t=0}=f_{0}(x,v):=\f{F(0,x,v)-F_*(x,v)}{\mu^{\frac{1}{2}}(v)},\\
&f(t,x,v)|_{\g_{-}}=P_{\g}f+\frac{\mu_{\theta}-\mu}{\sqrt{\mu}}\int_{n(x)\cdot u>0}f(t,x,u)\sqrt{\mu(u)}|n(x)\cdot u|\dd u.
\end{aligned}
\right.
\end{equation}
Here $P_{\g}f$ is defined in \eqref{S3.4}, the linearized collision operator $L$ is defined in \eqref{1.9-1}, the nonlinear term $\Gamma(f,f)$ is defined in \eqref{1.10} and
\begin{align*}
L_{\sqrt{\mu}f_{*}}f:=-\frac{1}{\sqrt{\mu}}\left[Q(\sqrt{\mu}f_{*},\sqrt{\mu}f)+Q(\sqrt{\mu}f,\sqrt{\mu}f_{*})\right].
\end{align*}
Recall \eqref{WF}. Let
\begin{align}\label{4.2}
h(t,x,v):=w(v)f(t,x,v).
\end{align}
Then one can reformulate \eqref{4.1} as
\begin{equation}\label{nph}
\left\{\begin{aligned}
&\partial_{t}h+v\cdot\nabla_{x}h+\nu_*(x,v)h-K_{w}h=-wL_{\sqrt{\mu}f_*}f+w\Gamma(f,f),\\
&h|_{t=0}=wf_0:=h_0,\\
&h|_{\g_{-}}=\frac{1}{\tilde{w}(v)}\int_{n(x)\cdot u>0}h(u)\tilde{w}(u)\dd \sigma(x)+w(v)\frac{\mu_{\theta}-\mu}{\sqrt{\mu}}\int_{n(x)\cdot u>0}h(u)\tilde{w}(u)\dd \sigma(x),
\end{aligned}\right.
\end{equation}
where $\tilde{w}$ and $K_wh$ are the same as ones defined before, and for each $x\in \pa\Omega$, $\dd \sigma(x)$ denotes the probability measure 
$$ 
\dd\sigma(x)=\{n(x)\cdot u\}\dd u
$$ 
in the velocity space
$\CV(x):=\{u\in \mathbb{R}^3:n(x)\cdot u>0\}$. Also, we have denoted the  collision frequency ${\nu}_*(x,v)$ at the stationary solution as
\begin{align}\label{4.4}
\nu_*(x,v):=\int_{\mathbb{R}^3}\int_{\mathbb{S}^2}B(|v-u|,\omega)F_*(x,u)\dd u.
\end{align}
For simplicity, to the end we denote
\begin{align}\label{4.1.5-2}
r[h](t,x,v):=\f{\mu_{\theta}-\mu}{\sqrt{\mu}}\int_{n(x)\cdot u>0}h(t,x,v)\tilde{w}(u)\dd \sigma(x).
\end{align}

\subsection{Characteristics for time-dependent problem}

Given $(t,x,v),$ let $[X(s),V(s)]$ 
be the backward bi-characteristics associated with the initial-boundary value problem \eqref{4.1} on the Boltzmann equation, which is determined by
\begin{align*}
\begin{cases}
\dis \frac{{\dd}X(s)}{{\dd}s}=V(s),
\frac{{\dd}V(s)}{{\dd}s}=0,\\[2mm]
[X(t),V(t)]=[x,v].
\end{cases}
\end{align*}
The solution is then given by
\begin{equation*}
[X(s),V(s)]=[X(s;t,x,v),V(s;t,x,v)]=[x-(t-s)v,v].
\end{equation*}
Similarly as in the steady case, for each $(x,v)$ with $x\in \bar{\Omega}$ and $v\neq 0,$ we define its {backward exit time} $t_{\mathbf{b}}(x,v)\geq 0$ to be the last moment at which the
back-time straight line $[X({-\tau}
;0,x,v),V({-\tau}
;0,x,v)]$ remains in $\bar{\Omega}$:
\begin{equation*}
t_{\mathbf{b}}(x,v)={\sup\{s \geq 0:x-\tau v\in\bar{\Omega}\text{ for }0\leq \tau\leq s\}.}
\end{equation*}
We also define
\begin{equation*}
x_{\mathbf{b}}(x,v)=x(t_{\mathbf{b}})=x-t_{\mathbf{b}}v\in \partial \Omega .
\end{equation*}
Let $x\in \bar{\Omega}$, $(x,v)\notin \gamma _{0}\cup \g_{-}$ and 
$
(t_{0},x_{0},v_{0})=(t,x,v)$. For $v_{k+1}\in \mathcal{V}_{k+1}:=\{v_{k+1}\cdot n(x_{k+1})>0\}$, the back-time cycle is defined as
\begin{equation}\label{2.18}
\left\{\begin{aligned}
X_{cl}(s;t,x,v)&=\sum_{k}\Fi_{[t_{k+1},t_{k})}(s)\{x_{k}-v_k(t_{k}-s)\},\\[1.5mm]
V_{cl}(s;t,x,v)&=\sum_{k}\Fi_{[t_{k+1},t_{k})}(s)v_{k},
\end{aligned}\right.
    \end{equation}
with
\begin{equation*}
(t_{k+1},x_{k+1},v_{k+1})=(t_{k}-t_{\mathbf{b}}(x_{k},v_{k}),x_{\mathbf{b}}(x_{k},v_{k}),v_{k+1}).
\end{equation*}
For $k\geq 2$,
the iterated integral  means that 
\begin{equation*}
\int_{\Pi _{l=1}^{k-1}\mathcal{V}_{l}}\Pi _{l=1}^{k-1}\dd\sigma _{l}:=
\int_{\mathcal{V}_{1}}...\left\{ \int_{\mathcal{V}_{k-1}}\,\dd\sigma_{k-1}\right\} \cdots\dd\sigma _{1}.
\end{equation*}
where
\begin{align*}
{\dd\sigma_i:=\mu(v_i)\{n(x_i)\cdot v_i\}\dd v_i.}
\end{align*}
Note that all $v_{l}$ ($l=1,2,\cdots$) are independent variables, and $t_k$, $x_k$ depend on $t_l$, $x_l$, $v_l$ for $l\leq k-1$, and the velocity space $\mathcal{V}_{l}$ implicitly depends on $(t,x,v,v_{1},v_{2},\cdots,v_{l-1}).$\\


Define the near-grazing set of $\gamma_{+}$ and  $\gamma_{-}$ as
\begin{align}\label{S2.1}
\gamma^\v_{\pm}=\left\{(x,v)\in\gamma_{\pm}~:~ |v\cdot n(x)|<\v~\mbox{or}~|v|\geq\v~\mbox{or}~|v|\leq\f1\v\right\}.
\end{align}
Then we have

\begin{lemma}[\cite{CIP,Guo2}]\label{lemS1.1}
Let $\v>0$ be a small positive constant, then it holds that
\begin{align}\label{S2.2}
|f\Fi_{\gamma_{\pm}\setminus \gamma_{\pm}^\v}|_{L^1(\ga)}
\leq C_{\v,\Omega} \Big\{\|f\|_{L^1} + \|v\cdot \nabla_x f\|_{L^1}\Big\},
\end{align}
and
\begin{align}\label{S2.3}
\int_0^t |f(\tau)\Fi_{\gamma_+\setminus \gamma_{+}^\v}|_{L^1(\ga)}\dd\tau\leq C_{\v,\Omega} \bigg\{\|f(0)\|_{L^1}+\int_0^t \Big[\|f(\tau)\|_{L^1}+\|[\partial_{\tau}+v\cdot\nabla_x] f(\tau)\|_{L^1}\Big]\dd s\bigg\},
\end{align}
where the positive constant $C_{\v}>0$ depends only on $\v$.
\end{lemma}

\subsection{Linear problem}
In this part, we study the following linear inhomogeneous problem:
\begin{equation}\label{4.1.0}
\left\{
\begin{aligned}
&\pa_{t}f+v\cdot\nabla_{x}f+Lf=g, \\
&f|_{\gamma_{-}}=P_{\g}f+\f{\mu_{\theta}-\mu}{\sqrt{\mu}}\int_{n(x)\cdot u>0}f\sqrt{\mu}|n(x)\cdot u|\dd u,\\
&f(t,x,v)|_{t=0}=f_0(x,v),
\end{aligned}
\right.
\end{equation}
where $g$ is a given function. Recall that $h(t,x,v)$ defined in \eqref{4.2}. Then the equation of $h$ as well as the boundary condition read:
\begin{equation}\label{4.1.1}
\left\{
\begin{aligned}
&\pa_{t}h+v\cdot\nabla_{x}h+\nu_*(x,v)h-K_wh=wg, \\
&h|_{\gamma_{-}}=\frac{1}{\tilde{w}(v)}\int_{n(x)\cdot u>0}h(u)\tilde{w}(u)\dd\sigma(x)+w(v)\f{\mu_{\theta}-\mu}{\sqrt{\mu}}\int_{n(x)\cdot u>0}h(u)\tilde{w}(u)\dd\sigma(x).\\
\end{aligned}
\right.
\end{equation}
Recall the stochastic cycle defined in \eqref{2.18}. For any $0\leq s\leq t$, we define
$$
\dis I(t,s):=e^{-\int_s^t\nu_*(X_{cl}(\tau),V_{cl}(\tau))\dd\tau}.
$$
The following Lemma is to establish the mild formulation for \eqref{4.1.1}. Since its proof is almost the same as for \cite[Lemma 24]{Guo2}, we omit details for brevity.

\begin{lemma}
Let $k\geq 1$ be an integer and $h(t,x,v) \in L^{\infty}$ satisfy \eqref{4.1.1}. For any $t>0$, for almost every $(x,v)\in\bar{\Omega}\times \mathbb{R}^3\setminus\g_0\cup\g_-$ and for any $0\leq s\leq t$, it holds that
\begin{align}\label{4.1.5}
h(t,x,v)=\sum_{n=1}^4\CI_n+\sum_{n=5}^{14}\Fi_{\{t_1>s\}}\CI_{n},
\end{align}
where
\begin{align}
&\CI_1=\Fi_{\{t_1\leq s\}}I(t,s)h(s,x-(t-s)v,v)\nonumber\\
&\CI_2+\CI_3+\CI_4=\int_{\max\{t_{1},s\}}^{t}I(t,\tau)[K_w^mh+K_w^ch+g](\tau,x-(t-\tau)v,v)\dd\tau\nonumber\\
&\CI_5=\frac{I(t,t_{1})}{\tilde{w}(v)}\int_{\prod_{j=1}^{k-1}\CV_{j}}\sum^{k-1}_{l=1}\Fi_{\{t_{l+1}\leq s<t_{l}\}}h(s,x_{l}-(t_{l}-s)v_{l},v_{l})\dd\Sigma_{l}(s)\nonumber\\
&\CI_6+\CI_7+\CI_8=\frac{I(t,t_{1})}{\tilde{w}(v)}\int_{\prod_{j=1}^{k-1}\CV_{j}}\int_{s}^{t_{l}}\sum^{k-1}_{l=1}\Fi_{\{t_{l+1}\leq s<t_{l}\}}\notag\\
&\qquad\qquad\qquad\qquad\qquad\qquad\qquad\quad\times[K_w^mh+K^c_wh+wg](\tau,x_{l}-(t_{l}-\tau)v_{l},v_{l})\dd\Sigma_{l}(\tau)\dd\tau\nonumber\\
&\CI_9+\CI_{10}+\CI_{11}=\frac{I(t,t_{1})}{\tilde{w}(v)}\int_{\prod_{j=1}^{k-1}\CV_{j}}\int_{t_{l+1}}^{t_{l}}\sum^{k-1}_{l=1}\Fi_{\{t_{l+1}> s\}}\notag\\
&\qquad\qquad\qquad\qquad\qquad\qquad\qquad\quad\times[K_w^mh+K^c_wh+wg](\tau,x_{l}-(t_{l}-\tau)v_{l},v_{l})\dd\Sigma_{l}(\tau)\dd\tau\nonumber
\end{align}
\begin{align}
&\CI_{12}=I(t,t_{1})wr(t_{1},x_{1},v),\quad\CI_{13}=\frac{I(t,t_{1})}{\tilde{w}(v)}\int_{\prod_{j=1}^{k-1}\CV_{j}}\sum^{k-2}_{l=1}\Fi_{\{t_{l+1}> s\}}wr(t_{l+1},x_{l+1},v_{l})\dd\Sigma_{l}(t_{l+1})\nonumber\\
&\CI_{14}=\frac{I(t,t_{1})}{\tilde{w}(v)}\int_{\prod_{j=1}^{k-1}\CV_{j}}\Fi_{\{t_{k}>s\} }h(t_{k},x_{k},v_{k-1})\dd\Sigma_{k-1}(t_{k}),\notag
\end{align}
where the measure $\dd\Sigma_{l}(\tau)$ is defined by
\begin{align}
d\Sigma_{l}(\tau):=\{\prod_{j=l+1}^{k-1}\dd\sigma_{j}\}\{\tilde{w}(v_{l})I(t_{l},\tau)\dd\sigma_{l}\}
\{\prod_{j=1}^{l-1}I(t_{j},t_{j+1})\dd\sigma_{j}\}\nonumber.
\end{align}
\end{lemma}

The same as Lemma \ref{lemS3.1}, we have
\begin{lemma}
Let $(\eta,\zeta)$ belong to
$$\{\zeta=2, 0\leq\eta<1/2\}\cup\{0\leq \zeta<2,\eta\geq 0\}.
$$
For $T_0$ sufficiently large, there exist constants $\hat{C}_3$ and $\hat{C}_4$ independent of $T_0$ such that for $k=\hat{C}_3T_0^{\frac54}$ and $(t,x,v)\in[s,s+T_0]\times\bar{\Omega}\times\mathbb{R}^3$, it holds that
\begin{align}\label{4.1.5-1}
\int_{\Pi _{j=1}^{k-1}\hat{\mathcal{V}}_{j}} \Fi_{\{{t}_k>s\}}~  \Pi _{j=1}^{k-1} e^{\eta|v_j|^\zeta} \dd{\sigma} _{j}\leq \left(\frac12\right)^{\hat{C}_4T_0^{\frac54}}.
\end{align}
\end{lemma}

The following proposition is to clarify the solvability of the linear  problem \eqref{4.1.0}.

\begin{proposition}\label{prop4.1}
Let $-3<\gamma<0$, $\beta>3+|\ka|,$ and $(\varpi,\zeta)$ satisfy \eqref{index}. Assume that
\begin{align*}
\int_{\Omega}\int_{\mathbb{R}^3} f_0(x,v)\mu^{\frac{1}{2}}(v)\dd x\dd v=\int_{\Omega} \int_{\mathbb{R}^3} g(t,x,v)\mu^{\frac{1}{2}}(v)\dd x\dd v=0,
\end{align*}
and also assume  that 
$$\|wf_0\|_{L^{\infty}}+\sup_{s\geq 0}e^{\lambda_0 s^\alpha}\|\nu^{-1}wg(s)\|_{L^{\infty}}<\infty,
$$
where $\la_0>0$ is a small constant to be chosen in the proof.
Then the linear IBVP problem \eqref{4.1.0} admits a unique solution $f(t,x,v)$ satisfying
\begin{align}\label{4.1.3}
\sup_{0\leq s\leq t}e^{\lambda_0 s^\alpha}\{\|wf(s)\|_{L^{\infty}}+|wf(s)|_{L^{\infty}(\g)}\}\leq C\|wf_0\|_{L^{\infty}}+C\sup_{0\leq s\leq t}e^{\lambda_0 s^\alpha}\|\nu^{-1}wg(s)\|_{L^{\infty}},
\end{align}
for any $t\geq 0$.
Moreover, if $\Omega$ is convex, $f_0(x,v)$ is continuous except on $\g_0$, $g$ is continuous in $(0,\infty)\times\Omega\times \mathbb{R}^3$,
\begin{align*}
f_0|_{\gamma_-}=P_{\g}f_0+\f{\mu_{\theta}-\mu}{\sqrt{\mu}}\int_{\{n(x)\cdot v'>0\}}f_0\sqrt{\mu}|n(x)\cdot v'|\dd v',
\end{align*}
and $\theta(x) $ is continuous on $\partial \Omega$, then $f(t,x,v)$ is continuous over $[0,\infty)\times \{\bar{\Omega}\times \mathbb{R}^{3}\setminus\g_0\}$.
\end{proposition}

The proof of Proposition \ref{prop4.1} will be given after we prepare two lemmas.

\begin{lemma}\label{lm4.2} Let $h(t,x,v)$ be the $L^\infty$-solution of the linear problem \eqref{4.1.1}. Then for any $s\geq 0$, for any $s\leq t \leq s+T_0$ with $T_0>0$ is sufficiently large, and for almost everywhere $(x,v)\in \bar{\Omega}\times \mathbb{R}^3\setminus\g_0$, it holds that
\begin{align}\label{4.1.6}
|h(t,x,v)|\leq& CT_0^{5/2}e^{-\lambda_1(t-s)^{\alpha}}\|h(s)\|_{L^\infty}\nonumber\\
&+CT_{0}^{5/2}e^{-\tilde{\lambda}t^{\alpha}}\{m^{3+\ka}+\delta+2^{-T_0}+\f{1}{N}\}\cdot\sup_{s\leq \tau\leq t}e^{\tilde{\lambda}\tau^\alpha}\{\|h(\tau)\|_{L^\infty}+|h(\tau)|_{L^\infty(\g)}\}\nonumber\\
&+C_{N, T_0,m}e^{-\tilde{\lambda}t^{\alpha}}\sup_{s\leq \tau\leq t}\|e^{\tilde{\lambda}\tau^{\alpha}}f(\tau)\|_{L^2}+CT_0^{5/2}e^{-\tilde{\lambda}t^\alpha}\sup_{s\leq \tau\leq t}\|\nu^{-1}e^{\tilde{\lambda}\tau^\alpha}wg(\tau)\|_{L^\infty},
\end{align}
where $\la_1>0$ is a generic constant given in \eqref{4.1.11-1}, and $0< \tilde{\lambda}<\lambda_1$ is an arbitrary constant to be chosen later. Here  $m>0$ can be chosen arbitrarily small and $N$ can be chosen arbitrarily large.
\end{lemma}

\begin{proof}
We take $T_0$ sufficiently large and $k=\hat{C}_3T_0^{5/4}$ such that \eqref{4.1.5-1} holds for $\eta=\f5{16}$. Recall ${\nu}_*(x,v)$ defined in \eqref{4.4}. Then it holds from \eqref{S1.2} that
\begin{align}
\f{\nu(v)}{2}\leq {\nu}_*(x,v)\leq \f{3\nu(v)}{2},\nonumber
\end{align}
provided that $\delta>0$ is suitably small. Hence it holds that
\begin{align*}
e^{-\f{3\nu(v)}{2}(t-s)}\leq I(t,s)\leq  e^{-\f{\nu(v)}{2}(t-s)},
\end{align*}
for any $0\leq s\leq t$. We first estimate $\CI_1$. On one hand, if $s\leq t\leq s+1$, it obviously holds that
\begin{align}\label{4.1.8-1}|\CI_1|\leq \|h(s)\|_{L^\infty}\leq e^{\tilde{\lambda}}e^{-\tilde{\lambda}(t-s)}\|h(s)\|_{L^\infty},
\end{align}
for any $\tilde{\lambda}>0$. On the other hand, if $t>s+1$, we note that
$$
0\leq t_{\mathbf{b}}(x,v)\leq \f{d_\Omega}{|v|},
$$
where $d_\Omega:=\sup_{x,y\in\Omega}|x-y|$ is the diameter of the bounded domain $\Omega$. Then for $|v|>d_{\Omega}$, it holds that $$t_1-s=t-t_{\mathbf{b}}(x,v)-s
>0.$$
In other words, $\CI_1$  appears only when the particle velocity $|v|$ is not greater than $d_\Omega$, so that we have
\begin{align}\label{4.1.8}
|\CI_1|\leq Ce^{-\f{\nu(v)(t-s)}{2}}\Fi_{\{|v|\leq d_\Omega\}}\|h(s)\|_{L^\infty}\leq C_{\Omega}e^{-\bar{\nu}_0(t-s)}\|h(s)\|_{L^\infty},
\end{align}
where $$\bar{\nu}_0:=\f{\inf_{|v|\leq d_{\Omega}}\nu(v)}{2}>0$$ depends only on $d_\Omega$. Collecting the estimates \eqref{4.1.8-1} and \eqref{4.1.8} on these two cases, we have
\begin{align}\label{4.1.8-2}
|\CI_1|\leq C_{\Omega}e^{-\bar{\nu}_0(t-s)}\|h(s)\|_{L^\infty}.
\end{align}
For $\CI_4$, we split the velocity to estimate it as
\begin{align}
|\CI_4|&\leq\int_{\max\{t_1,s\}}^te^{-\f{\nu(v)(t-\tau)}{2}}\nu(v)\cdot\{\Fi_{\{|v|\leq d_\Omega\}}+\Fi_{\{|v|> d_{\Omega}\}}\}\cdot\|\nu^{-1}wg(\tau)\|_{L^\infty}\dd\tau\nonumber.
\end{align}
For $|v|>d_\Omega$, we have $\max\{t_1,s\}\geq t_1\geq t-t_{\mathbf{b}}(x,v)\geq t-1$, which implies, for any $\tilde{\lambda}>0$. that
$$\begin{aligned}&\int_{\max\{t_1,s\}}^te^{-\f{\nu(v)(t-\tau)}{2}}\nu(v)\Fi_{\{|v|> d_\Omega\}}\|\nu^{-1}wg(\tau)\|_{L^\infty}\\
&\quad\leq C\int_{\max\{t_1,s\}}^t\Fi_{\{t-1\leq \tau\leq t\}}\|\nu^{-1}wg(\tau)\|_{L^\infty}\\
&\quad\leq e^{\tilde{\lambda}}\int_{\max\{t_1,s\}}^te^{-\tilde{\lambda}(t-\tau)}\Fi_{\{t-1\leq \tau\leq t\}}\|\nu^{-1}wg(\tau)\|_{L^\infty}\dd\tau\\
&\quad\leq C_{\tilde{\lambda},\alpha}e^{-\tilde{\lambda}t^\alpha}\cdot\sup_{s\leq \tau\leq t}\|e^{\tilde{\lambda}\tau^\alpha}\nu^{-1}wg(\tau)\|_{L^\infty}.
\end{aligned}
$$
For $|v|\leq d_\Omega$, we have
\begin{multline*}
\int_{\max\{t_1,s\}}^te^{-\f{\nu(v)(t-\tau)}{2}}\nu(v)\Fi_{\{|v|\leq d_{\Omega}\}}\|\nu^{-1}wg(\tau)\|_{L^\infty}\dd\tau\\
\leq C_{\Omega}\int_{\max\{t_1,s\}}^t e^{-\bar{\nu}_0(t-\tau)}\|\nu^{-1}wg(\tau)\|_{L^\infty}\dd\tau\\
\leq C_{\Omega,\tilde{\lambda},\alpha}e^{-\tilde{\lambda}t^\alpha}\cdot\sup_{s\leq \tau\leq t}\|e^{\tilde{\lambda}\tau^\alpha}\nu^{-1}wg(\tau)\|_{L^\infty}.
\end{multline*}
Conbining these two estimates, we have
\begin{align}\label{4.1.9}
|\CI_4|\leq Ce^{-\tilde{\lambda}t^\alpha}\cdot\sup_{s\leq \tau\leq t}\|e^{\tilde{\lambda}\tau^\alpha}\nu^{-1}wg(\tau)\|_{L^\infty}.
\end{align}
Similarly, by using \eqref{2.31}, it holds that
\begin{align}\label{4.1.11}
|\CI_2|\leq Cm^{3+\ka}e^{-\tilde{\lambda}t^\alpha}e^{-\f{|v|^2}{16}}\sup_{s\leq \tau\leq t}\|e^{\tilde{\lambda}\tau^\alpha}h(\tau)\|_{L^\infty}.
\end{align}
For $\CI_{5}$, we take $|v_m|=\max\{|v_1|,|v_2|,\cdots,|v_l|\}$. By a direct computation, we have, for some positive constant $c>0$, that
\begin{align}\label{4.1.11-2}
&\frac{I(t,t_{1})}{\tilde{w}(v)}\bigg[\prod_{j=1}^{l-1}I(t_{j},t_{j+1})\bigg]\times
\tilde{w}(v_{l})I(t_{l},s)\nonumber\\
&\quad\leq \f{e^{-c\langle v\rangle^{\ka}(t-t_1)}}{\tilde{w}(v)}\prod_{j=1}^{l-1}e^{-c\langle v_j\rangle^{\ka}(t_j-t_{j+1})}\cdot e^{-c\langle v_l\rangle^{\ka}(t_l-s)}\tilde{w}(v_l)\nonumber\\
&\quad \leq   \f{Ce^{-c\langle v\rangle^{\ka}(t-t_1)}}{\tilde{w}(v)} e^{-c\langle v_m\rangle^{\ka}(t_1-s)}e^{\f{|v_m|^2}{4}}.
\end{align}
Here we have denoted $\langle v\rangle:=(1+|v|^2)^{1/2}$. We note the fact from Young's inequality that
\begin{align}\label{4.1.11-1}
c\langle v\rangle^{\ka}(\tau_1-\tau_2)+\f{|v|^{\zeta}}{16}\geq \lambda_1(\tau_1-\tau_2)^{\alpha},
\end{align}
for a generic constant $\lambda_1>0$ depending only on $\zeta$. Then the right-hand side of \eqref{4.1.11-2} is further bounded as
$$
\begin{aligned}
C  \f{e^{-c\langle v\rangle^{\ka}(t-t_1)}e^{-\f{|v|^\zeta}{16}}}{\tilde{w}(v)e^{-\f{|v|^\zeta}{16}}} e^{-c\langle v_m\rangle^{\ka}(t_1-s)-\f{|v_m|^\zeta}{16}}e^{\f{5|v_m|^2}{16}}\leq Ce^{-\f{|v|^2}{16}}e^{-\lambda_1(t-s)^\alpha}e^{\f{5|v_m|^2}{16}}.
\end{aligned}
$$
Here we have used the elementary fact that $x^{\alpha}+y^{\alpha}\geq (x+y)^{\alpha}$ for $x,y\geq 0$ and $0\leq \alpha\leq 1$. Therefore, it holds that
\begin{align}\label{4.1.12}
|\CI_5|\leq& \sum_{l=1}^{k-1}\sum_{m=1}^{l}Ce^{-\f{|v|^2}{16}}e^{-\lambda_1(t-s)^{\alpha}}\|h(s)\|_{L^\infty} \nonumber\\
&\times\bigg\{\int_{\prod_{j=1}^{l}\CV_{j}}\Fi_{\{t_{l+1}\leq s<t_{l}\}}\times\Fi_{\big\{|v_{m}|= \max\big[|v_{1}|,|v_{2}|,...|v_{l}|\big]\big\}}e^{\frac{5|v_{m}|^{2}}{16}}
\prod_{j=1}^{l}\dd\sigma_{j}\bigg\}\nonumber\\
\leq &Ck^2e^{-\f{|v|^2}{16}}e^{-\lambda_1(t-s)^{\alpha}}\|h(s)\|_{L^\infty}\cdot\sup_{j}\left|\int_{\CV_j}e^{\frac{5|v_{j}|^{2}}{16}}\dd\sigma_j\right|\nonumber\\
\leq& Ck^2e^{-\f{|v|^2}{16}}e^{-\lambda_1(t-s)^{\alpha}}\|h(s)\|_{L^\infty}.
\end{align}
Similarly, we have
\begin{align}\label{4.1.14}
|\CI_{8}|+|\CI_{11}|\leq Ck^2e^{-\f{|v|^2}{16}}e^{-\tilde{\lambda}t^{\alpha}}\cdot\sup_{s\leq \tau\leq t}\|e^{\tilde{\lambda}\tau^\alpha}\nu^{-1}wg(\tau)\|_{L^\infty}.
\end{align}
Recall $r$ defined in \eqref{4.1.5-2}. Similar for obtaining \eqref{4.1.12}, we have
\begin{align}\label{4.1.13}
|\CI_{12}|+|\CI_{13}|\leq C\delta k^2e^{-\f{|v|^2}{16}}e^{-\tilde{\lambda}t^{\alpha}}\cdot\sup_{s\leq \tau\leq t}|e^{\tilde{\lambda}\tau^\alpha}h(\tau)|_{L^\infty (\g_+)}.
\end{align}
By \eqref{2.31}, it holds that
\begin{align}\label{4.1.15}
|\CI_{6}|+|\CI_{9}|\leq Ck^2m^{3+\ka}e^{-\f{|v|^2}{16}}e^{-\tilde{\lambda}t^{\alpha}}\cdot\sup_{s\leq \tau\leq t}\|e^{\tilde{\lambda}\tau^\alpha}h(\tau)\|_{L^\infty}.
\end{align}
For $\CI_{14}$, we have, from \eqref{4.1.5-1}, that
\begin{align}\label{4.1.16}
|\CI_{14}|\leq &Ce^{-\f{|v|^2}{16}}e^{-\tilde{\lambda}t^{\alpha}}\cdot\sup_{s\leq \tau\leq t}|e^{\tilde{\lambda}\tau^\alpha}h(\tau)|_{L^\infty(\g_-)}\cdot\int_{\Pi _{j=1}^{k-1}\hat{\mathcal{V}}_{j}} \Fi_{\{{t}_k>s\}}~  \Pi _{j=1}^{k-1} e^{\f{5|v_j|^2}{16}} \dd{\sigma} _{j}\nonumber\\
\leq &Ce^{-\f{|v|^2}{16}}e^{-\tilde{\lambda}t^{\alpha}}\cdot\sup_{s\leq \tau\leq t}|e^{\tilde{\lambda}\tau^\alpha}h(\tau)|_{L^\infty(\g_-)}\cdot \left(\frac12\right)^{\hat{C}_4T_0^{\frac54}}.
\end{align}
Now we consider the terms involving $K^c_{w}$. Similar as in \eqref{4.1.12}, we have 
\begin{align*}
|\CI_7| &\leq  C e^{-\f{|v|^2}{16}} \sum_{l=1}^{k-1}\sum_{m=1}^{l}\int_{\Pi _{j=1}^{l-1}{\mathcal{V}}_{j}}\dd{\sigma}_{l-1}\cdots \dd{\sigma}_1 \nonumber\\
&\qquad\times \int_{\mathcal{V}_l}\int_{\mathbb{R}^3} \int_s^{{t}_l} e^{-\lambda_1(t-\tau)^{\alpha}} \Fi_{\{{t}_{l+1}\leq s<{t}_l\}} e^{\f{5|v_m|^2}{16}} |k^c_w(v_l,v') h(\tau,{x}_l-{v}_l({t}_l-\tau),v')|\dd \tau\dd v' \dd{\sigma}_l. 
\end{align*}
Then, by splitting the integral domain, we further have
\begin{align}\label{4.1.17}
|\CI_7| 
&\leq C e^{-\f{|v|^2}{16}} \sum_{l=1}^{k-1}\sum_{m=1}^{l-1}\int_{\Pi _{j=1}^{l-1}{\mathcal{V}}_{j}}e^{\f{5|v_m|^2}{16}} \dd{\sigma}_{l-1}\cdots \dd{\sigma}_1 \int_{\mathcal{V}_l\cap \{|v_l|\geq N\}}\int_{\mathbb{R}^3}\int_s^{{t}_l} \Delta\dd \tau\dd v' \dd{\sigma}_l \nonumber\\
&\quad+C e^{-\f{|v|^2}{16}} \sum_{l=1}^{k-1}\sum_{m=1}^{l-1}\int_{\Pi _{j=1}^{l-1}{\mathcal{V}}_{j}}e^{\f{5|v_m|^2}{16}}  \dd{\sigma}_{l-1}\cdots \dd{\sigma}_1 \int_{\mathcal{V}_l\cap \{|v_l|\leq N\}}\int_{\mathbb{R}^3} \int_{t_l-\f1N}^{{t}_l}  \Delta\dd \tau\dd v' \dd{\sigma}_l\nonumber\\
&\quad+ C e^{-\f{|v|^2}{16}} \sum_{l=1}^{k-1}\sum_{m=1}^{l-1}\int_{\Pi _{j=1}^{l-1}{\mathcal{V}}_{j}} e^{\f{5|v_m|^2}{16}} \dd{\sigma}_{l-1}\cdots \dd{\sigma}_1 \int_{\mathcal{V}_l\cap \{|v_l|\leq N\}}\int_{\{|v'|\geq 2N\}} \int_s^{{t}_l-\f1N}  \Delta \dd \tau\dd v' \dd{\sigma}_l\nonumber\\
&\quad+ C e^{-\f{|v|^2}{16}} \sum_{l=1}^{k-1}\sum_{m=1}^{l-1}\int_{\Pi _{j=1}^{l-1}{\mathcal{V}}_{j}} e^{\f{5|v_m|^2}{16}} \dd{\sigma}_{l-1}\cdots \dd{\sigma}_1 \int_{\mathcal{V}_l\cap \{|v_l|\leq N\}}\int_{\{|v'|\leq 2N\}} \int_s^{{t}_l-\f1N}  \Delta \dd \tau\dd v' \dd{\sigma}_l\nonumber\\
:&=Ce^{-\f{|v|^2}{16}}\sum_{l=1}^{k-1}\bigg\{\CI_{71l}+\CI_{72l}+\CI_{73l}+\CI_{74l}\bigg\},
\end{align}
where we have denoted that
\begin{align}\Delta:=e^{-\lambda_1(t-\tau)^{\alpha}} \Fi_{\{{t}_{l+1}\leq s<{t}_l\}}e^{\f{5|v_l|^2}{16}}|k^c_w(v_l,v') h(\tau,{x}_l-{v}_l({t}_l-\tau),v')|.\nonumber
\end{align}
For $\CI_{71l}$, we use \eqref{2.40-1} to obtain, for any $0<\tilde{\lambda}<\lambda_1$, that
\begin{align}\label{4.1.18}
\CI_{71l}&\leq Ck e^{-\tilde{\lambda}t^{\alpha}}\cdot\sup_{s\leq \tau\leq t}\|e^{\tilde{\lambda}\tau^\alpha}h(\tau)\|_{L^\infty}\cdot\sup_{j=0,\cdots,l-1}\left|\int_{\CV_j}e^{\f{5|v_j|^2}{16}}\dd\sigma_j\right|\cdot \left|\int_{|v_l|\geq N}e^{-\f{|v_l|^2}{16}}\dd v_l\right|\nonumber\\
&\leq Cke^{-\frac{N^2}{32}} e^{-\tilde{\lambda}t^{\alpha}}\cdot\sup_{s\leq \tau\leq t}\|e^{\tilde{\lambda}\tau^\alpha}h(\tau)\|_{L^\infty}.
\end{align}
For $\CI_{72l}$, it is straightforward to see that \begin{align}\label{4.1.19}
\CI_{72l}\leq \f{Ck}{N} e^{-\tilde{\lambda}t^{\alpha}}\cdot\sup_{s\leq \tau\leq t}\|e^{\tilde{\lambda}\tau^\alpha}h(\tau)\|_{L^\infty}.
\end{align}
For $\CI_{73l}$, since $|v'-v_l|\geq N$, then by \eqref{2.40-1}, it holds that
\begin{align}\label{4.1.20}
|\CI_{73l}|\leq Cke^{-\frac{N^2}{32}} e^{-\tilde{\lambda}t^{\alpha}}\cdot\sup_{s\leq \tau\leq t}\|e^{\tilde{\lambda}\tau ^\alpha}h(\tau)\|_{L^\infty}.
\end{align}
For $\CI_{74l}$, by H\"older's inequality, it holds that
\begin{align}
 &\int_{\mathcal{V}_l\cap \{|v_l|\leq N\}}\int_{\{|v'|\leq 2N\}} \int_s^{{t}_l-\f1N}  \Delta \dd \tau\dd v' \dd{\sigma}_l\nonumber\\
 &\quad\leq C \int_s^{t_l-\f1N}e^{-\lambda_1(t-\tau)^{\alpha}}\dd\tau\bigg\{\int_{\mathcal{V}_l\cap \{|v_l|\leq N\}}\int_{|v'|\leq 2N} e^{-\frac18|v_l|^2} |k^c_w(v_l,v')|^2 \dd v' \dd v_l \bigg\}^{1/2}\nonumber\\
&\qquad\quad\times \bigg\{\int_{\mathcal{V}_l\cap \{|v_l|\leq N\}}\int_{|v'|\leq 2N} \Fi_{\{{t}_{l+1}\leq s<{t}_l\}} \left|h(\tau,{x}_l-{v_l}({t}_l-\tau),v')\right|^2 \dd v' \dd v_l \bigg\}^{1/2}\nonumber\\
&\quad\leq C_{N} m^{\ka-1}\int_s^{t_l-\f1N}e^{-\lambda_1(t-\tau)^{\alpha}}\dd\tau\nonumber\\
&\qquad\times\bigg\{\int_{\mathcal{V}_l\cap \{|v_l|\leq N\}}\int_{|v'|\leq 2N} \Fi_{\{{t}_{l+1}\leq s<{t}_l\}} \left|f(\tau,{x}_l-{v_l}({t}_l-\tau),v')\right|^2 \dd v' \dd v_l\bigg\}^{1/2}.\nonumber
\end{align}
Here we have used \eqref{2.33} for $\alpha=1$ in the last inequality. Note that $y_l:={x}_l-{v}({t}_l-\tau)\in \Omega$ for $s\leq \tau\leq t_l-\f1N$. Making change of variable $v_l\rightarrow y_l$, we obtain that
\begin{align}
\int_{\mathcal{V}_l\cap \{|v_l|\leq N\}}\int_{\mathbb{R}^3} \int_s^{{t}_l-\f1N}  \Delta \dd \tau\dd v' \dd{\sigma}_l\leq C_Nm^{\ka-1}e^{-\tilde{\lambda}t^\alpha}\cdot\sup_{s\leq \tau\leq t}e^{\tilde{\lambda}{\tau}^\alpha}\|f(\tau)\|_{L^2}.\nonumber
\end{align}
Combing this with \eqref{4.1.17}, \eqref{4.1.18}, \eqref{4.1.19}, \eqref{4.1.20}, we have
\begin{align}\label{4.1.21}
|\CI_7|\leq \f{Ck}{N}e^{-\f{|v|^2}{16}}e^{-\tilde{\lambda}t^\alpha}\cdot\sup_{s\leq \tau\leq t}\|e^{\tilde{\lambda}\tau^\alpha}h(\tau)\|_{L^\infty}+C_Nm^{\ka-1}e^{-\f{|v|^2}{16}}e^{-\tilde{\lambda}t^\alpha}\cdot\sup_{s\leq \tau\leq t}e^{\tilde{\lambda}{\tau}^\alpha}\|f(\tau)\|_{L^2}.
\end{align}
Similarly,
\begin{align}\label{4.1.22}
|\CI_{10}|\leq \f{Ck}{N}e^{-\f{|v|^2}{16}}e^{-\tilde{\lambda}t^\alpha}\cdot\sup_{s\leq \tau\leq t}\|e^{\tilde{\lambda}\tau^\alpha}h(\tau)\|_{L^\infty}+C_Nm^{\ka-1}e^{-\f{|v|^2}{16}}e^{-\tilde{\lambda}t^\alpha}\cdot\sup_{s\leq \tau\leq t}e^{\tilde{\lambda}{\tau}^\alpha}\|f(\tau)\|_{L^2}.
\end{align}
Substituting \eqref{4.1.8-2}, \eqref{4.1.9}, \eqref{4.1.11}, \eqref{4.1.12}, \eqref{4.1.13}, \eqref{4.1.14}, \eqref{4.1.15}, \eqref{4.1.16}, \eqref{4.1.21}, \eqref{4.1.22} into \eqref{4.1.5}, we have
\begin{align}\label{4.1.23}
|h(t,x,v)|\leq\int_{\max\{t_1,s\}}^t\int_{\mathbb{R}^3}I(t,\tau)\left|k_w^c(v,u)h(\tau,x-(t-\tau)v,u)\right|\dd u\dd\tau+\hat{A}(t,v),
\end{align}
where we have denoted
\begin{align*}
\hat{A}(t,v):=&Ck^2e^{-\lambda_1(t-s)^{\alpha}}\|h(s)\|_{L^\infty}\nonumber\\
&+Ck^2e^{-\tilde{\lambda}t^\alpha}e^{-\f{|v|^2}{16}}\{m^{3+\ka}+\delta+2^{-T_0}+\f{1}{N}\}\cdot\sup_{s\leq \tau\leq t}e^{\tilde{\lambda}\tau^\alpha}\{\|h(\tau)\|_{L^\infty}+|h(\tau)|_{L^\infty(\g)}\}\nonumber\\
&+C_{N,T_0,m}e^{-\f{|v|^2}{16}}e^{-\tilde{\lambda}t^\alpha}\cdot\sup_{s\leq \tau\leq t}\|e^{\tilde{\lambda}{\tau}^\alpha}f(\tau)\|_{L^2}\notag\\
&+Ck^2e^{-\tilde{\lambda}t^\alpha}\sup_{s\leq \tau\leq t}\|\nu^{-1}e^{\tilde{\lambda}\tau^\alpha}wg(\tau)\|_{L^\infty}.
\end{align*}
Denote $x':=x-(t-\tau)v$ and $t_{1}':=t_1(\tau,x',u)$. Now we use \eqref{4.1.23} for $h(\tau,x-(t-\tau)v,u)$ to evaluate
\begin{align}\label{4.1.24}
|h(t,x,v)|&\leq \hat{A}(t,v)+\int_{\max\{t_1,s\}}^t\int_{\mathbb{R}^3}I(t,\tau)\left|k_w^c(v,u)\hat{A}(\tau,u)\right|\dd u\dd\tau\nonumber\\
&\quad+\int_{s}^t\int_{s}^\tau\int_{\mathbb{R}^3}\int_{\mathbb{R}^3}\Fi_{\{\max\{t_1,s\}\leq \tau\leq t\}}\Fi_{\{\max\{t_1',s\}\leq \tau'\leq \tau\}}\nonumber\\
&\qquad\times I(t,\tau')\left|k_w^c(v,u)k_w^c(u,u')h(\tau',x'-(\tau-\tau')u,u')\right|\dd u\dd u'\dd\tau'\dd\tau\nonumber\\
&:=\hat{A}(t,v)+\hat{B}_1+\hat{B}_2.
\end{align}
Similar for obtaining \eqref{4.1.9}, we have
\begin{align}\label{4.1.25}
|\hat{B}_1|\leq &\int_{\max\{t_1,s\}}^t\int_{\mathbb{R}^3}\left\{\Fi_{\{|v|\leq d_\Omega\}}+\Fi_{\{|v|>d_{\Omega}\}}\right\}I(t,\tau)|k_w(v,u)A(\tau,u)|\dd u\dd\tau\nonumber\\
\leq& C_{\Omega}\int_{\max\{t_1,s\}}^t\int_{\mathbb{R}^3}\left\{e^{-\bar{\nu}_0(t-\tau)}\Fi_{\{t-1\leq\tau\leq t\}}+e^{-\bar{\nu}_0(t-\tau)}\Fi_{\{|v|\leq d_\Omega\}}\right\}|k_w(v,u)A(\tau,u)|\dd u\dd\tau\nonumber\\
\leq& Ck^2e^{-\lambda_1(t-s)^{\alpha}}\|h(s)\|_{L^\infty}\nonumber\\
&+Ck^2e^{-\tilde{\lambda}t^\alpha}\{m^{3+\ka}+\delta+2^{-T_0}+\f{1}{N}\}\cdot\sup_{s\leq \tau\leq t}e^{\tilde{\lambda}\tau^\alpha}\{\|h(\tau)\|_{L^\infty}+|h(\tau)|_{L^\infty(\g)}\}\nonumber\\
&+C_{N,T_0,m}e^{-\tilde{\lambda}t^\alpha}\cdot\sup_{s\leq \tau\leq t}\|e^{\tilde{\lambda}{\tau}^\alpha}f(\tau)\|_{L^2}+Ck^2e^{-\tilde{\lambda}t^\alpha}\sup_{s\leq \tau\leq t}\|\nu^{-1}e^{\tilde{\lambda}\tau^\alpha}wg(\tau)\|_{L^\infty}.
\end{align}
Finally, we estimate $\hat{B}_2$. If $|v|>N$, we have from \eqref{2.40-1} that
\begin{align}\label{4.1.26}
|\hat{B}_2|&\leq C(1+|v|)^{-2}e^{-\tilde{\lambda}t^\alpha}\sup_{s\leq \tau\leq t}\|e^{\tilde{\lambda}\tau^\alpha}h(\tau)\|_{L^\infty}\nonumber\\
&\leq \f{C}{N^2}e^{-\tilde{\lambda}t^\alpha}\sup_{s\leq \tau\leq t}\|e^{\tilde{\lambda}\tau^\alpha}h(\tau)\|_{L^\infty}.
\end{align}
If $|v|\leq N$, we denote the integrand of $B_2$
as $U(\tau',v',v'';\tau,v)$, and split the integral domain 
 with respect to $\dd\tau'\dd v''\dd v'$ into the following four parts:
\begin{align*}
\cup_{i=1}^4\mathcal{O}_i:=&\{|v'|\geq 2N\}\\
&\cup\{|v'|\leq 2N, |v''|>3N\}\\
&\cup\{|v'|\leq 2N, |v''|\leq 3N, \tau-\f1N\leq \tau'\leq \tau\}\nonumber\\
&\cup\{|v'|\leq 2N, |v''|\leq 3N, s\leq \tau'\leq \tau-\f1N\}.
\end{align*}
Over $\mathcal{O}_1$ and $\mathcal{O}_2$, we have either $|v-v'|\geq N$ or $|v'-v''|\geq N $, so that one of the following is valid:
\begin{equation}\left\{
\begin{aligned}
&|k^c_w(v,v')|\leq e^{-\f{N^2}{32}}e^{\f{|v-v'|^2}{32}}|k^{c}_w(v,v')|,\nonumber\\
&|k^c_w(v',v'')|\leq e^{-\f{N^2}{32}}e^{\f{|v'-v''|^2}{32}}|k^{c}_w(v,v')|.
\end{aligned}\right.
\end{equation}
By \eqref{2.40-1}, one has
\begin{equation}\label{4.1.27}
\int_{s}^t\int_{\mathcal{O}_1\cup\mathcal{O}_2}U(\tau',v',v'';\tau,v)\dd v''\dd\tau'\dd v'\dd \tau
\leq Ce^{-\f{N^2}{32}}e^{-\tilde{\lambda}t^\alpha}\sup_{s\leq \tau\leq t}\|e^{\tilde{\lambda}\tau^\alpha}h(\tau)\|_{L^{\infty}}.
\end{equation}
Over $\mathcal{O}_3$, it is direct to obtain
\begin{equation}\label{4.1.28}
\int_{s}^t\int_{\mathcal{O}_3}U(\tau',v',v'';\tau,v)\dd v''\dd\tau'\dd v'\dd \tau
\leq \f{C}{N}e^{-\tilde{\lambda}t^\alpha}\sup_{s\leq \tau\leq t}\|e^{\tilde{\lambda}\tau^\alpha}h(\tau)\|_{L^{\infty}}.
\end{equation}
For $\mathcal{O}_4$, it holds from Holder's inequality that
\begin{align}\label{4.1.29}
&\int_{\mathcal{O}_4}U(\tau',v',v'';\tau,v)\dd v''\dd\tau'\dd v'\nonumber\\
&\quad\leq C_Ne^{-\bar{\nu}_0(t-\tau)}\left(\int_{\mathcal{O}_4}e^{-\bar{\nu}_0(\tau-\tau')}|k_w^c(v,u)k_{w}^c(u,u')|^2\dd\tau'\dd u\dd u'\right)^{1/2}\nonumber\\
&\qquad\times\left(\int_{\mathcal{O}_4}e^{-\bar{\nu}_0(\tau-\tau')}\Fi_{\{y'\in \Omega\}}|h(\tau',y',u')|^2\dd\tau'\dd u\dd u'\right)^{1/2}\nonumber\\
&\quad\leq C_{N}m^{2(\ka-1)}e^{-\bar{\nu}_0(t-\tau)}\left(\int_{\mathcal{O}_4}e^{-\bar{\nu}_0(\tau-\tau')}\Fi_{\{y'\in \Omega\}}|f(\tau',y',u')|^2\dd\tau'\dd u\dd u'\right)^{1/2},
\end{align}
where we have denoted $y':=x'-(\tau-\tau')u.$ Making change of variable $u\rightarrow y'$ and noting that the Jacobian $\left|\f{\dd y'}{\dd u}\right|\geq \f{1}{N^3}>0$ for $s\leq \tau'<\tau-\f1N$, the right-hand side of \eqref{4.1.29} is bounded by
$$ C_{N}m^{2(\ka-1)}e^{-\tilde{\lambda}t^\alpha}e^{-\f{\bar{\nu}_0(t-\tau)}{2}}\cdot\sup_{s\leq \tau\leq t}\|e^{\tilde{\lambda}{\tau}^\alpha}f(\tau)\|_{L^2},
$$
which implies that
\begin{align}
\int_s^t\int_{\mathcal{O}_4}U(\tau',v',v'';\tau,v)\dd v''\dd\tau'\dd v'\dd\tau\leq C_{N}m^{2(\ka-1)}e^{-\tilde{\lambda}t^\alpha}\cdot\sup_{s\leq \tau\leq t}\|e^{\tilde{\lambda}{\tau}^\alpha}f(\tau)\|_{L^2}.\nonumber
\end{align}
Combining this with \eqref{4.1.26}, \eqref{4.1.27}, \eqref{4.1.28} yields that
\begin{align}\label{4.1.30}
|\hat{B}_2|\leq  \f{C}{N}e^{-\tilde{\lambda}t^\alpha}\sup_{s\leq \tau\leq t}\|e^{\tilde{\lambda}\tau^\alpha}h(\tau)\|_{L^\infty}+C_{N}m^{2(\ka-1)}e^{-\tilde{\lambda}t^\alpha}\cdot\sup_{s\leq \tau\leq t}\|e^{\tilde{\lambda}{\tau}^\alpha}f(\tau)\|_{L^2}.
\end{align}
Substituting \eqref{4.1.25} and \eqref{4.1.30} into \eqref{4.1.24}, one has \eqref{4.1.6}. The proof of Lemma \ref{lm4.2} is complete.
\end{proof}

The below lemma gives the $L^2$-decay of the solution.

\begin{lemma}\label{lm4.3}
If
\begin{align}\label{4.1.31}
\int_{\Omega}\int_{\mathbb{R}^3}f_0(x,v)\mu^{\frac{1}{2}}(v)\dd x\dd v=\int_{\Omega}\int_{\mathbb{R}^3}g(t,x,v)\mu^{\frac{1}{2}}(v)\dd x\dd v=0,
\end{align}
and $|\theta-1|_{L^\infty(\pa\Omega)}$ is sufficiently small, then there exists a constant $\lambda_2>0$ such that for any $t\geq 0$,
\begin{align}\label{4.1.32}
\|f(t)\|_{L^2}\leq Ce^{-\lambda_2t^{\alpha}}\|wf_0\|_{L^\infty}+\int_0^te^{-\lambda_2(t-s)^{\alpha}}\|\nu^{-1}wg(s)\|_{L^\infty}\dd s.
\end{align}
\end{lemma}

\begin{proof}
We first consider the case $g\equiv 0$. Multiplying  both sides of \eqref{4.1.0} by $f$, we have
\begin{equation}\label{4.1.33}
\f{1}{2}\|f(t)\|_{L^2}^2+\f12\int_0^t|f(s)|_{L^2(\g_+)}^2\dd s+\int_0^t\langle Lf(s),f(s)\rangle\dd s
=\f{1}{2}\|f_0\|_{L^2}^2+\f12\int_0^t|P_{\g}f+r|_{L^2(\g_-)}^2\dd s,
\end{equation}
where $r$ is defined in \eqref{4.1.5-2}. By the coercivity estimate \eqref{1.10-2}, it holds that
\begin{align}\label{4.1.34}
\int_0^t\langle Lf(s),f(s)\rangle\dd s\geq c_0\int_0^t\|{\nu^{1/2}}(I-P)f(s)\|_{L^2}^2\dd s.
\end{align}
Notice that $P_{\g}r\equiv0$. Therefore, it follows that
\begin{align}\label{4.1.35}
\f12\int_0^t|P_{\g}f+r|_{L^2(\g_-)}^2\dd s&= \f12\int_{0}^t|P_{\g}f(s)|_{L^2(\g_+)}^2\dd s+\f12\int_{0}^t|r(s)|_{L^2(\g_-)}^2\dd s\nonumber\\
&\leq \f12\int_{0}^t|P_{\g}f(s)|_{L^2(\g_+)}^2\dd s+C\delta\int_{0}^t|f(s)|_{L^2(\g_+)}^2\dd s.
\end{align}
To estimate $|P_{\g}f|_{L^2(\g_+)}$, recall the cutoff function $\Fi_{\g_+^{\vep}}$ with respect to the near grazing set $\g_{+}^\vep$ defined in \eqref{S2.1}. Then we have
\begin{align}\label{4.1.36}
|P_{\g}f(s)|_{L^2(\g_+)}^2&=\int_{\g_-}\mu(v)|n(x)\cdot v|\dd \gamma\times\left(\int_{n(x)\cdot v'>0}f(s)\{\Fi_{\g_{+}^{\vep}}+\Fi_{\g_+^\vep\setminus\g_{+}^{\vep}}\}\sqrt{\mu}|n(x)\cdot v'|\dd v'\right)^2\nonumber\\
&\leq C\vep |f(s)|_{L^2(\g_+)}^2+C|e^{-\f{|v|^2}{16}}f(s)\Fi_{\g_{+}\setminus\g_+^\vep}|_{L^2(\g_+)}^2.
\end{align}
Notice that
$$
{\f12}(\pa_t+v\cdot\nabla_x)e^{-\f{|v|^2}{8}}f^2=e^{-\f{|v|^2}{8}}fLf,
$$
which implies that
$$\int_0^t\|(\pa_t+v\cdot\nabla_x)e^{-\f{|v|^2}{8}}f^2(s)\|_{L^1}\dd s\leq C\int_0^t\|e^{-\f{|v|^2}{16}}f(s)\|_{L^2}^2\dd s.
$$
Therefore, by the trace estimate \eqref{S2.3}, we have
\begin{align}
\int_0^t|e^{-\f{|v|^2}{16}}f(s)\Fi_{\g_{+}\setminus\g_+^\vep}|_{L^2(\g_+)}^2\dd s&=\int_0^t|e^{-\f{|v|^2}{8}}f^2(s)\Fi_{\g_{+}\setminus\g_+^\vep}|_{L^1(\g_+)}\dd s\nonumber\\
&\leq C_\vep\|e^{-\f{|v|^2}{16}}f_0\|_{L^2}^2+C_\vep\int_0^t\|e^{-\f{|v|^2}{16}}f(s)\|_{L^2}^2\dd s
\nonumber\\
&\leq C_{\vep}\|f_0\|_{L^2}^2+C_{\vep}\int_0^t\|\nu^{1/2}f(s)\|_{L^2}^2\dd s\nonumber.
\end{align}
Combining this with \eqref{4.1.36}, we have
\begin{align}\label{4.1.37}
&\int_0^t|f(s)|_{L^{2}(\g_+)}^2\dd s\nonumber\\
&\quad=\int_0^t|P_{\g}f(s)|_{L^{2}(\g_+)}^2\dd s+\int_0^t|(I-P_{\g})f(s)|_{L^{2}(\g_+)}^2\dd s\nonumber\\
&\quad\leq C\int_0^t|(I-P_{\g})f(s)|_{L^{2}(\g_+)}^2\dd s+C\vep\int_0^t|f(s)|_{L^{2}(\g_+)}^2\dd s+C_\vep\|f_0\|_{L^2}^2+C_{\vep}\int_0^t\|\nu^{1/2}f(s)\|_{L^2}^2\dd s\nonumber\\
&\quad\leq  C\int_0^t|(I-P_{\g})f(s)|_{L^{2}(\g_+)}^2\dd s+C\|f_0\|_{L^2}^2+C\int_0^t\|\nu^{1/2}f(s)\|_{L^2}^2\dd s.
\end{align}
Here we have taken $\vep>0$ suitably small in the last inequality. For the macroscopic part $Pf$, we multiply $\sqrt{\mu}$ to both sides of the first equation in \eqref{4.1.0} and use \eqref{4.1.31} to get
\begin{align}
\int_{\Omega}\int_{\mathbb{R}^3}f(t,x,v)\mu^{\frac{1}{2}}(v)\dd x\dd v=\int_{\Omega}\int_{\mathbb{R}^3}f_0(x,v)\mu^{\frac{1}{2}}(v)\dd x\dd v=0.\nonumber
\end{align}
Then similar as \cite[Lemma 6.1]{EGKM}, there exists a functional $\mathfrak{e}_{f}(t)$ with $|\mathfrak{e}_{f}(t)|\lesssim\|f(t)\|_{L^2}^2$ such that
\begin{align}
\int_0^t\|\nu^{1/2}Pf(s)\|_{L^2}^2\dd s &\lesssim\bigg(\mathfrak{e}_{f}(t)-\mathfrak{e}_{f}(0)\bigg)+\int_0^t\|\nu^{1/2}(I-P)f(s)\|_{L^2}^2\dd s\nonumber\\
&\qquad
+\int_0^t|r(s)|^2_{L^2(\g_-)}\dd s+\int_0^t|(I-P_{\g})f(s)|_{L^2(\g_+)}^2\dd s.\nonumber\\
&\lesssim\bigg(\mathfrak{e}_{f}(t)-\mathfrak{e}_{f}(0)\bigg)+\int_0^t\|\nu^{1/2}(I-P)f(s)\|_{L^2}^2\dd s\nonumber\\
&\qquad
+\delta\int_0^t|f(s)|^2_{L^2(\g_+)}\dd s+\int_0^t|(I-P_{\g})f(s)|_{L^2(\g_+)}^2\dd s.\nonumber
\end{align}
Suitably combining the estimate above with \eqref{4.1.33}, \eqref{4.1.34}, \eqref{4.1.35} and \eqref{4.1.37} and taking $\delta>0$ suitably small, we have
\begin{align} \label{4.1.38}
\|f(t)\|_{L^2}^2+\int_{0}^t\|\nu^{1/2}f(s)\|_{L^2}^2\dd s+\int_{0}^t|f(s)|_{L^2(\g_+)}^2\dd s\leq C\|f_0\|_{L^2}^2.
\end{align}
Next we need to obtain the weighted $L^2$ estimate in order to obtain $L^2$ decay of $f$. Multiplying $e^{\f{\varpi|v|^\zeta}{2}}\Red{f}
$ to both sides of the first equation in \eqref{4.1.0}, we have
\begin{align}\label{4.1.39}
&\f{1}{2}\|e^{\f{\varpi|\cdot|^\zeta}{4}}f(t)\|_{L^2}^2+\f{1}{2}\int_0^t|e^{\f{\varpi|\cdot|^\zeta}{4}}f(s)|_{L^2(\g_+)}^2\dd s
+\int_0^t\|\nu^{1/2}e^{\f{\varpi|\cdot|^\zeta}{4}}f(s)\|_{L^2}^2\dd s\nonumber\\
&\quad=\f{1}{2}\|e^{\f{\varpi|\cdot|^\zeta}{4}}f_0\|_{L^2}^2+\f{1}{2}\int_0^t|e^{\f{\varpi|\cdot|^\zeta}{4}}f(s)|_{L^2(\g_-)}^2
+\int_0^t\langle Kf(s), e^{\f{\varpi|\cdot|^\zeta}{2}}f(s)\rangle\dd s.
\end{align}
A direct computation shows that
\begin{align}\label{4.1.40}
\int_0^t|e^{\f{\varpi|\cdot|^\zeta}{4}}f(s)|_{L^2(\g_-)}^2&\leq C\int_0^t|e^{\f{\varpi|\cdot|^\zeta}{4}}P_{\g}f(s)|_{L^2(\g_-)}^2\dd s+C\int_0^t|e^{\f{\varpi|\cdot|^\zeta}{4}}r(s)|_{L^2(\g_-)}^2\dd s\nonumber\\
&\leq C\int_0^t|f(s)|_{L^2(\g_+)}^2\dd s.
\end{align}
As for the last term on the right-hand side of \eqref{4.1.39}, we use \eqref{2.28-1} to obtain
\begin{align}\label{4.1.41}
\int_0^t\left|\langle Kf(s), e^{\f{\varpi|\cdot|^\zeta}{2}}f(s)\rangle\right|\dd s\leq \eta\int_0^t\|\nu^{1/2}e^{\f{\varpi|\cdot|^\zeta}{4}}f(s)\|_{L^2}^2\dd s+C_\eta \int_0^t\|\nu^{1/2}f(s)\|_{L^2}^2\dd s.
\end{align}
Therefore, suitably combining \eqref{4.1.38}, \eqref{4.1.39}, \eqref{4.1.40} and \eqref{4.1.41} and taking $\eta>0$ suitably small, we have
\begin{align*}
\|e^{\f{\varpi|\cdot|^\zeta}{4}}f(t)\|_{L^2}^2+\int_{0}^t\|e^{\f{\varpi|\cdot|^\zeta}{4}}\nu^{1/2}f(s)\|_{L^2}^2\dd s+\int_{0}^t|e^{\f{\varpi|\cdot|^\zeta}{4}}f(s)|_{L^2(\g_+)}^2\dd s\leq C\|e^{\f{\varpi|\cdot|^\zeta}{4}}f_0\|_{L^2}^2.
\end{align*}
Now we are ready for obtaining $L^2$ decay of $f$. Let $\tilde{f}=e^{\lambda' (1+t)^{\alpha}}f$, with $\lambda'>0$ is a suitably small constant to be determined later. Then applying the same energy estimate for obtaining \eqref{4.1.38}, we have
\begin{align}\label{4.1.43}
\|\tilde{f}(t)\|_{L^2}^2+\int_0^t\|\nu^{1/2}\tilde{f}(s)\|_{L^2}^2\dd s\leq C\|f_0\|_{L^2}^2+C\lambda'\int_0^t(s+1)^{\alpha-1}\|\tilde{f}(s)\|_{L^2}^2\dd s.
\end{align}
To estimate the last term, we split the $v$-integration domain into
$$\mathbb{R}^3_v=\{1+|v|\geq (1+t)^{\f{1}{\zeta+|\ka|}}\}\cup\{1+|v|< (1+t)^{\f{1}{\zeta+|\ka|}}\}=\CM(t)\cup\CM^c(t).$$
One one hand, we have
$$
\begin{aligned}\int_0^t(s+1)^{\alpha-1}\|\tilde{f}\Fi_{\CM(s)}(s)\|_{L^2}^2\dd s&\leq C\int_0^t\left\|e^{\lambda'(1+|\cdot|)^{\zeta}}f(s)\right\|_{L^2}^2\dd s\leq C\int_0^t\|\nu^{1/2}e^{\f{\varpi|\cdot|^\zeta}{4}}f(s)\|_{L^2}^2\dd s\\
&\leq C\|e^{\f{\varpi|\cdot|^\zeta}{4}}f_0\|_{L^2}^2,
\end{aligned}
$$
by taking $\lambda'>0$ suitably small. On the other hand, in $\CM^c(s)$, we have
$$(1+s)^{\alpha-1}\leq {(1+|v|)^{(\alpha-1)\cdot(\zeta+|\ka|)}\Fi_{\CM^c(s)}}\leq C\nu(v)\Fi_{\CM^c(s)}.
$$
Therefore, we have
$$\int_0^t(s+1)^{\alpha-1}\|\tilde{f}\Fi_{\CM^c(s)}(s)\|_{L^2}^2\dd s\leq C\int_0^t\|\nu^{1/2}\tilde{f}(s)\|_{L^2}^2\dd s.
$$
Combining these estimates with \eqref{4.1.43} and taking $\lambda'>0$ suitably small, we have
\begin{align}\label{4.1.44}
\|\tilde{f}(t)\|_{L^2}^2+\int_0^t\|\nu^{1/2}\tilde{f}(s)\|_{L^2}^2\dd s&\leq C\|e^{\f{\varpi|\cdot|^\zeta}{4}}f_0\|_{L^2}^2+C\lambda'\int_0^t\|\nu^{1/2}\tilde{f}(s)\|_{L^2}^2\dd s\nonumber\\
&\leq C\|e^{\f{\varpi|\cdot|^\zeta}{4}}f_0\|_{L^2}^2.
\end{align}
Then \eqref{4.1.32} for $g\equiv0$ naturally follows. We denote $G(t)$ as the solution operator to the linear homogeneous problem \eqref{4.1.0} with $g\equiv0$. Then for non-trivial $g$, from Duhamel's formula, it holds that
\begin{align}\label{4.1.45}
\|f(t)\|_{L^2}&\leq \|G(t)f_0\|_{L^2}+\int_0^t\|G(t-s)g(s)\|_{L^2}\dd s\nonumber\\
&\leq Ce^{-\lambda't^{\alpha}}\|e^{\f{\varpi|\cdot|^\zeta}{4}}f_0\|_{L^2}+C\int_0^te^{-\lambda'(t-s)^{\alpha}}\|e^{\f{\varpi|\cdot|^\zeta}{4}}g(s)\|_{L^2}\dd s\nonumber\\
&\leq Ce^{-\lambda't^{\alpha}}\|wf_0\|_{L^\infty}+\int_0^te^{-\lambda'(t-s)^{\alpha}}\|\nu^{-1}wg(s)\|_{L^\infty}\dd s.
\end{align}
Here we have used \eqref{4.1.44} in the second inequality. Then \eqref{4.1.32} follows from \eqref{4.1.45} by taking $\lambda_2=\lambda'$. The proof of Lemma \ref{lm4.3} is complete.
\end{proof}

\noindent{\it Proof of Proposition \ref{prop4.1}:} The local existence and uniqueness of solutions to the linear inhomogeneous problem \eqref{4.1.1} can be obtained in a similar way as in Proposition \ref{prop6.1}. We omit the details for brevity. In what follows we will show the decay estimate \eqref{4.1.3}.  Recall the finite-time estimate \eqref{4.1.6}. We define $
\la_0=\min\{\f{\lambda_1}{4},\f{\lambda_2}{4}\}$, and
$$m=\left(\f{\eta}{2C}\right)^{\f{1}{3+\ka}}T_0^{-\f{5}{2(3+\ka)}}$$
with $\eta>0$ suitably small to be determined later. Then we choose $T_0$ suitably large and $\delta>0$ suitably small, and also take $N$ suitably large, such that
$$
CT_0^{5/2}\{m^{3+\ka}+\delta+2^{-T_0}+\f{1}{N}\}\leq \eta,\quad 
CT_0^{5/2}e^{-\f{\lambda_1}{2}T_0^\alpha}\leq 1.
$$
Then it holds from \eqref{4.1.6} with the choice of $\tilde{\lambda}=\la_0$ that for any $s\geq 0$ and any $t\in [s,s+T_0]$,
\begin{equation}\label{4.1.46}
\|h(t)\|_{L^\infty}+|h(t)|_{L^\infty(\g)}\leq e^{-\f{\lambda_1}{2}(t-s)^\alpha}\|h(s)\|_{L^\infty}+e^{-{\lambda_0}t^\alpha}\CD(t,s),
%
\end{equation}
where  we have defined
\begin{align}\label{ads4.defdts}
\CD(t,s)& =\eta \sup_{s\leq \tau\leq t}e^{{\lambda_0}\tau^\alpha}\{\|h(\tau)\|_{L^\infty}+|h(\tau)|_{L^\infty(\g)}\}\nonumber\\
&\quad+C\sup_{s\leq \tau\leq t}\|e^{{\lambda_0}\tau^\alpha}f(\tau)\|_{L^2}
+C\sup_{s\leq \tau\leq t}\|\nu^{-1}e^{{\lambda_0}\tau^\alpha}wg(\tau)\|_{L^\infty}.
\end{align}
For any $t>0$, there exists a positive integer $n\geq 1$, such that $nT_0\leq t<(n+1)T_0$. Then applying \eqref{4.1.46} to $[0,T_0], [T_0,2T_0], \cdots, [(n-1)T_0,nT_0]$ inductively, we have
\begin{align}\label{4.1.47}
\|h(nT_0)\|_{L^\infty}&\leq e^{-\f{\lambda_1}{2}T_0^\alpha}\|h(n-1)T_0\|_{L^\infty}+e^{-{\lambda_0}(nT_0)^\alpha}D(nT_0,[n-1]T_0)\nonumber\\
&\leq e^{-\lambda_1T_0^\alpha}\|h(n-2)T_0\|_{L^\infty}+e^{-{\lambda_0}(nT_0)^\alpha}D(nT_0,[n-1]T_0)\nonumber\\
&\quad+e^{-\f{\lambda_1T_0^\alpha}{2}-{\lambda_0}([n-1]T_0)^\alpha}
D([n-1]T_0,[n-2]T_0)\nonumber\\
&\leq  e^{-\lambda_1T_0^\alpha}\|h(n-2)T_0\|_{L^\infty}+e^{-{\lambda_0}(nT_0)^\alpha}
\left\{1+e^{-\f{\lambda_1T_0^\alpha}{4}}\right\}D(nT_0,[n-2]T_0)\nonumber\\
&\leq \cdots\nonumber\\
&\leq e^{-\f{\lambda_1nT_0^\alpha}{2}}\|h_0\|_{L^\infty}+e^{-{\lambda_0}(nT_0)^\alpha}\left\{1+e^{-\f{\lambda_1T_0^\alpha}{4}}+
\cdots+e^{-\f{(n-1)\lambda_1T_0^\alpha}{4}}\right\}\CD(nT_0,0)\nonumber\\
&\leq Ce^{-\f{\lambda_1(nT_0)^\alpha}{2}}\|h_0\|_{L^\infty}+Ce^{-{\lambda_0}(nT_0)^\alpha}\CD(nT_0,0),
\end{align}
where in the third inequality we have used $0< {\lambda_0}\leq \f{\lambda_1}{4}$. Here we also have used the elementary fact that $x^{\alpha}+y^\alpha\geq (x+y)^\alpha$ for $x$, $y\geq 0$ and $0\leq \alpha\leq 1$. Finally applying \eqref{4.1.46} in $[nT_0, (n+1)T_0]$ and using \eqref{4.1.47}, we have
\begin{align}
&\|h(t)\|_{L^\infty}+|h(t)|_{L^\infty(\g)}\notag\\
&\leq e^{-\f{\lambda_1(t-nT_0)^\alpha}{2}}\|h(nT_0)\|_{L^\infty}+e^{-{\lambda_0}t^\alpha}\CD(t,nT_0)\nonumber\\
&\leq e^{-{\lambda_0}t^\alpha}\CD(t,nT_0)+ e^{-\f{\lambda_1(t-nT_0)^\alpha}{2}}\left\{Ce^{-\f{\lambda_1(nT_0)^\alpha}{2}}\|h_0\|_{L^\infty}+Ce^{-{\lambda_0}(nT_0)^\alpha}D(nT_0,0)\right\}\nonumber\\
&\leq Ce^{-\f{\lambda_1t^\alpha}{2}}\|h_0\|_{L^\infty}+Ce^{-{\lambda_0}t^\alpha}\CD(t,0).\label{ads4.p1}
\end{align}
Recall \eqref{ads4.defdts}. Let $\eta>0$ be suitably small, then \eqref{4.1.3} follows from \eqref{ads4.p1} and \eqref{4.1.32}. Therefore, the proof of Proposition \ref{prop4.1} is complete.\qed

\subsection{Proof of Theorem \ref{thm1.2}.} The local existence and uniqueness of the solution to nonlinear problem \eqref{4.1} is provided in Proposition \ref{prop6.1}. In what follows we will show \eqref{D1.3}. Notice that for any $t>0$, it holds that
$$\int_{\Omega}\int_{\mathbb{R}^3}L_{\sqrt{\mu}f_*}f\dd x\dd v=\int_{\Omega}\int_{\mathbb{R}^3}\Gamma(f,f)\dd x\dd v\equiv0.
$$
Then applying the linear theory Proposition \ref{prop4.1} to $f$, we have
\begin{align}\label{4.2.2}
&\sup_{0\leq s\leq t}e^{\lambda_0 s^\alpha}\{\|wf(s)\|_{L^\infty}+|wf(s)|_{L^\infty(\g)}\}\nonumber\\
&\quad\leq C\|wf_0\|_{L^\infty}+C\sup_{0\leq s\leq t}e^{\lambda_0 s^\alpha}\{\|\nu^{-1}wL_{\sqrt{\mu}f_*}f(s)\|_{L^\infty}+\|\nu^{-1}w\Gamma(f,f)(s)\|_{L^\infty}\}\nonumber\\
&\quad\leq C\|wf_0\|_{L^\infty}+C\{\delta+\sup_{0\leq s\leq t}\|wf(s)\|_{L^\infty}\}\cdot\sup_{0\leq s\leq t}e^{\lambda_0 s^\alpha}\|wf(s)\|_{L^\infty},
\end{align}
where we have used the nonlinear estimate \eqref{2.43}. 
Now we make the a priori assumption that
$$
\sup_{0\leq s\leq t}e^{\lambda_0 s^\alpha}\|wf(s)\|_{L^\infty}\leq 2C\|wf_0\|_{L^\infty}.
$$
Then from \eqref{4.2.2}, we have
\begin{multline*}
\sup_{0\leq s\leq t}e^{\lambda_0 s^\alpha}\{\|wf(s)\|_{L^\infty}+|wf(s)|_{L^\infty(\g)}\}\\
\leq C\|wf_0\|_{L^\infty}+2C^2\|wf_0\|_{L^\infty}\cdot\{\delta+2C\|wf_0\|_{L^\infty}\}
\leq \f{3C}{2}\|wf_0\|,
\end{multline*}
provided that both $\delta>0$ and $\|wf_0\|_{L^\infty}$ are suitably small. This justifies that the a priori assumption can be closed. Then from a standard continuity argument, the global existence together with the estimate \eqref{D1.3} follow.  Therefore, the proof of Theorem \ref{thm1.2} is complete. \qed

\section{Dynamical stability under a class of large perturbations}

The section is devoted to proving Theorem \ref{thm1.3}. Recall $h(t,x,v):=wf(t,x,v)$. In what follows, we make the following a priori assumption:
\begin{align}\label{5.0.1}
\sup_{0\leq s\leq T}\|h(s)\|_{L^\infty}+\sup_{0\leq s\leq T}|h(s)|_{
{L^\infty(\g)}}\leq \bar{M},
\end{align}
where $T>0$ is an arbitrary constant and $\bar{M}$ is a positive constant depending only on $M_0$ as given in \eqref{M0}. We emphasize here that $M_0$ is not necessarily small and will be determined at the end of the proof.

\subsection{$L^p_{x,v}$ estimates}

First of all, we have the following

\begin{lemma}[\cite{GKTT-IM}]
Let $1<p<\infty$. Assume that $f$, $\pa_t+v\cdot \nabla_xf\in L^p([0,T]; L^p)$ and $f\Fi_{\g_-}\in L^p([0,T]; L^p(\g))$. Then $f\in C^{0}([0,T];L^p)$ and $f\Fi_{\g_+}\in L^p([0,T]; L^p(\g))$ and for almost every $t\in[0,T]$:
\begin{equation}\label{5.1.1}
\|f(t)\|_{L^p}^p+\int_0^t|f(s)|_{L^p(\g_+)}^p=\|f_0\|_{L^p}^p+\int_0^t|f(s)|_{L^p(\g_-)}^p+p\int_0^t\int_{\Omega}\int_{\mathbb{R}^3}\{\pa_tf+v\cdot \pa_xf\}|f|^{p-2}f.
\end{equation}
\end{lemma}

Moreover, we prove that the $L^p$ bound of solutions grows in time exponentially related to $\bar{M}$. Note that $\bar{M}$ is to be chosen depending only on $M_0$,  so that within a finite time interval, the solution can be uniformly small in $L^p$ if it is so initially.

\begin{lemma}\label{lm5.2}
Let $1<p<\infty$. Under the assumption \eqref{5.0.1}, it holds that
\begin{align}\label{5.1.2}
\|f(t)\|_{L^p}\leq e^{C_3\bar{M}t}\|f_0\|_{L^p},
\end{align}
for any $t\in [0,T]$. Here $C_3>1$ is a generic constant depending only on $\ka$ and $p$.
\end{lemma}
\begin{proof}
By Green's identity \eqref{5.1.1}, one has
\begin{align}\label{5.1.3}
&\|f(t)\|_{L^p}^p+\int_0^t|f(s)|_{L^p(\g_+)}^p\dd s+p\int_0^t\|\nu^{1/p}f(s)\|_{L^p}^p\dd s\nonumber\\
&\quad=\|f_0\|_{L^p}^p+\int_0^t|f(s)|_{L^p(\g_-)}^p+p\int_0^t\langle|f(s)|^{p-2}f(s),Kf(s)\rangle\dd s\nonumber\\
&\qquad+p\int_0^t\langle|f(s)|^{p-2}f(s),-L_{\sqrt{\mu} f_*}f(s)+\Gamma(f,f)(s)\rangle\dd s.
\end{align}
It is straightforward to see from \eqref{2.27} that $K$ is bounded from $L^p$ to $L^p$. Therefore, one has
\begin{align}\label{5.1.4}
\left|\int_0^t\langle|f(s)|^{p-2}f(s),Kf(s)\rangle\dd s\right|\leq C\int_0^t\|f(s)\|_{L^p}^{p-1}\cdot\|Kf(s)\|_{L^p}\dd s\leq C\int_0^t\|f(s)\|_{L^p}^{p}\dd s.
\end{align}
As for the last term on the right-hand side of \eqref{5.1.3}, it holds from \eqref{lm2.6-1} that
\begin{align}\label{5.1.5}
\left|\langle|f|^{p-2}f,-L_{\sqrt{\mu} f_*}f+\Gamma(f,f)\rangle\right|&\leq \|\nu^{1/p}f\|_{L^p}^{p-1}\|\nu^{-1/p'}\left|-L_{\sqrt{\mu} f_*}f+\Gamma(f,f)\right|\|_{L^p}\nonumber\\
&\leq C\{\|wf_*\|_{L^\infty}+\sup_{0\leq s\leq t}\|wf(s)\|_{L^\infty}\}\cdot\|\nu^{1/p}f\|_{L^p}^p\nonumber\\
&\leq C(\bar{M}+1)\|\nu^{1/p}f\|_{L^p}^p.
\end{align}
To treat the boundary term $|f|_{L^p(\g_-)}$, the same as before, we introduce the cutoff function $\Fi_{\g_+^\vep}$ near the grazing set $\g_+^\vep$ defined in \eqref{S2.1}. Then by a direct computation, we have
\begin{align}\label{5.1.6}
|f|_{L^p(\g_-)}^p&\leq C|P_\g f|_{L^p(\g_-)}^p+C|r|_{L^p(\g_-)}^p\leq C|P_\g f|_{L^p(\g_-)}^p+C\delta^p|f|_{L^p(\g_+)}^p\nonumber\\
&\leq \int_{\g_-}\mu^{p/2}(v)|n(x)\cdot v|\dd \gamma\left(\int_{\CV(x)}\sqrt{\mu(u)}\{f\Fi_{\g_+^\vep}+f\Fi_{\g_+\setminus\g_+^\vep}\}|n(x)\cdot u|\dd u\right)^p+C\delta^{{p}}|f|_{L^p(\g_+)}^p\nonumber\\
&\leq C\{\vep^p+C\delta^p\}\cdot|f|_{L^p(\g_+)}^p+C|f\Fi_{\g_+\setminus\g_+^\vep}|_{L^p(\g_+)}^p.
\end{align}
From the trace 
estimate \eqref{S2.3}, it holds that
\begin{multline}\label{5.1.7}
\int_0^t|f\Fi_{\g_+\setminus\g_+^\vep}(s)|_{L^p(\g_+)}^p\dd s=\int_0^t\left||f|^p\Fi_{\g_+\setminus\g_+^\vep}(s)\right|_{L^1(\g_+)}\dd s   \\
\leq C_{\vep,\Omega}\left\{\|f_0\|_{L^p}^p+\int_0^t\|f(s)\|_{L^p}^p\dd s+\int_0^t\left\|[\pa_t+v\cdot\nabla_x]|f(s)|^p\right\|_{L^1}\dd s\right\}.
\end{multline}
Notice that
$$[\pa_t+v\cdot\nabla_x]|f|^p=p|f|^{p-2}f\{\pa_tf+v\cdot \nabla_xf\}=p|f|^{p-2}f\left\{-\nu f+Kf+\Gamma(f,f)-L_{\sqrt{\mu}f_*}f\right\}.
$$
Then from \eqref{5.1.4} and \eqref{5.1.5}, it holds that
$$\int_0^t\left\|[\pa_t+v\cdot\nabla_x]|f(s)|^p\right\|_{L^1}\dd s\leq C(\bar{M}+1)\int_0^t\|{f(s)}\|_{L^p}^p\dd s.
$$
Combining this with \eqref{5.1.6} and \eqref{5.1.7}, we obtain that
\begin{equation}\label{5.1.8}
\int_0^t|f(s)|_{L^p(\g_-)}^p\dd s\leq C\{\vep^{p}+\delta^p\}\cdot\int_0^t|f(s)|_{L^p(\g_+)}^p\dd s+C_{\vep}\|f_0\|_{L^p}^p+C_{\vep}(\bar{M}+1)\int_0^t\|f(s)\|_{L^p}^p\dd s.
\end{equation}
Substituting \eqref{5.1.4}, \eqref{5.1.5} and \eqref{5.1.8} into \eqref{5.1.3}, and then taking both $\vep>0$ as well as $\delta>0$ suitably small, we have
\begin{align}
\|f(t)\|_{L^p}^p+\int_0^t|f(s)|_{L^p}^p\dd s+\int_0^t\|\nu^{1/p}f(s)\|_{L^p}^p\dd s\leq C\|f_0\|_{L^p}^p+C(\bar{M}+1)\int_0^t\|f(s)\|_{L^p}^p\dd s. \nonumber
\end{align}
Then \eqref{5.1.2} follows from Gronwall's inequality. Therefore, the proof of Lemma \ref{lm5.2} is complete.
\end{proof}

\subsection{$L^\infty_{x,v}$-estimate}
\begin{lemma}\label{lm5.3-1}
Under the a priori assumption \eqref{5.0.1}, for $t\in [0,\min\{T,T_0\}]$ and for almost every $(x,v)\in \bar{\Omega}\times \mathbb{R}^3\setminus \g_0$, it holds that:
\begin{align}\label{5.4.0}
|h(t,x,v)|\leq &S(t)+ \int_{\max\{t_1,0\}}^tI(t,s)\left|w\Gamma(f,f)(s,x-(t-s)v,v)\right|\dd s,
\end{align}
where
$$\begin{aligned}S(t)=&Ce^{-\lambda t^\alpha}M_0+CT_0^{5/2}\big\{m^{3+\ka}+\delta+2^{-\hat{C}_4T_0^{5/4}}+\f{1}N+\f{1}{N^{\beta-4}}\big\}\cdot\{\bar{M}+\bar{M}^2\}\nonumber\\
&+C_{N,T_0,m}\{\sup_{0\leq s\leq t}\|f(s)\|_{L^p}+\sup_{0\leq s\leq t}\|f(s)\|_{L^p}^2\}.
\end{aligned}$$
Here the positive constants $T_0$ and $N$ can be taken arbitrarily large and $m$ can be taken arbitrarily small.
\end{lemma}

\begin{proof}
We denote $\tilde{G}(t)$ as the solution operator of \eqref{4.1.1} provided by Proposition \ref{prop4.1}. Then the solution $h$ of \eqref{nph} 
is given in terms of Duhamel's formula as
\begin{align}\label{5.4.1}
h(t,x,v)=\left(\tilde{G}(t)h_0\right)(t,x,v)+\int_0^t \left(\tilde{G}(t-s)[-wL_{\sqrt{\mu}f_*}f(s)+w\Gamma(f,f)(s)]\dd s\right)(t,x,v).
\end{align}
Using \eqref{4.1.3} and \eqref{2.43}, we have
\begin{align}\label{5.4.2}
\|\tilde{G}(t)h_0\|_{L^\infty}\leq Ce^{-\lambda_0 t^\alpha}\|h_0\|_{L^\infty},
\end{align}
and
\begin{equation}\label{5.4.3}
\left|\int_0^t \tilde{G}(t-s)wL_{\sqrt{\mu}f_*}f(s)\dd s\right|\leq C\|wf_*\|_{L^\infty}\cdot\int_0^te^{-\lambda_0(t-s)^\alpha}\|h(s)\|_{L^\infty}\dd s\leq C\delta\sup_{0\leq s\leq t}\|h(s)\|_{L^\infty}.
\end{equation}
To estimate the last term on the right-hand side of \eqref{5.4.1}, denoting
$$Z(t,x,v):=\bigg(\tilde{G}(t-s)w\Gamma(f,f)(s)\bigg)(t,x,v)
$$
and then applying the mild formulation \eqref{4.1.5} to $Z(t,x,v)$, we obtain that
\begin{align}\label{5.4.4}
\left(\tilde{G}(t-s)w\Gamma(f,f)\right)(t,x,v)=\sum_{i=1}^3H_i+\Fi_{\{t_1>s\}}\sum_{i=4}^{11}H_i,
\end{align}
where
\begin{align}
&H_1=\Fi_{\{t_1\leq s\}}I(t,s)w\Gamma(f,f)(s,x-(t-s)v,v)\nonumber\\
&H_2+H_3=\int_{\max\{t_{1},s\}}^{t}I(t,\tau)[K_w^m+K_w^c]Z(\tau,x-(t-\tau)v,v)\dd\tau,\quad H_4=I(t,t_{1})wr[Z](t_{1},x_{1},v)\nonumber\\
&H_5=\frac{I(t,t_{1})}{\tilde{w}(v)}\int_{\prod_{j=1}^{k-1}\CV_{j}}\sum^{k-1}_{l=1}\Fi_{\{t_{l+1}\leq s<t_{l}\}}w\Gamma(f,f)(s,x_{l}-(t_{l}-s)v_{l},v_{l})\dd\Sigma_{l}(s)\nonumber
\end{align}
\begin{align}
&H_6+H_7=\frac{I(t,t_{1})}{\tilde{w}(v)}\int_{\prod_{j=1}^{k-1}\CV_{j}}\int_{s}^{t_{l}}\sum^{k-1}_{l=1}\Fi_{\{t_{l+1}\leq s<t_{l}\}}[K_w^m+K^c_w]Z(\tau,x_{l}-(t_{l}-\tau)v_{l},v_{l})\dd\Sigma_{l}(\tau)\dd\tau\nonumber\\
&H_8+H_9=\frac{I(t,t_{1})}{\tilde{w}(v)}\int_{\prod_{j=1}^{k-1}\CV_{j}}\int_{t_{l+1}}^{t_{l}}\sum^{k-1}_{l=1}\Fi_{\{t_{l+1}> s\}}[K_w^m+K^c_w]Z(\tau,x_{l}-(t_{l}-\tau)v_{l},v_{l})\dd\Sigma_{l}(\tau)\dd\tau\nonumber
\end{align}
\begin{align}
&H_{10}=\frac{I(t,t_{1})}{\tilde{w}(v)}\int_{\prod_{j=1}^{k-1}\CV_{j}}\sum^{k-2}_{l=1}\Fi_{\{t_{l+1}> s\}}wr[Z](t_{l+1},x_{l+1},v_{l})\dd\Sigma_{l}(t_{l+1})\nonumber\\
&H_{11}=\frac{I(t,t_{1})}{\tilde{w}(v)}\int_{\prod_{j=1}^{k-1}\CV_{j}}\Fi_{\{t_{k}>s\} }Z(t_{k},x_{k},v_{k-1})\dd\Sigma_{k-1}(t_{k}),\nonumber
\end{align} 
and
$r$ is defined in \eqref{4.1.5-2}. Here the same as before, $k=\hat{C}_3T_0^{5/4}$ such that \eqref{4.1.5-1} holds for $\eta=\f{5}{16}$. We first consider terms $H_2$, $H_5$ and $H_7$ involving $K_w^m$. On one hand, similar for obtaining \eqref{4.1.11}, we have
\begin{align}\label{5.4.5}
|H_2|&\leq \int_{\max\{t_{1},s\}}^{t}I(t,\tau)\left|K_w^mZ(\tau,x-(t-\tau)v,v)\right|\dd\tau\nonumber\\
&\leq \int_{\max\{t_{1},s\}}^{t}I(t,\tau)\{\Fi_{\{|v|\leq d_\Omega\}}+\Fi_{\{|v|>d_\Omega\}}\}\left|K_w^mZ(\tau,x-(t-\tau)v,v)\right|\dd\tau\nonumber\\
&\leq Cm^{3+\ka}\int_{\max\{t_1,s\}}^te^{-\bar{\nu}_0(t-\tau)}\|\tilde{G}(\tau-s)w\Gamma(f,f)(s)\|_{L^\infty}\dd\tau\nonumber\\
&\leq Cm^{3+\ka}\|h(s)\|_{L^\infty}^{{2}}\cdot\int_{\max\{t_1,s\}}^te^{-\bar{\nu}_0(t-\tau)-\lambda_0(\tau-s)^\alpha}\dd\tau\nonumber\\
&\leq  Cm^{3+\ka}e^{-\lambda_0(t-s)^\alpha}\|h(s)\|^2_{L^\infty}.
\end{align}
On the other hand, similar for obtaining \eqref{4.1.12}, we have
\begin{align}\label{5.4.6}
|H_6|\leq& \sum_{l=1}^{k-1}\sum_{m=1}^{l}C\bigg\{\int_{\prod_{j=1}^{l}\CV_{j}}\Fi_{\{t_{l+1}\leq s<t_{l}\}}\times\Fi_{\big\{|v_{m}|= \max\big[|v_{1}|,|v_{2}|,...|v_{l}|\big]\big\}}e^{\frac{5|v_{m}|^{2}}{16}}
\prod_{j=1}^{l}\dd\sigma_{j}\nonumber\\
&\qquad\times \int_s^{t_l}e^{-\lambda_1(t-\tau)^\alpha}\|K_w^mZ(\tau)\|_{L^\infty}\dd\tau\bigg\}\nonumber\\
\leq& Ck^2m^{3+\ka}e^{-\lambda_0(t-s)^{\alpha}}\|h(s)\|^2_{L^\infty}\cdot\sup_{j}\left|\int_{\CV_j}e^{\frac{5|v_{j}|^{2}}{16}}\dd\sigma_j\right|\nonumber\\
\leq& Ck^2m^{3+\ka}e^{-\lambda_0(t-s)^{\alpha}}\|h(s)\|^2_{L^\infty}.
\end{align}
Similarly, it holds that
\begin{align}\label{5.4.7}
|H_8|\leq Ck^2m^{3+\ka}e^{-\lambda_0(t-s)^{\alpha}}\|h(s)\|^2_{L^\infty}.
\end{align}
For the terms $H_4$ and $H_{10}$ involving $r$, we see from \eqref{4.1.3} and \eqref{4.1.5-2} that
$$\begin{aligned}
|r[Z](\tau)|_{L^\infty(\g_-)}\leq C\delta e^{-\lambda_0(\tau-s)^\alpha}|h(s)|^2_{L^\infty(\g_+)}.
\end{aligned}
$$
Therefore, similar for obtaining \eqref{5.4.6}, we have
\begin{align}\label{5.4.8}
|H_{10}|\leq& \sum_{l=1}^{k-1}\sum_{m=1}^{l}C\bigg\{\int_{\prod_{j=1}^{l}\CV_{j}}\Fi_{\{t_{l+1}>s\}}\nonumber\\
&\times\Fi_{\big\{|v_{m}|= \max\big[|v_{1}|,|v_{2}|,...|v_{l}|\big]\big\}}e^{\frac{5|v_{m}|^{2}}{16}}e^{-\lambda_1(t-t_{l+1})^\alpha}|r[Z](t_{l+1})|_{L^\infty(\g_-)}
\prod_{j=1}^{l}\dd\sigma_{j}\bigg\}\nonumber\\
\leq& Ck^2\delta e^{-\lambda_0(t-s)^{\alpha}}|h(s)|^2_{L^\infty(\g_+)}\cdot\sup_{j}\left|\int_{\CV_j}e^{\frac{5|v_{j}|^{2}}{16}}\dd\sigma_j\right|\nonumber\\
\leq& Ck^2\delta e^{-\lambda_0(t-s)^{\alpha}}|h(s)|^2_{L^\infty(\g_+)},
\end{align}
and
\begin{align}\label{5.4.9}
|H_4|&\leq C\Fi_{\{t_1>s\}}I(t,t_{1})\{\Fi_{\{|v|\leq d_\Omega\}}+\Fi_{\{|v|>d_\Omega\}}\}\left|wr[Z](t_{1},x_{1},v)\right|\nonumber\\
&\leq C e^{-\bar{\nu}_0(t-t_1)}\left|wr[Z](t_{1},x_{1},v)\right|\leq C\delta e^{-\lambda_0(t-s)^\alpha}|h(s)|^2_{L^\infty(\g_+)}.
\end{align}
For $H_{11}$, we note that
$$|Z(t_{k},x_k,v_{k-1})|\leq \left|\tilde{G}(t_{k}-s)w\Gamma(f,f)(s)\right|_{L^\infty{(\g_-)}
}\leq e^{-\lambda_0(t_k-s)^\alpha}\|h(s)\|^2_{L^\infty}.
$$
Then by \eqref{4.1.5-1}, we have
\begin{align}\label{5.4.10}
|H_{11}|\leq Ce^{-\lambda_0(t-s)^\alpha}\|h(s)\|^2_{L^\infty}\cdot\left(\f12\right)^{\hat{C}_4T_0^{5/4}}.
\end{align}
For the terms $H_7$ and $H_9$ involving $K_w^c$, similar for obtaining \eqref{4.1.17}, we have
\begin{align*}
|H_8| &\leq  C \sum_{l=1}^{k-1}\sum_{m=1}^{{l}
}\int_{\Pi _{j=1}^{l-1}{\mathcal{V}}_{j}}\dd{\sigma}_{l-1}\cdots \dd{\sigma}_1 \nonumber\\
&\qquad\times \int_{\mathcal{V}_l}\int_{\mathbb{R}^3} \int_s^{{t}_l} e^{-\lambda_1(t-\tau)^{\alpha}} \Fi_{\{{t}_{l+1}\leq s<{t}_l\}} e^{\f{5|v_m|^2}{16}} |k^c_w(v_l,v') Z(\tau,{x}_l-{v}_l({t}_l-\tau),v')|\dd \tau\dd v' \dd{\sigma}_l,
\end{align*}
and further split it as
\begin{align*}
|H_8| 
&\leq C \sum_{l=1}^{k-1}\sum_{m=1}^{l-1}\int_{\Pi _{j=1}^{l-1}{\mathcal{V}}_{j}}e^{\f{5|v_m|^2}{16}} \dd{\sigma}_{l-1}\cdots \dd{\sigma}_1 \int_{\mathcal{V}_l\cap \{|v_l|\geq N\}}\int_{\mathbb{R}^3}\int_s^{{t}_l} \bar{\Delta}\dd \tau\dd v' \dd{\sigma}_l \nonumber\\
&\quad+C \sum_{l=1}^{k-1}\sum_{m=1}^{l-1}\int_{\Pi _{j=1}^{l-1}{\mathcal{V}}_{j}}e^{\f{5|v_m|^2}{16}}  \dd{\sigma}_{l-1}\cdots \dd{\sigma}_1 \int_{\mathcal{V}_l\cap \{|v_l|\leq N\}}\int_{\mathbb{R}^3} \int_{t_l-\f1N}^{{t}_l}  \bar{\Delta}\dd \tau\dd v' \dd{\sigma}_l\nonumber\\
&\quad+ C \sum_{l=1}^{k-1}\sum_{m=1}^{l-1}\int_{\Pi _{j=1}^{l-1}{\mathcal{V}}_{j}} e^{\f{5|v_m|^2}{16}} \dd{\sigma}_{l-1}\cdots \dd{\sigma}_1 \int_{\mathcal{V}_l\cap \{|v_l|\leq N\}}\int_{\{|v'|\geq 2N\}} \int_s^{{t}_l-\f1N}  \bar{\Delta} \dd \tau\dd v' \dd{\sigma}_l\nonumber\\
&\quad+ C \sum_{l=1}^{k-1}\sum_{m=1}^{l-1}\int_{\Pi _{j=1}^{l-1}{\mathcal{V}}_{j}} e^{\f{5|v_m|^2}{16}} \dd{\sigma}_{l-1}\cdots \dd{\sigma}_1 \int_{\mathcal{V}_l\cap \{|v_l|\leq N\}}\int_{\{|v'|\leq 2N\}} \int_s^{{t}_l-\f1N}  \bar{\Delta} \dd \tau\dd v' \dd{\sigma}_l,\nonumber
\end{align*}
where
$$
\bar{\Delta}:=e^{-\lambda_1(t-\tau)^{\alpha}} \Fi_{\{{t}_{l+1}\leq s<{t}_l\}}e^{\f{5|v_l|^2}{16}}|k^c_w(v_l,v') Z(\tau,{x}_l-{v}_l({t}_l-\tau),v')|.
$$
Then it follows that 
\begin{align}\label{5.4.11}
|H_8| 
&\leq \f{Ck^2e^{-\lambda_0(t-s)^\alpha}}{N}\sup_{s\leq \tau\leq t}\|h(\tau)\|^2_{L^\infty}\nonumber\\
&\quad+C  \sum_{l=1}^{k-1}\sum_{m=1}^{l-1}\int_{\Pi _{j=1}^{l-1}{\mathcal{V}}_{j}} e^{\f{5|v_m|^2}{16}} \dd{\sigma}_{l-1}\cdots \dd{\sigma}_1 \int_{\mathcal{V}_l\cap \{|v_l|\leq N\}}\int_{\{|v'|\leq 2N\}} \int_s^{{t}_l-\f1N}  \bar{\Delta} \dd \tau\dd v' \dd{\sigma}_l,
\end{align}
By {H\"older} inequality, it holds that
\begin{equation}\label{5.4.12}\begin{aligned}
&\int_{\mathcal{V}_l\cap \{|v_l|\leq N\}}\int_{\{|v'|\leq 2N\}} \int_s^{{t}_l-\f1N}  \bar{\Delta} \dd \tau\dd v' \dd{\sigma}_l\\
&\quad\leq C\int_s^{t_l-\f1N}e^{-\lambda_1(t-\tau)^\alpha}\dd\tau\left(\int_{\CV_l\cap\{|v_l|\leq N\}}\int_{\{|v'|\leq 2N\}}e^{-\f{|v_l|^2}{8}}|k_w^c(v_l,v')|^{p'}\dd v'\dd v_l\right)^{\f{1}{p'}}\\
&\qquad\times \left(\int_{\CV_l\cap\{|v_l|\leq N\}}\int_{\{|v'|\leq 2N\}}\Fi_{{\{}t_{l+1}\leq s<t_l{\}}}|Z(\tau,x_{l}-(t_l-\tau)v_l,v')|^p\dd v'\dd v_l\right)^{\f{1}{p}}.
\end{aligned}\end{equation}
Since $1<p'=\f{p}{p-1}<3$, then by \eqref{2.33} with $a=1$, it holds that  
$$
\int_{\CV_l\cap\{|v_l|\leq N\}}\int_{\{|v'|\leq 2N\}}e^{-\f{|v_l|^2}{8}}|k_w^c(v_l,v')|^{p'}\dd v'\dd v_l\leq Cm^{p'(\ka-1)}.
$$
Therefore, it holds that
$$
\begin{aligned}
\text{R.H.S. of \eqref{5.4.12}} \leq& C_Nm^{{\ka}
-1}\int_s^{t_l-\f1N}e^{-\lambda_1(t-\tau)^\alpha}\dd\tau\\
&\times\left(\int_{\CV_l\cap\{|v_l|\leq N\}}\int_{\{|v'|\leq 2N\}}\Fi_{{\{}t_{l+1}\leq s<t_l{\}}}\left|\f{Z}{w}(\tau,x_{l}-(t_l-\tau)v_l,v')\right|^p\dd v'\dd v_l\right)^{\f{1}{p}}.
\end{aligned}
$$
Note that $y_l:=x_l-(t_l-\tau)v_l\in\Omega$ for $s\leq \tau\leq t_l-\f1N$. Then making chang of variable $v_l\rightarrow y_l$,
we obtain that
$$\int_{\mathcal{V}_l\cap \{|v_l|\leq N\}}\int_{\{|v'|\leq 2N\}} \int_s^{{t}_l-\f1N}  \bar{\Delta} \dd \tau\dd v' \dd{\sigma}_l\leq C_{N,m}\int_{s}^{t_l-\f1N}e^{-\lambda_1(t-\tau)^\alpha}\left\|\f{Z(\tau)}{w}\right\|_{L^p}\dd\tau.
$$
Similar for obtaining \eqref{5.1.2}, we have
\begin{align}\label{5.4.13}
\left\|\f{Z(\tau)}{w}\right\|_{L^p}&=\left\|\f{\tilde{G}(\tau-s)w\Gamma(f,f)(s)}{w}\right\|_{L^p}\leq Ce^{C (\tau-s)}\|\Gamma(f,f)(s)\|_{L^p}\nonumber\\
&\leq C_{T_0}e^{-\lambda_0 (\tau-s)^\alpha}\|\Gamma(f,f)(s)\|_{L^p}\leq C_{T_0}e^{-\lambda_0 (\tau-s)^\alpha}\|h(s)\|_{L^\infty}\cdot\|f(s)\|_{L^p}\nonumber\\
&\leq \f{Ce^{-\lambda_0 (\tau-s)^\alpha}}{N}\|h(s)\|^2_{L^\infty}+C_{N,T_0}e^{-\lambda_0 (\tau-s)^\alpha}\|f(s)\|_{L^p}^2,
\end{align}
which implies that
$$\int_{\mathcal{V}_l\cap \{|v_l|\leq N\}}\int_{\{|v'|\leq 2N\}} \int_s^{{t}_l-\f1N}  \bar{\Delta} \dd \tau\dd v' \dd{\sigma}_l\leq \f{Ce^{-\lambda_0 (t-s)^\alpha}}{N}\|h(s)\|^2_{L^\infty}+C_{N,T_0,m}e^{-\lambda_0 (t-s)^\alpha}\|f(s)\|_{L^p}^2.
$$
Substituting this into \eqref{5.4.11}, we have
\begin{align}\label{5.4.14}
|H_8|\leq \f{Ck^2e^{-\lambda_0(t-s)^\alpha}}{N}\sup_{0\leq s\leq t}\|h(s)\|^2_{L^\infty}+C_{N,T_0,m}e^{-\lambda_0(t-s)^\alpha}\sup_{0\leq s\leq t}\|f(s)\|^2_{L^p}.
\end{align}
Similarly, we have
\begin{align}\label{5.4.15}
|H_{10}|\leq \f{Ck^2e^{-\lambda_0(t-s)^\alpha}}{N}\sup_{0\leq s\leq t}\|h(s)\|^2_{L^\infty}+C_{N,T_0,m}e^{-\lambda_0(t-s)^\alpha}\sup_{0\leq s\leq t}\|f(s)\|^2_{L^p}.
\end{align}
For $H_5$, similar as \eqref{4.1.12}, we have
\begin{align}
\int_0^t|H_{5}| \dd s&\leq  C\int_0^t e^{-\lambda_1(t-s)^{\alpha}}\dd s\sum_{l=1}^{k-1}\sum_{m=1}^{l}\int_{\Pi _{j=1}^{l-1}{\mathcal{V}}_{j}}\dd{\sigma}_{l-1}\cdots \dd{\sigma}_1 \nonumber\\
&\qquad\times \int_{\mathcal{V}_l}\int_{\mathbb{R}^3}  \Fi_{\{{t}_{l+1}\leq s<{t}_l\}} e^{\f{5|v_m|^2}{16}} \left|w\Gamma(f,f)(s,x_l-(t_l-s)v_l,v_l)\right|\dd \tau\dd v' \dd{\sigma}_l. \nonumber
\end{align}
Then it follows that
\begin{align}\label{5.4.15-1}
\int_0^t|H_{5}| \dd s
&\leq C\int_0^{t}e^{-\lambda_1(t-s)^\alpha}\|h(s)\|_{L^\infty}\dd s \sum_{l=1}^{k-1}\sum_{m=1}^{l-1}\int_{\Pi _{j=1}^{l-1}{\mathcal{V}}_{j}}e^{\f{5|v_m|^2}{16}} \dd{\sigma}_{l-1}\cdots \dd{\sigma}_1\nonumber\\
&\quad\times\int_{\mathcal{V}_l}e^{-\f{|v_l|^2}{8}}\Fi_{\{t_{l+1}\leq s<t_l\}}\dd v_l\left\{\int_{\mathbb{R}^3} \left|h(s,x_{l}-(t_l-s)v_l,v')\right|^p\langle v'\rangle^{-4-p(\beta-4)}\dd v'\right\}^{1/p}\nonumber\\
&\leq  C\int_0^{t}e^{-\lambda_1(t-s)^\alpha}\|h(s)\|_{L^\infty}\dd s\nonumber\\
&\quad \times \bigg(\sum_{l=1}^{k-1}\sum_{m=1}^{l-1}\int_{\Pi _{j=1}^{l-1}{\mathcal{V}}_{j}}e^{\f{5|v_m|^2}{16}} \dd{\sigma}_{l-1}\cdots \dd{\sigma}_1\left\{\int_{\CV_l}\int_{\mathbb{R}^3}\tilde{\Delta}\dd v_l\dd v'\right\}^{1/p}\bigg),
\end{align}
where we have used \eqref{2.41} and \eqref{2.42} in the first  inequality, and also denoted that 
$$
\tilde{\Delta}:=e^{-\f{|v_l|^2}{8}}\Fi_{\{t_{l+1}\leq s<t_l\}}\left|h(s,x_{l}-(t_l-s)v_l,v')\right|^p\langle v'\rangle^{-4-p(\beta-4)}.
$$
Now, we consider the integral in \eqref{5.4.15-1} over either $\{|v_l|\geq N\}$ or $\{|v_l|\leq N, |v'|\geq N\}$ or $\{|v_l|\leq N, |v'|\leq N, t_l-1/N\leq s \leq t_l\}$ or $\{|v_l|\leq N, |v'|\leq N, 0\leq s \leq t_l-1/N\}$. Over $\{|v_l|\geq N\}$ or $\{|v_l|\leq N, |v'|\geq N\}$ or $\{|v_l|\leq N, |v'|\leq N, t_l-1/N\leq s \leq t_l\}$, it is bounded by
$$ Ck^2\left(\f1N+\f1{N^{\beta-4}}\right)\bar{M}^2.
$$
Over $\{|v_l|\leq N, |v'|\leq N, 0\leq s \leq t_l-1/N\}$, it is bounded by
\begin{align}
&C_N\sum_{l=1}^{k-1}\sum_{m=1}^{l-1}\int_0^{t}e^{-\lambda_1(t-s)^\alpha}\|h(s)\|_{L^\infty}\dd s \int_{\Pi _{j=1}^{l-1}{\mathcal{V}}_{j}}e^{\f{5|v_m|^2}{16}} \dd{\sigma}_{l-1}\cdots \dd{\sigma}_1\nonumber\\
&\quad\times\left\{\int_{\CV_l\cap\{|v_l|\leq N\}}\int_{\{|v'|\leq N\}}e^{-\f{|v_l|^2}{8}}\Fi_{\{t_{l+1}\leq s<t_l-\f1N\}}\left|f(s,x_{l}-(t_l-s)v_l,v')\right|^p\dd v_l\dd v'\right\}^{1/p}\nonumber\\
&\leq C_Nk^2\int_0^{t}e^{-\lambda_1(t-s)^\alpha}\|h(s)\|_{L^\infty}\cdot\|f(s)\|_{L^p}\dd s\leq \f{Ck^2}{N}\bar{M}^2+C_Nk^2\sup_{0\leq s\leq t}\|f(s)\|_{L^p}^2,\nonumber
\end{align}
where we have used the change of variable $v_l\rightarrow y_l:=x_l-(t_l-s)v_l$ above. Therefore, for $H_5$, it holds that
\begin{align}\label{5.4.16}
\int_0^t|H_5|\dd s\leq Ck^2\left(\f{1}{N}+\f{1}{N^{\beta-4}}\right)\bar{M}^2+C_Nk^2\sup_{0\leq s\leq t}\|f(s)\|_{L^p}^2.
\end{align}
Substituting \eqref{5.4.5}, \eqref{5.4.6}, \eqref{5.4.7}, \eqref{5.4.8}, \eqref{5.4.9}, \eqref{5.4.10}, \eqref{5.4.14} \eqref{5.4.15} and \eqref{5.4.16} into \eqref{5.4.4}, we have
\begin{align}\label{5.4.17}
&\int_0^t\left|\left(\tilde{G}(t-s)w\Gamma(f,f)(s)\right)(t,x,v)\right|\dd s\nonumber\\
&\leq \int_0^t\Fi_{\{t_1\leq s\}}I(t,s)w\Gamma(f,f)(s,x-(t-s)v,v)\dd s\nonumber\\
&\quad+\int_0^t\dd s\int_{\max\{t_{1},s\}}^{t}I(t,\tau)\dd\tau\int_{\mathbb{R}^3}\left|k_w^c(v,u)\left(\tilde{G}(\tau-s)w\Gamma(f,f)(s)\right)(\tau,x-(t-\tau)v,u)\right|\dd u\nonumber\\
&\quad+CT_0^{5/2}\bar{M}^2\big\{m^{3+\ka}+\delta+2^{-\hat{C}_4T_0^{5/4}}+\f{1}N+\f{1}{N^{\beta-4}}\big\}+C_{N,T_0,m}\sup_{0\leq s\leq t}\|f(s)\|_{L^p}^2\nonumber\\
&:=H_{12}+H_{13}+H_{14}.
\end{align}
To further estimate $H_{13}$, we denote $x'=x-(t-\tau)v$ and $\tau_1'=t_1(\tau,x',u).$
Then by Fubini Theorem and \eqref{5.4.17}, it holds that
\begin{align}\label{5.4.19}
|H_{13}|&=\left|\int_{\max\{t_1,0\}}^t\int_{\mathbb{R}^3}I(t,\tau)k_w^c(v,u)\dd u\dd\tau\int_{0}^\tau\left(\tilde{G}(\tau-s)w\Gamma(f,f)(s)\right)(\tau,x',u)\dd s\right|\nonumber\\
&\leq H_{13,1}+H_{13,2}+H_{13,3},
\end{align}
where
\begin{align}
H_{13,1}&=
\int_{\max\{t_1,0\}}^t\int_0^\tau\int_{\mathbb{R}^3}|k_w^c(v,u)|\Fi_{\{t_1'\leq s\}}I(t,s)\left|w\Gamma(f,f)(s,x'-(\tau-s)u,u)\right|\dd u\dd s\dd\tau,  \nonumber\\
H_{13,2}&=\int_{\max\{t_1,0\}}^t\int_0^\tau\int_{\mathbb{R}^3}\int_{\max\{t_1',s\}}^\tau I(t,\tau')|k_w^c(v,u)|\dd\tau'\dd u\dd s\dd\tau\nonumber\\
&\qquad\times\int_{\mathbb{R}^3} \left|k_{w}^c(u,u')\left(\tilde{G}(\tau'-s)w\Gamma(f,f)(s)\right)(\tau',x'-(\tau-\tau')u,u')\right|\dd u',\nonumber\\
H_{13,3}&=|H_{14}|\cdot\int_{\max\{t_1,0\}}^t\int_{\mathbb{R}^3}I(t,\tau)|k_w^c(v,u)|\dd u\dd\tau.\nonumber
\end{align}
Similar as before, we have
\begin{align}\label{5.4.21}
|H_{13,3}|&\leq |H_{14}|\cdot\int_{\max\{t_1,0\}}^tI(t,\tau)\{\Fi_{\{|v|\leq d_\Omega\}}+\Fi_{\{|v|>d_\Omega\}}\}\dd\tau\int_{\mathbb{R}^3}|k_w^c(v,u)|\dd u\nonumber\\
&\leq C|H_{14}|\cdot\int_{\max\{t_1,0\}}^te^{-\bar{\nu}_0(t-\tau)}\dd\tau\leq CH_{14}.
\end{align}
For $H_{13,1}$, we have from \eqref{2.41} and \eqref{2.42} that
\begin{align*}
|H_{13,1}|&\leq C \int_{\max\{t_1,0\}}^t\dd\tau\int_{\mathbb{R}^3}|k_w^c(v,u)|\dd u\int_{0}^\tau\Fi_{\{t_1'\leq s\}}e^{\Red{-}\bar{\nu}_0(t-s)}\|h(s)\|_{L^\infty}\dd s\nonumber\\
&\qquad\times\left(\int_{\mathbb{R}^3}\left|h(s,x'-(\tau-s)u,u')\right|^p\langle u' \rangle^{-4-p(\beta-4)}\dd u'\right)^{1/p}.
\end{align*}
Now we divide the estimates by the following cases.

\medskip
\noindent{\it Case 1.} $|v|\geq N$. We have from \eqref{2.40-1} that
$$
|H_{13,1}|\leq \f{C}{N}\sup_{0\leq s\leq t}\|h(s)\|_{L^\infty}^2.
$$
{\it Case 2.} $|v|\leq N, |u|\geq 2N$. In this case, we have $|v-u|\geq N$, so that by \eqref{2.40-1},
$$\int_{\{|u|\leq 2N\}}|k_{w}^c(v,u)|\dd u\leq e^{-\f{N^2}{32}}\int_{\{|u|\leq 2N\}}|k_{w}^c(v,u)e^{\f{|v-u|^2}{32}}|\dd u\leq Ce^{-\f{N^2}{32}}.
$$
It then follows that
$$|H_{13,1}|\leq Ce^{-\f{N^2}{32}}\sup_{0\leq s\leq t}\|h(s)\|_{L^\infty}^2.
$$
{\it Case 3.} $|v|\leq N, |u|\leq 2N, |u'|>N$. By $\beta>4$, it holds that
$$|H_{13,1}|\leq \f{C}{N^{\beta-4}}\sup_{0\leq s\leq t}\|h(s)\|_{L^\infty}^2.
$$
{\it Case 4.} $|v|\leq N, |u|\leq 2N, |u'|\leq N, \tau-1/N<s\leq \tau$. It is straightforward to see that
$$|H_{13,1}|\leq \f{C}{N}\sup_{0\leq s\leq t}\|h(s)\|_{L^\infty}^2.
$$
{\it Case 5.} $|v|\leq N, |u|\leq 2N, |u'\leq N, 0\leq s\leq \tau-1/N.$ By H\"older's inequality, we have
$$\begin{aligned}
|H_{13,1}|&\leq C_N\sup_{0\leq s\leq t}\|h(s)\|_{L^\infty}\cdot\int_{\max\{t_1,0\}}^t\int_0^{\tau-\f1N}e^{-\bar{\nu}_0(t-s)}\dd s\dd\tau\left(\int_{\{|u|\leq 2N\}}|k_w^c(v,u)|^{p'}\dd u\right)^{1/p'}\\
&\qquad\cdot\left(\int_{\{|u|\leq 2N\}}\int_{\{|u'|\leq N\}}\Fi_{\{t_1'\leq s\}}\left|f(s,x'-(\tau-s)u,u')\right|^p\dd u\dd u'\right)^{1/p}\\
&\leq C_{N,m}\sup_{0\leq s\leq t}\|h(s)\|_{L^\infty}\cdot \sup_{0\leq s\leq t}\|f(s)\|_{L^p}\leq \f{C}{N}\sup_{0\leq s\leq t}\|h(s)\|^2_{L^\infty}+C_{N,m}\sup_{0\leq s\leq t}\|f(s)\|_{L^p}^2.
\end{aligned}
$$
Collecting the estimates for these cases, we have
\begin{align}\label{5.4.22}
|H_{13,1}|\leq C\left(\f1N+\f1{N^{\beta-4}}\right)\sup_{0\leq s\leq t}\|h(s)\|^2_{L^\infty}+C_{N,m}\sup_{0\leq s\leq t}\|f(s)\|_{L^p}^2.
\end{align}
Similarly, for $H_{13,2}$, we have
\begin{align}\label{5.4.23}
|H_{13,2}|\leq &\int_{\max\{t_1,0\}}^t\int_{0}^\tau\dd s\dd \tau\int_{s}^\tau e^{-\bar{\nu}_0(t-\tau')}\dd\tau'\nonumber\\
&\times\int_{\mathbb{R}^3}\int_{\mathbb{R}^3}\left|k_{w}^c(v,u)k_{w}^c(u,u')\Fi_{\{\max\{t_1', s\}\leq \tau'\leq \tau\}}Z(\tau',x'-(\tau-\tau')u,u')\right|\dd u\dd u'\nonumber\\
\leq &\f{C}{N}\sup_{0\leq s\leq t}\|h(s)\|_{L^\infty}^2+C_N\int_{\max\{t_1,0\}}^t\int_{0}^\tau\dd s\dd \tau\int_{s}^{\tau-\f1N} e^{-\bar{\nu}_0(t-\tau')}\dd\tau'\nonumber\\
&\quad\times\left(\int_{\{|u|\leq 2N\}}\int_{\{|u'|\leq 3N\}}\left|k_{w}^c(v,u)k_{w}^c(u,u')\right|^{p'}\dd u\dd u'\right)^{1/p'}\nonumber\\
&\quad\times \left(\int_{\{|u|\leq 2N\}}\int_{\{|u'|\leq 3N\}} \Fi_{\{\max\{t_1', s\}\leq \tau'\leq \tau\}}\left|\f{Z}{w}(\tau',x'-(\tau-\tau')u,u')\right|^p\dd u\dd u'\right)^{1/p}\nonumber\\
\leq &\f{C}{N}\sup_{0\leq s\leq t}\|h(s)\|_{L^\infty}^2+C_{N,m}\int_{\max\{t_1,0\}}^t\int_{0}^\tau\dd s\dd \tau\int_{s}^{\tau-\f1N} e^{-\bar{\nu}_0(t-\tau')}\left\|\f{Z(\tau')}{w}\right\|_{L^p}\dd\tau'\nonumber\\
\leq &\f{C}{N}\sup_{0\leq s\leq t}\|h(s)\|_{L^\infty}^2+C_{N,{T_0,}m}\sup_{0\leq s\leq t}\|f(s)\|_{L^p}^2.
\end{align}
Here we have used \eqref{5.4.13} in the last inequality. Substituting \eqref{5.4.21}, \eqref{5.4.22} and \eqref{5.4.23} into \eqref{5.4.19}, we have
\begin{align}\label{5.4.24}
&\int_0^t\left|\left(\tilde{G}(t-s)w\Gamma(f,f)(s)\right)(t,x,v)\right|\dd s\nonumber\\
&\quad \leq\int_0^t\Fi_{\{t_1\leq s\}}I(t,s)w\Gamma(f,f)(s,x-(t-s)v,v)\dd s\nonumber\\
&\qquad+  CT_0^{5/2}\big\{m^{3+\ka}+\delta+2^{-\hat{C}_4T_0^{5/4}}+\f{1}N+\f{1}{N^{\beta-4}}\big\}\bar{M}^2+C_{N,T_0,m}\sup_{0\leq s\leq t}\|f(s)\|_{L^p}^2.
\end{align}
Then \eqref{5.4.0} naturally follows from \eqref{5.4.2}, \eqref{5.4.3} and \eqref{5.4.24}. Therefore, the proof of Lemma \ref{lm5.3-1} is complete.
\end{proof}

\begin{lemma}\label{lm5.4}
Under the assumption \eqref{5.0.1}, there exists a constant $C>0$ independent of $t$, such that, for any $0\leq t\leq \min\{T,T_0\}$, it holds that:
\begin{align}\label{5.4.25}
\|h(t)\|_{L^\infty}+|h(t)|_{L^\infty{(\g)}}\leq&   CM_0e^{-\lambda_0t^\alpha}\left(1+\int_{0}^t\|h(\tau)\|_{L^\infty}\dd\tau\right)\nonumber\\
&+CT_0^{5/2}\big\{m^{3+\ka}+\delta+2^{-\hat{C}_4T_0^{5/4}}+\f{1}N+\f{1}{N^{\beta-4}}\big\}\cdot\{\bar{M}+\bar{M}^3\}\nonumber\\
&+C_{N,T_0,m}\{\sup_{0\leq s\leq t}\|f(s)\|_{L^p}+\sup_{0\leq s\leq t}\|f(s)\|_{L^p}^3\}.
\end{align}
Here the positive constants $T_0$ and $N$ can be chosen arbitrarily large and $m$ can be chosen arbitrarily small.
\end{lemma}

\begin{proof}
We denote $x':=x-(t-s)v$ and $t_1':=t_{1}(s,x',u)$.
It suffices to consider the last term on the right-hand side of \eqref{5.4.0}. By \eqref{2.41} and \eqref{2.42}, it holds that
\begin{align}
&\int_{\max\{t_1,0\}}^tI(t,s)\left|w\Gamma(f,f)(s,x-(t-s)v,u)\right|\dd u\dd s\nonumber\\
&\leq \int_{\max\{t_1,0\}}^te^{-\bar{\nu}_0(t-s)}\|wf(s)\|_{L^\infty}\dd s\left(\int_{\mathbb{R}^3}\left|h(s,x-(t-s)v,u)\right|^p\langle u\rangle^{-4-p(\beta-4)}\right)^{1/p}\nonumber \\
&\leq \f{C}{N^{\beta-4}}\sup_{0\leq s\leq t}\|h(s)\|_{L^\infty}^2\nonumber\\
&\quad+C\int_{\max\{t_1,0\}}^te^{-\bar{\nu}_0(t-s)}\|wf(s)\|_{L^\infty}\dd s\times\left(\int_{\{|u|\leq N\}}\left|h(s,x',u)\right|^p\langle u\rangle^{-4-p(\beta-4)}\right)^{1/p}.\nonumber
\end{align}
Notice that $x':=x-(t-s)v\in \Omega$. Then applying \eqref{5.4.0} to $h(s,x',u)$, we have
\begin{equation}\label{5.4.27}
\int_{\max\{t_1,0\}}^tI(t,s)\left|w\Gamma(f,f)(s,x-(t-s)v,u)\right|\dd u\dd\tau
\leq \f{C}{N^{\beta-4}}\sup_{0\leq s\leq t}\|h(s)\|_{L^\infty}^2+R_1+R_2,
\end{equation}
where we have defined
$$
R_1=C \int_{\max\{t_1,0\}}^te^{-\bar{\nu}_0(t-s)}\|wf(s)\|_{L^\infty}S(s)\dd s,
$$
and
\begin{align*}
R_2&=C \int_{\max\{t_1,0\}}^te^{-\bar{\nu}_0(t-s)}\|wf(s)\|_{L^\infty}\dd s\times\bigg(\int_{\{|u|\leq N\}}\langle u \rangle^{-4-p(\beta-4)}\dd u\nonumber\\
&\quad\qquad\times\bigg\{\int_{\max\{t_1',0\}}^s e^{-\bar{\nu}_0(s-s')}|w\Gamma(f,f)(s',x'-(s-s')u,u)|\dd\tau'\bigg\}^p\bigg)^{1/p}.
\end{align*}
A direct computation shows that
\begin{align}\label{5.4.28}
|R_1|\leq &C\int_{\max\{t_1,0\}}^te^{-\bar{\nu}_0(t-s)}\|wf(s)\|_{L^\infty}e^{-\lambda_0 s^\alpha}\|h_0\|_{L^\infty}\dd s\nonumber\\
&+CT_0^{5/2}\sup_{0\leq s\leq t}\|h(s)\|_{L^\infty}\times \big\{m^{3+\ka}+\delta+2^{-T_0}+\f{1}N+\f{1}{N^{\beta-4}}\big\}\cdot\{\bar{M}+\bar{M}^2\}\nonumber\\
&+C_{N,T_0,m}\sup_{0\leq s\leq t}\|h(s)\|_{L^\infty}\cdot\{\sup_{0\leq s\leq t}\|f(s)\|_{L^p}+\sup_{0\leq s\leq t}\|f(s)\|_{L^p}^2\}\nonumber\\
\leq &CM_0e^{-\lambda_0 t^\alpha}\int_{0}^t\|h(s)\|_{L^\infty}\dd s+CT_0^{5/2}\big\{m^{3+\ka}+\delta+2^{-T_0}+\f1N+\f{1}{N^{\beta-4}}\big\}\cdot\{\bar{M}+\bar{M}^3\}\nonumber\\
&+C_{N,T_0,m}\{\sup_{0\leq s\leq t}\|f(s)\|_{L^p}+\sup_{0\leq s\leq t}\|f(s)\|_{L^p}^3\}.
\end{align}
For $R_2$, using \eqref{2.41} and \eqref{2.42} again, we have
\begin{align}\label{5.4.29}
|R_2|\leq&  C\sup_{0\leq s\leq t}\|h(s)\|_{L^\infty}^2\cdot\int_{\max\{t_1,0\}}^te^{-\bar{\nu}_0(t-s)}\dd s\nonumber\\
&\times \bigg(\int_{|u|\leq N}\int_{\max\{t_1',0\}}^\tau\int_{\mathbb{R}^3}e^{-\bar{\nu}_0(s-s')}|h(s',x'-(s-s')u,u')|^p\big[\langle u\rangle\langle u'\rangle\big]^{-4-p(\beta-4)}\bigg)^{1/p}\nonumber\\
\leq &C\left(\f{1}{N}+\f{1}{N^{\beta-4}}\right)\sup_{0\leq s\leq t}\|h(s)\|_{L^\infty}^3+C_N\sup_{0\leq s\leq t}\|h(s)\|_{L^\infty}^2\cdot\int_{\max\{t_1,0\}}^te^{-\bar{\nu}_0(t-s)}\dd s\nonumber\\
&\times \bigg(\int_{0}^{s-\f1N}e^{-\bar{\nu}_0(s-s')}\dd s'\int_{|u|\leq N}\int_{|u'|\leq N}\Fi_{\max\{t_1',0\}\leq s'\leq s}|f(s',x'-(s-s')u,u')|^p\dd u\dd u'\bigg)^{1/p}\nonumber\\
\leq & C\left(\f{1}{N}+\f{1}{N^{\beta-4}}\right)\sup_{0\leq s\leq t}\|h(s)\|_{L^\infty}^3+C_{N}\sup_{0\leq s\leq t}\|f(s)\|_{L^p}^3.
\end{align}
Here we have used the change of variable $u\rightarrow y':=x'-(s-s')u$. Then from combining \eqref{5.4.0}, \eqref{5.4.27}, \eqref{5.4.28} and \eqref{5.4.29}, \eqref{5.4.25} follows. Therefore, the proof of Lemma \ref{lm5.4} is complete.
\end{proof}

\subsection{Proof of Theorem \ref{thm1.3}.}
Let 
$$
\CE(t):=1+\int_0^t\left(\|h(s)\|_{L^\infty}{+|h(s)|_{L^\infty(\g)}}\right)\dd s.
$$
Then it holds from \eqref{5.4.25} that
\begin{align}\label{5.3.1}
\|h(t)\|_{L^\infty}+|h(t)|_{L^\infty(\g)}=\CE'(t)\leq CM_0e^{-\lambda_0 t^\alpha}\CE(t)+\CD,
\end{align}
where
\begin{multline*}
\CD:=CT_0^{5/2}\{m^{3+\ka}+\delta+\f{1}{N}+\f{1}{N^{\beta-4}}\}\bar{M}^3+2^{-T_0}\bar{M}^3\\
+C_{N,T_0,m}\left\{e^{C_3\bar{M}T_0}\|f_0\|_{L^p}+
\left(e^{C_3\bar{M}T_0}\|f_0\|_{L^p}\right)^3\right\}.
\end{multline*}
From \eqref{5.3.1}, we have
\begin{align}\label{5.3.2}
\CE(t)\leq \CE(0)e^{CM_0\int_{0}^te^{-\lambda_0 s^\alpha}\dd s}+\CD\cdot\int_0^te^{CM_0\int_s^te^{-\lambda_0\tau^\alpha}\dd\tau}\dd s\leq (1+\CD t)e^{CM_0}.
\end{align}
Substituting \eqref{5.3.2} into \eqref{5.3.1}, we have
\begin{align}\label{5.3.3}
\|h(t)\|_{L^\infty}+|h(t)|_{L^\infty(\g)}\leq CM_0e^{CM_0}(1+\CD t)e^{-\lambda_0 t^\alpha}+\CD\leq e^{CM_0}(1+\CD)e^{-\f{\lambda_0 t^\alpha}{2}}+\CD.
\end{align}
Take $\bar{M}=2e^{CM_0}$, $\tilde{\vep}=\min\{\vep_0,(2C_0)^{-1}\}$ where $C_0$ and $\vep_0$ are the same as ones in Theorem \ref{thm1.2}, and 
$$
T_0:=\max\left\{3\left(\log_2\bar{M}+1\right)+|\log_2\tilde{\vep}|, \left(\f{2\left(\log4\bar{M}+|\log\tilde{\vep}|\right)}{\lambda_0}\right)^{1/\alpha}\right\},
$$
such that
$
2^{-T_0}\bar{M}^3\leq \f{\tilde{\vep}}{8}$ and $\bar{M}e^{-\f{\lambda_0 T_0^\alpha}{2}}\leq \f{\tilde{\vep}}{4}$.
Then it holds that
$$CT_0^{5/2}\delta \bar{M}^3\leq C_4\left[|\log\tilde{\vep}|^{\f{5}{2\alpha}}+1\right]e^{C_4M_0}\delta,
$$
for some universal constant $C_4>1$. Let 
$$
0<\delta\leq \delta_0<\left(\f{\tilde{\vep}}{16C_4\left[|\log\tilde{\vep}|^{\f{5}{2\alpha}}+1\right]}\right)^{2},
$$
and $0<M_0\leq \f{|\log\delta|}{2C_4}.$ Then it is straightforward to see that
$CT_0^{5/2}\delta \bar{M}^3\leq \f{\tilde{\vep}}{16}$. 
Now we take $0<m<1$ suitably small and $N$ suitably large, and finally take $\|f_0\|_{L^p}\leq \vep_{1}$, with $\vep_1>0$ sufficiently small, such that
$$CT_0^{5/2}\{m^{3+\ka}+\f{1}{N}+\f{1}{N^{\beta-4}}\}\bar{M}^3+C_{N,T_0,m}\left\{e^{C_3\bar{M}T_0}\|f_0\|_{L^p}+
\left(e^{C_3\bar{M}T_0}\|f_0\|_{L^p}\right)^3\right\}\leq \f{\tilde{\vep}}{16}.
$$
Therefore, we have $\CD\leq \f{\tilde{\vep}}{4}.$ From \eqref{5.3.3}, it holds, for any $t\in[0,T_0]$, that
\begin{align}\label{5.3.4}\|h(t)\|_{L^\infty}+|h(t)|_{L^\infty(\g)}\leq e^{CM_0}(1+\CD)+\CD\leq (1+\f{\tilde{\vep}}{4})e^{CM_0}+\f{\tilde{\vep}}{4}\leq \f{3\bar{M}}{4}.
\end{align}
Notice that at $t=T_0$, we have from \eqref{5.3.3} that
$$\|h(T_0)\|_{L^\infty}\leq e^{CM_0}(1+\CD)e^{-\f{\lambda_0 T_0^{\alpha}}{2}}+\CD\leq \f{\tilde{\vep}}{2}.
$$
Then from \eqref{D1.3}, we have, for $t>T_0$,
\begin{align}\label{5.3.5}
\|h(t)\|_{L^\infty}+|h(t)|_{L^\infty(\g)}\leq C_{{0}}\|h(T_0)\|_{{L^\infty}}\leq C_0\tilde{\vep}\leq\f{3\bar{M}}{4}.
\end{align}
A combination of \eqref{5.3.4} and \eqref{5.3.5} justifies that the a priori assumption \eqref{5.0.1} can be closed by our choice. Notice that the local existence has been established in Proposition \ref{prop6.1}. Then the global existence of the solution follows from a standard continuity argument. For large time behavior, it holds, for $t\in[0,T_0]$, that
\begin{align}\label{5.3.6}
\|h(t)\|_{L^\infty}+|h(t)|_{L^\infty(\g)}\leq \bar{M}\leq e^{2CM_0}e^{\lambda_0 T_0^{\alpha} }e^{-\lambda_0 t^\alpha} \leq C_5e^{C_5M_0}e^{-\lambda_0 t^\alpha},
\end{align}
for some constant $C_5>1$. For $t>T_0$, it holds from \eqref{D1.3} that
\begin{align}\label{5.3.7}
\|h(t)\|_{L^\infty}+|h(t)|_{L^\infty(\g)}\leq C_0e^{-\lambda_0(t-T_0)^\alpha}\|h(T_0)\|_{L^\infty}\leq C_0C_5e^{C_5M_0}e^{-\lambda_0 t^\alpha}.
\end{align}
By taking $C_2=C_0C_5$, \eqref{D1.3a} follows from \eqref{5.3.6} and \eqref{5.3.7}. Therefore, the proof of Theorem \ref{thm1.3} is complete.\qed

\section{Appendix}
\subsection{An iteration lemma}

\begin{lemma}\label{lemA.1}
Consider  a sequence $\{a_i\}_{i=0}^\infty $  with each $a_i\geq0$. 
For any fixed $k\in\mathbb{N}_+$, we denote $$A_i^k=\max\{a_i, a_{i+1},\cdots, a_{i+k}\}.$$

\noindent{(1)} Assume $D\geq0$.  If $a_{i+1+k}\leq \f18 A_i^{k}+D$ for $i=0,1,\cdots$, then it holds that
\begin{equation}\label{A.1}
A_i^k\leq \left(\f18\right)^{\left[\frac{i}{k+1}\right]}\cdot\max\{A_0^k, \ A_1^k, \cdots, \ A_k^k \}+\f{8+k}{7} D,\quad\mbox{for}\quad i\geq k+1.
\end{equation}

\noindent{(2)} 
Let $0\leq \eta<1$ with $\eta^{k+1}\geq\frac14$.  If $a_{i+1+k}\leq \f18 A_i^{k}+C_k \cdot \eta^{i+k+1}$ for $i=0,1,\cdots$, then it holds that
\begin{align}\label{A.1-1}
A_i^k\leq \left(\f18\right)^{\left[\frac{i}{k+1}\right]}\cdot\max\{A_0^k, \ A_1^k, \cdots, \ A_k^k \}+2C_k\f{8+k}{7} \eta^{i+k},\quad\mbox{for}\quad i\geq k+1.
\end{align}
\end{lemma}

\begin{proof}
We first show \eqref{A.1}. By iteration in $i\geq 0$,  we obtain that
\begin{align} \label{A.2}
	a_{i+1+2k}&\leq \f18 A_{i+k}^k+D=\f18 \max\{a_{i+2k}, a_{i+2k-1},\cdots,a_{i+k}\}+D\nonumber\\
	&\leq \f18\max\{a_{i+2k}, \ A_{i+k-1}^k\}+D\leq \f18\max\left\{\f18 A_{i+k-1}^k+D, \ A_{i+k-1}^k\right\}+D\nonumber\\
	&\leq \f18 A_{i+k-1}^k+(1+\f18)D \cdots
	\leq  \f18 A_{i}^k+(1+\f{k}8)D.
	\end{align}
Similarly, for all $j=0,1,\cdots,k-1$, we also have
\begin{equation}\label{A.2-1}
a_{i+1+j+k}\leq \f18 A_{i}^k+\f{8+k}{8}D.
\end{equation}
Therefore, for $1\leq \frac{i}{k+1}\in\mathbb{N}_+$, it follows from \eqref{A.2} and \eqref{A.2-1} that
	\begin{align}\label{A.3}
	A_{i+1+k}^k&=\max\big\{a_{i+1+2k}, \ a_{i+2k},\ \cdots,\ a_{i+1+k}\big\}\nonumber\\
	&\leq \f18 A_{i}^k+\f{8+k}{8}D\leq \left(\f18\right)^2A_{i-2(k+1)}^k+\f18 \f{8+k}{8}D+\f{8+k}{8}D\nonumber\\
	&=\left(\f18\right)^2A_{i-2(k+1)}^k+\f{8+k}{8}\left(1+\f18\right)D=\cdots\nonumber\\
	&\leq \left(\f18\right)^{\frac{i}{k+1}}A_0^k+\f{8+k}{8}\left(1+\f18+\big(\f18\big)^2+\cdots\right)D\nonumber\\
	&\leq \left(\f18\right)^{\frac{i}{k+1}}A_0^k+\f{8+k}{7}D.
	\end{align}
If $\frac{i}{k+1}\notin \mathbb{N}_+$ and $i=(k+1) \left[\frac{i}{k+1}\right]+j$ for some  $1\leq j\leq k$, then by similar arguments we have
\begin{align}\label{A.4}
A_{i+1+k}^k&\leq \left(\f18\right)^{\left[\frac{i}{k+1}\right]+1} A_j^k+\f{8+k}{7}D.
\end{align}
Hence, from \eqref{A.3} and \eqref{A.4},  we complete the proof of \eqref{A.1}.



It remains to show \eqref{A.1-1}.
Noting $\eta<1$ and by similar arguments as in \eqref{A.2} and \eqref{A.2-1}, we can get
\begin{align}\label{A.7}
a_{i+j+k+1}\leq \frac18 A_{i}^k +C_k \left(1+\frac{k}{8}\right) \eta^{i+k+1},\  \mbox{for} \  0\leq j\leq k.
\end{align}
Hence, for $1\leq \frac{i}{k+1}\in\mathbb{N}_+$, noting $\frac18\cdot \eta^{-k-1}\leq \frac12$ and using \eqref{A.7}, then we have
\begin{align}\label{A.8}
&A^k_{i+k+1}=\max\{a_{i+2k+1},  \cdots,  a_{i+k+1} \}\leq \frac18 A_{i}^k +C_k \left(1+\frac{k}{8}\right) \eta^{i+k+1}\nonumber\\
&\leq \cdots \leq \left(\f18\right)^{\frac{i}{k+1}} A_0^k+ C_k \left(1+\frac{k}{8}\right) \eta^{i+k+1}\cdot\left\{1+\frac18 \eta^{-k-1}+\left(\frac18 \eta^{-k-1}\right)^2+\cdots\right\}\nonumber\\
&\leq \left(\f18\right)^{\frac{i}{k+1}} A_0^k+2C_k \left(1+\frac{k}{8}\right) \eta^{i+k+1}.
\end{align}
If $\frac{i}{k+1}\notin \mathbb{N}_+$ and $i=(k+1) \left[\frac{i}{k+1}\right]+j$ for some  $1\leq j\leq k$, then by similar arguments as above, we have
\begin{align}\label{A.9}
A^k_{i+k+1}
&\leq \left(\f18\right)^{\left[\frac{i}{k+1}\right]+1} A_j^k+2C_k \left(1+\frac{k}{8}\right) \eta^{i+k+1}.
\end{align}
Thus we prove \eqref{A.1-1} from \eqref{A.8} and \eqref{A.9}.  Therefore the proof of lemma \ref{lemA.1} is complete.
\end{proof}

\subsection{Local-in-time existence}

\begin{proposition}\label{prop6.1}
Let $w(v)$ be the weight function defined in \eqref{WF}. Assume $$
|\theta-\theta_0|_{L^{\infty}(\partial \Omega)}=\delta\ll1,\quad F_0(x,v)=F_*(x,v)+\mu^{\frac{1}{2}}(v)f_0(x,v)\geq0
$$
and
$
\|wf_0\|_{L^\infty}:=M_{0}<\infty.
$
Then there exists a positive time
$$
\hat{t}:=\bigg[\hat{C}\bigg(1+M_{0}\bigg)\bigg]^{-1}
$$
such that the IBVP \eqref{1.1}, \eqref{diffuseB.C} and \eqref{ad.id}
has a unique nonnegative solution $$F(t,x,v)=F_*(x,v)+\mu^{\frac{1}{2}}(v)f(t,x,v)\geq0$$ in $[0,\hat{t}]$ satisfying
\begin{align*}
\sup_{0\leq t\leq \hat{t}}\bigg\{\|wf(t)\|_{L^\infty}+|wf(t)|_{L^{\infty}(\g)}\bigg\}\leq 2\tilde{C}(M_0+1).
\end{align*}
Here $\tilde{C}>0$ and $\hat{C}>1$ are generic constants independent of {$M_0$}. 
Moreover, if the domain $\Omega$ is strictly convex, $\theta(x)$ is continuous over $\partial \Omega$, the initial data $F_0(x,v)$ {is} 
continuous except on $\g_0$ and satisfies \eqref{D1.2a}
then the solution $F(t,x,v)$ is continuous in $[0,\hat{t}]\times\{\Omega\times\mathbb{R}^3\setminus\g_0\}$.
\end{proposition}

\begin{proof}
We consider the following iteration scheme:
\begin{equation}\label{6.3}
\left\{\begin{aligned}
&\partial_tF^{n+1}+v\cdot\nabla_x F^{n+1}+F^{n+1}\cdot R(F^{n})=Q^+(F^n,F^n),\\
&F^{n+1}(t,x,v)\Big|_{t=0}=F_0(x,v)\geq0,\\
&F^{n+1}(t,x,v)|_{\g_-}=\mu_{\theta}(x,v)\int_{n(x)\cdot u>0} F^{n+1}(t,x,u)\{u \cdot n(x)\}\,\dd u,\\
&F^0(t,x,v)=\mu(v),
\end{aligned}
\right.
\end{equation}
where
$$
{R}(F^n)(t,x,v)=\int_{\mathbb{R}^{3}\times \mathbb{S}^{2}}{B}(v-u,\omega)F^n(t,x,u)\dd u\dd\omega.
$$
Let
$$
f^{n+1}(t,x,v)=\f{F^{n+1}(t,x,v)-F_{*}(x,v)}{\mu^{\frac{1}{2}}(v)},\quad h^{n+1}(t,x,v)=wf^{n+1}(t,x,v).
$$
Then the equation of $h^{n+1}$ reads as
\begin{equation}\label{6.4}
\left\{\begin{aligned}
&\partial_th^{n+1}+v\cdot\nabla_x h^{n+1}+h^{n+1}\cdot {R}(F^{n})=wK_{*}f^n+w\Gamma^+(f^n,f^n),\\
&h^{n+1}(t,x,v)\Big|_{t=0}=h_0(x,v),\\
&h^{n+1}\Big|_{\g_-}=\f{1}{\tilde{w}(v)}\int_{\{n(x)\cdot u >0\}} h^{n+1} \tilde{w}(u)\dd\s(x)+w(v)\frac{\mu_{\theta}-\mu}{\sqrt{\mu}}\int_{n(x)\cdot u>0}h^{n+1}\tilde{w}(u)\dd\sigma(x),\\
&h^{0}(t,x,v)=-wf_*(x,v),
\end{aligned}
\right.
\end{equation}
where we have denoted
$$K_{*}f^n:=-\sqrt{\mu}^{-1}\left\{R(\sqrt{\mu}f^n)F_*+\Gamma^{+}(\sqrt{\mu}f^{n},F_*)+\Gamma^+(F_*,\sqrt{\mu}f^n)\right\}.$$Now we shall use the induction  on $n=0,1,\cdots$ to show that there exists a positive time $\hat{t}_1>0$, independent of $n$, such that \eqref{6.3} or equivalently \eqref{6.4} admits a unique mild solution on the time interval $[0,\hat{t}_1]$, and the following uniform bound and positivity hold true:
\begin{align}\label{6.5}
\|h^n(t)\|_{L^\infty}+|h^n(t)|_{L^{\infty}(\g)}\leq 2\tilde{C}[\|h_0\|_{L^\infty}+1],
\end{align}
and \begin{align}\label{6.5-1}
F^n(t,x,v)\geq 0,
\end{align}
for $0\leq t\leq \hat{t}=\left(\hat{C}\{1+\|h_0\|_{L^\infty}\}\right)^{-1}$ and suitably chosen constants $\tilde{C}>0$ and $\hat{C}$ independent of $t$. Thanks to the fact
$$
\|h^0(t)\|_{L^\infty}+|h^0(t)|_{L^\infty(\g)}\leq \|wf_*\|_{L^\infty}+|wf_*|_{L^\infty(\g)}\leq C\delta,
$$
we see that \eqref{6.5} is obviously true for $n=0$. To proceed, we assume that \eqref{6.5} holds true up to $n\geq 0$. Since $F^n\geq 0$, it holds that ${R}(F^n)\geq 0$.  Then by using a similar argument as in \cite[Lemma 3.4]{DW}, one can construct the solution operator $G^n(t)$ to the following linear problem
\begin{equation}\nonumber
\left\{
\begin{aligned}
&\partial_th+v\cdot \nabla_x h+{R}(F^n)h=0, \quad t>0,\quad x\in \Omega, \quad v\in \mathbb{R}^3,\\
&h(t,x,v)|_{t=0}=h_0(x,v),\\
&h(t,x,v)\Big|_{\g_-}=\f{1}{\tilde{w}(v)}\int_{n(x)\cdot u>0} h(t,x,u) \tilde{w}(u)\dd\s(x)+w(v)\frac{\mu_{\theta}-\mu}{\sqrt{\mu}}\int_{n(x)\cdot u>0}h(t,x,u)\tilde{w}(u)\dd\sigma(x)
\end{aligned}
\right.
\end{equation}
over $(0,\rho)$ for some universal constant $\rho>0$ independent of $n$, provided that $|\theta-1|_{L^\infty(\pa\Omega)}$ is sufficiently small.
Moreover, $G^n(t)$ satisfies the following estimate
\begin{align}\label{6.6}
\|G^{n}(t)h_0\|_{L^{\infty}}+|G^n(t)h_0|_{L^\infty(\g)}\leq C_{\rho}\|h_0\|_{L^{\infty}}.
\end{align}
Here the constant $C_\rho>0$ is independent of $n$. Then applying Duhamel's formula to \eqref{6.4}, we have, for $0<t<\rho,$ that
\begin{align}\label{6.7}
h^{n+1}(t)= {G}^{n}(t,0)h_0+\int_0^t {G}^{n}(t,s)[w K_*f^n(s)+w\Gamma^+(f^n,f^n)(s)]\dd s.
\end{align}
Taking $L^\infty$-norm on the both sides of \eqref{6.7} and using \eqref{2.43} and \eqref{6.6}, we have
\begin{align}\label{6.8}
&\|h^{n+1}(t)\|_{L^\infty}+|h^{n+1}(t)|_{L^\infty{(\g)}}\notag\\
&\leq C_\rho\|h_0\|_{L^\infty}+C_\rho\int_0^t\|wK_{*}f^n(s)\|_{L^\infty}+\|w\Gamma^+(f^n,f^n)(s)\|_{L^\infty}\dd s\nonumber\\
&\leq {C}
\|h_0\|_{L^\infty}+C\int_0^t\|h^n(s)\|_{L^\infty}+\|h^n(s)\|^2_{L^\infty}\dd s\nonumber\\
&\leq C_{{3}}\|h_0\|_{L^{\infty}}+C_{{3}}t\cdot\{\sup_{0\leq s\leq t}\|h^n(s)\|_{L^{\infty}}+\sup_{0\leq s\leq t}\|h^n(s)\|_{L^\infty}^2\},
\end{align}
for some constants $C_3>1$. Now we take $\tilde{C}=C_3$ and $\hat{C}=8C_3^2$. Then by the induction hypothesis \eqref{6.5}, for any $ 0<t<\hat{t}$, it follows from \eqref{6.8} that
$$\begin{aligned}
\|h^{n+1}(t)\|_{L^\infty}+|h^{n+1}(t)|_{L^\infty(\g)}&\leq C_{3}\{\|h_0\|_{L^{\infty}}+1\}\cdot\{1+2C_3t[1+2C_3]\cdot[1+\|h_0\|_{L^\infty}]\}\\
&\leq 2C_3\{\|h_0\|_{L^\infty}+1\}.
\end{aligned}
$$
This then proves \eqref{6.5} for $n+1$. Next we show the non-negativity \eqref{6.5-1} for $n+1$. We denote that
$$I^{n}(t,s):=\exp\left\{-\int_s^t[R(F^n)](\tau,X_{cl}(\tau),V_{cl}(\tau))\dd\tau\right\}
$$
and$$
\dd\Sigma _{l}^n(\tau):=\{\prod _{j=l+1}^{k-1}\dd\sigma _{j}\}\cdot  I^n(t_l,\tau)[v_l\cdot n(v_l)]\dd v_l\cdot\prod_{j=1}^{l-1}\{I^n(t_{j},t_{j+1})\mu_{\theta}(x_{j+1},v_{j})[v_j\cdot n(x_j)]\dd v_{j}\}.$$
Then we have the following mild formulation for $F^{n+1}$:
\begin{align}\label{6.9}
F^{n+1}(t,x,v)=& \Fi_{\{t_1\leq 0\}} \Big\{I^n(t,0) F_0(x-vt,v)+\int_0^tI^n(t,s) Q^+(F^n,F^n)(s,x-v(t-s),v)\,\dd s\Big\}\nonumber\\
&+  \Fi_{\{t_1>0\}}\mu_{\theta}(x_{1},v) I^n(t,t_1)  \bigg\{\int_{\Pi_{j=1}^{k-1}\mathcal{V}_j} \sum_{l=1}^{k-1} \Fi_{\{t_{l+1}\leq0< t_l\}} F_0(x-vt,v)\,\dd\Sigma_{l}^n(0)\nonumber\\
&+\int_{\Pi_{j=1}^{k-1}\mathcal{V}_j} \sum_{l=1}^{k-1} \int_0^{t_l}\Fi_{\{t_{l+1}\leq0< t_l\}} Q^+(F^n,F^n)(\tau,X_{cl}(\tau),V_{cl}(\tau))\, \dd\Sigma_{l}^{m}(\tau) \dd\tau\nonumber\\
&+\int_{\Pi_{j=1}^{k-1}\mathcal{V}_j} \sum_{l=1}^{k-1} \int_{t_{l+1}}^{t_l}\Fi_{\{t_{l+1}\leq0< t_l\}} Q^+(F^n,F^n)(\tau,X_{cl}(\tau),V_{cl}(\tau)) \,\dd\Sigma_{l}^{m}(\tau) \dd\tau\bigg\}\nonumber\\
&+ \Fi_{\{t_1>0\}}\mu_{\theta}(x_{1},v) I^n(t,t_1)\int_{\Pi_{j=1}^{k-1}\mathcal{V}_j} \Fi_{\{t_k>0\}}F^{n+1}(t_k,x_k,v_{k-1})\,\dd\Sigma_{k-1}^n(t_k),
\end{align}
for $t>0$, $x\in\bar{\Omega}\times \mathbb{R}^3\setminus \gamma_0\cup\gamma_-$ and integer $k\geq 1$. From \eqref{6.5}, it holds that
\begin{align}\label{6.10}
|F^{n+1}(t,x,v)|=\left|F_*(x,v)+\mu^{\frac{1}{2}}(v)\f{h^{n+1}(t,x,v)}{w(v)}\right|\leq C(1+\|h_0\|_{L^{\infty}})\mu^{\frac{1}{2}}(v),
\end{align}
for some constant $C>0$. Furthermore, a direct computation shows that
$$
0<\mu_{\theta}(x_{j+1},v_{j})\leq \f{1}{(1-\delta)^2}\exp\left\{\f{\delta|v_j|^2}{2(1-\delta)}\right\}\mu(v_j).
$$
Then similar as \eqref{4.1.5-1}, we have, for sufficiently large $T_0>0$ and for $k=\hat{C}_5T_0^{5/4}$, that
$$\int_{\prod_{j=1}^{k-2}\mathcal{V}_j}\Fi_{\{t_{k-1}>0\}}\prod_{j=1}^{k-2}\mu_{\theta}(x_{j+1},v_j)\{n(x_j)\cdot v_j\}\dd v_j\leq \left(\f12\right)^{\hat{C}_6T_0^{5/4}},
$$
for some generic constant $\hat{C}_5>0$ and $\hat{C}_6>0$. Then by \eqref{6.9} and \eqref{6.10}, we have
\begin{equation}
\begin{aligned}
F^{n+1}&\geq -C\mu_\theta(x_{1},v)\{1+\|h_0\|_{L^\infty}\} \int_{\prod_{j=1}^{k-2}\mathcal{V}_j}\Fi_{\{t_{k-1}>0\}}\prod_{j=1}^{k-2}\mu_{\theta}(x_{j+1},v_j)\{n(x_j)\cdot v_j\}\dd v_j\nonumber\\
&\geq -C\mu_\theta(x_{1},v)\{1+\|h_0\|_{L^\infty}\}\cdot\left(\frac12\right)^{\hat{C}_6T_0^{5/4}}.
\end{aligned}
\end{equation}
Since $T_0>0$ can be taken arbitrarily large, we have $F^{n+1}\geq 0.$ This then proves \eqref{6.5} and \eqref{6.5-1}. Finally, with the uniform estimates \eqref{6.5} in hand, we can use a similar argument as one in \cite[Theorem 4.1]{DW} to show that $h^n,$ $n=0,1,2\cdots,$ is a Cauchy sequence in $L^\infty$. We omit here for brevity. The solution is obtained by taking the limit $n\rightarrow\infty$. If $\Omega $ is convex and the compatibility condition \eqref{D1.2a} holds, the continuity is a direct consequence of the $L^\infty$-convergence. The uniqueness is standard. The proof of Proposition \ref{prop6.1} is complete.
\end{proof}

\noindent{\bf Acknowledgments.}  Renjun Duan is partially supported by the General Research Fund (Project No.~14302817). Feimin Huang is partially supported by National Center for Mathematics and Interdisciplinary Sciences, AMSS, CAS and NSFC Grant No.11688101. Yong Wang is partly supported by NSFC Grant No. 11401565 and 11688101.

\end{document}